\documentclass[12pt,a4paper]{article}
\usepackage[english]{babel}
\usepackage[whole]{bxcjkjatype}
\usepackage{parskip}

\usepackage[T1]{fontenc}
\usepackage[utf8]{inputenc}
\usepackage{lmodern}
\usepackage{parskip}

\usepackage{geometry}
\usepackage{verbatim}
\usepackage{float}
\usepackage{bm}
\usepackage{amsmath}
\usepackage{amssymb}
\usepackage{amsfonts}
\usepackage{amsthm}
\usepackage{mathtools}
\usepackage{graphicx}
\usepackage{dsfont}
\usepackage{array}
\usepackage{color}
\usepackage{enumitem}
\usepackage{authblk}
\usepackage{booktabs}

\numberwithin{equation}{section}
\mathtoolsset{showonlyrefs=true}

\usepackage{multirow}
\usepackage{makecell}
\usepackage{arydshln}
\usepackage[font=small]{caption}
\usepackage{subcaption}

\usepackage{algorithm}
\usepackage{algpseudocode}

\RequirePackage[authoryear]{natbib}

\RequirePackage[colorlinks,citecolor=blue,urlcolor=blue]{hyperref}
\usepackage[dvipsnames]{xcolor}
\usepackage{cleveref}

\usepackage{pgfplots}
\pgfplotsset{compat=1.18}
\usepackage{appendix}

\allowdisplaybreaks

\theoremstyle{plain}

\newtheorem{theorem}{Theorem}[section]
\newtheorem{lemma}[theorem]{Lemma}
\newtheorem{proposition}[theorem]{Proposition}
\newtheorem{corollary}[theorem]{Corollary}
\theoremstyle{definition}
\newtheorem{definition}[theorem]{Definition}
\newtheorem{remark}[theorem]{Remark}
\newtheorem{example}[theorem]{Example}

\newcommand{\CC}{\mathbb{C}}
\newcommand{\EE}{\mathbb{E}}
\newcommand{\KK}{\mathbb{K}}

\newcommand{\RR}{\mathbb{R}}

\newcommand{\ZZ}{\mathbb{Z}}
\newcommand{\Smat}{\mathbb{S}}

\newcommand{\cA}{\mathcal{A}}

\newcommand{\cC}{\mathcal{C}}

\newcommand{\cH}{\mathcal{H}}
\newcommand{\cI}{\mathcal{I}}
\newcommand{\cJ}{\mathcal{J}}

\newcommand{\cL}{\mathcal{L}}

\newcommand{\cO}{\mathcal{O}}
\newcommand{\cP}{\mathcal{P}}

\newcommand{\cR}{\mathcal{R}}

\newcommand{\cT}{\mathcal{T}}
\newcommand{\cU}{\mathcal{U}}
\newcommand{\cV}{\mathcal{V}}
\newcommand{\cW}{\mathcal{W}}
\newcommand{\cX}{\mathcal{X}}

\newcommand{\tr}{\mathop{\mathrm{tr}}}
\newcommand{\rank}{\mathop{\mathrm{rank}}}
\newcommand{\spec}{\mathop{\mathrm{spec}}}
\newcommand{\vech}{\mathop{\mathrm{vech}}}

\newcommand{\sgn}{\mathop{\mathrm{sgn}}}
\newcommand{\diag}{\mathrm{diag}}
\newcommand{\dd}{\mathop{}\!\mathrm{d}}
\DeclareMathOperator{\Sym}{\mathrm{Sym}}

\DeclareMathOperator{\var}{var}
\DeclareMathOperator{\cov}{cov}

\newcommand{\Sighat}{\hat{\Sigma}}
\newcommand{\indep}{\mathop{\perp\!\!\!\perp}}

\newcommand{\1}{\mathbf{1}}

\newcommand{\pto}{\xrightarrow{\mathrm{p}}}
\newcommand{\dto}{\rightsquigarrow}

\title{Self-Normalizing Denominators\\in Rational Covariance Estimators}
\author[1,2,$\ast$]{Shu Tamano}
\affil[1]{\small
Department of Multidisciplinary Sciences, Graduate School of Arts and Sciences, The University of Tokyo, 3-8-1 Komaba, Meguro-ku, Tokyo 153-8902, Japan
}
\affil[2]{
\small
Department of Epidemiology, National Institute of Infectious Diseases, Japan Institute for Health Security, 1-23-1 Toyama, Shinjuku-ku, Tokyo 162-0052, Japan
}
\affil[$\ast$]{\small
Email:
\href{mailto:tamano-shu212@g.ecc.u-tokyo.ac.jp}
{\texttt{tamano-shu212@g.ecc.u-tokyo.ac.jp}}
}
\date{}

\begin{document}
\maketitle

\begin{abstract}
    Many estimators are ratios of coprime polynomials in a sample covariance matrix, and their accuracy depends on the relative fluctuation of the sample denominator.
    Under Gaussian sampling in fixed dimension, we call a nonconstant polynomial denominator self-normalizing if the first-order variance of its relative error does not depend on the population covariance.
    We prove that these denominators are exactly the flag powers, nonzero constant multiples of products of positive integer powers of nested generalized variances.
    Equivalently, the denominator's sample-to-population ratio has a covariance-independent finite-sample law, which we determine explicitly.
    Sufficiency is classical;
    the new converse shows that a first-order variance condition forces an exact sampling law.
    We show that relative stability, meaning bounded first-order relative variance, characterizes uniform tightness of scaled relative errors over positive-definite covariances.
    It permits replacing the sample denominator by its population value in the limit theory of the ratio.
    Self-normalization is its rigid core.
    We locate these classes in applications, where regression on predecessors in a fixed order yields only constant or self-normalizing denominators, instrumental-variable formulas yield relatively unstable ones, and nonparametric identifiability does not guarantee relative stability.
\end{abstract}

\section{Introduction}
\label{sec:intro}

\subsection{Motivation}

Many estimators in multivariate analysis and linear causal inference are rational functions of a sample covariance matrix $\Sighat$.
Regression coefficients, entries of the inverse covariance matrix, generalized least-squares weights with a fixed design, and linear mediation and instrumental-variable (IV) estimators provide familiar examples.
Such an estimator has the form $\hat{\tau}=N(\Sighat)/D(\Sighat)$ and targets $\tau=N(\Sigma)/D(\Sigma)$, where $\Sigma$ is the population covariance matrix and $N$ and $D$ are coprime polynomials in its entries.
Throughout, $N/D$ is written in lowest terms, and we call $D$ the reduced denominator (Definition~\ref{def:strategy}).
We specify $N/D$ as an ambient rational function on all positive-definite matrices where $D\ne0$, since $\Sighat$ need not satisfy structural covariance restrictions.
We seek denominator guarantees over this entire domain.
The denominator is a distinct source of sampling uncertainty:
what matters is not the magnitude of $D(\Sigma)$ alone, but the fluctuation of $D(\Sighat)$ relative to $D(\Sigma)$.
We determine the denominators for which the law of the sample-to-population ratio $D(\Sighat)/D(\Sigma)$ does not depend on $\Sigma$, and we characterize the denominators for which this ratio deviates from one at the parametric rate uniformly in $\Sigma$.

A small denominator can have noise comparable with its own size.
This is the weak-instrument regime of IV regression, in which the instrument--treatment covariance in the denominator is of the order of its standard error, and it produces nonstandard ratio limits and confidence sets that must sometimes be unbounded \citep{staigerstock1997,gleserhwang1987,dufour1997}.
Alternatively, the noise can contract with the magnitude, so that a small denominator need not disrupt the usual limit theory.
In practice the two cases are told apart by comparing the denominator with its standard error, as with the first-stage $F$ statistic of IV regression.
Which case can arise is not a matter of magnitude:
it is a property of the reduced polynomial $D$ under the sampling model.

For orientation, let $X_1,\ldots,X_n$ be independent $N_d(0,\Sigma)$ vectors, where $N_d(0,\Sigma)$ denotes the $d$-dimensional centered Gaussian law, with $\Sigma=(\sigma_{jk})\succ0$, that is, positive definite, and $d\ge2$.
Let $\Sighat=n^{-1}\sum_{i=1}^nX_iX_i^\top$ be the known-mean sample covariance of the reference experiment in Section~\ref{sec:notation}, where $^\top$ denotes transpose.
Write $s_D^2(\Sigma)$ for the limiting variance of $\sqrt{n}\{D(\Sighat)-D(\Sigma)\}$ given by the delta method, and set $\kappa_D=s_D^2/D^2$ wherever $D\ne0$.
We call $\kappa_D$ the \emph{first-order relative variance} of $D$ (Section~\ref{sec:relative-variance}).
It is the limiting variance of $\sqrt{n}\{D(\Sighat)/D(\Sigma)-1\}$, where $D(\Sighat)/D(\Sigma)-1$ is the \emph{relative error} of the sample denominator.
For $\rho=\sigma_{12}/\sqrt{\sigma_{11}\sigma_{22}}$,
\begin{equation}
    \label{eq:intro-three-denominators}
    \kappa_{\sigma_{11}}=2
    ,\quad
    \kappa_{\sigma_{11}\sigma_{22}}=4+4\rho^2
    ,\quad
    \kappa_{\sigma_{12}}=\frac{1+\rho^2}{\rho^2}
    ,
\end{equation}
where only the last expression requires $\rho\ne0$.
Thus the first-order relative standard deviation $\sqrt{\kappa_D/n}$ is constant in covariance, bounded but nonconstant, or arbitrarily large.
We call a nonconstant polynomial $D$ \emph{self-normalizing} if $\kappa_D$ is a positive constant, and \emph{relatively stable} if $\kappa_D$ is bounded over the positive-definite cone (Definitions~\ref{def:self-normalization} and~\ref{def:bcd-stable}).
The first denominator is self-normalizing, and its sample-to-population ratio is exactly $\chi_n^2/n$ in distribution, where $\chi_n^2$ denotes a chi-squared variable with $n$ degrees of freedom.
The second is relatively stable without being self-normalizing, so a denominator can be harmless without having a covariance-free sample-to-population law.
The third is neither:
when $\rho$ is of order $n^{-1/2}$ its noise is comparable with the denominator itself, which is the weak-instrument regime.

\subsection{Main contributions}

We say that $D$ has a \emph{relative pivot} if the law of $D(\Sighat)/D(\Sigma)$ is the same for every $\Sigma$ with $D(\Sigma)\ne0$ at each sufficiently large sample size (Definition~\ref{def:relative-pivot}).
Under Gaussian sampling in fixed dimension, we ask which nonconstant polynomial denominators have a relative pivot, and how this class sits inside the relatively stable denominators.
The main findings are as follows.
\begin{itemize}
    \item \textit{Complete classification.}
    For a nonconstant polynomial $D$ the following are equivalent:
    $D$ is self-normalizing;
    $D$ has a relative pivot;
    $D$ is a \emph{flag power}, a nonzero constant times a product of positive integer powers of determinants of $\Sigma$ compressed to nested subspaces (Definition~\ref{def:flag-power} and Theorem~\ref{thm:sn-pivot-classification}).
    Degree, homogeneity and factorization are not assumed;
    homogeneity and the nesting are consequences.
    A first-order variance condition therefore forces an exact finite-sample law.
    That flag powers are pivotal is classical, by Bartlett's decomposition, which factors the sample covariance through a triangular matrix whose squared diagonal entries are independent chi-squared variables;
    our contribution is the converse, which derives the flag without a prescribed group action.
    
    \item \textit{Relative stability.}
    Uniformly over all $\Sigma\succ0$, tightness of the relative error multiplied by $\sqrt{n}$ is equivalent to boundedness of $\kappa_D$ (Theorem~\ref{thm:uniform-denominator}).
    Under this condition the sample denominator may be replaced by its population value whenever the normalized moment $\{N(\Sighat)-\tau D(\Sighat)\}/D(\Sigma)$ is tight after multiplication by the same deterministic scale as the ratio error (Corollary~\ref{cor:denominator-removal}).
    The condition is deterministic, and relative stability extends to observations $X_i=AZ_i$, uniformly over invertible matrices $A$, when $Z_1,\ldots,Z_n$ are independent with a common fixed law having mean zero, identity covariance and finite fourth moments (Proposition~\ref{prop:fixed-law-stability}).

    \item \textit{Sources and boundaries.}
    Every polynomial in the coefficients obtained by regressing each variable on all its predecessors in a fixed order has a reduced denominator that is constant or a flag power (Proposition~\ref{prop:bcd-regression-source}).
    Among products of covariance minors, relative stability holds exactly when every factor is principal, nested or not.
    A first-order weak denominator sequence in the sense of Definition~\ref{def:bcd-weak-sequence} exists exactly when a nonprincipal factor has multiplicity one (Proposition~\ref{prop:bcd-minor-products}).
    Positive polynomials admit no first-order weak denominator sequence (Proposition~\ref{prop:positive-boundary-degeneracy}).
    Nevertheless, the Gaussian reduction of a nonparametric identifying formula can have a positive but relatively unstable denominator (Example~\ref{ex:bcd-napkin}).

    \item \textit{Inference for a ratio.}
    For the linear mediation functional, the product of the treatment--mediator and mediator--outcome regression coefficients, which is also the front-door functional, the Wald statistic with an empirical-fourth-moment standard error is asymptotically standard normal along every sequence of mediator residual variances, however fast they vanish, for a fixed nonzero treatment--mediator coefficient and any error law with finite fourth moments under a mild nondegeneracy condition (Proposition~\ref{prop:mediation-boundary-inference}).
    Under Gaussian error terms the denominator keeps its covariance-free law while the standard error of the effect changes with the collapsing variance.
\end{itemize}

Section~\ref{sec:classification-proof-overview} outlines the necessity argument, with complete proofs in the Supplementary Material.

\subsection{Related work}

Generalized variance and its sampling theory originate in classical multivariate analysis \citep{wilks1932,bartlett1934,muirhead1982}.
First and second moments of Wishart minors have been derived \citep{drton2008,drtongoia2012}, and joint moments of nested principal minors, together with the corresponding problem for principal minors on disjoint index sets, have been revisited \citep{genest2024}.
Group invariance provides a general framework for constructing pivots \citep{eaton1989}, including determinant and Schur-complement laws on symmetric cones \citep{massamneher1997}.
Flag powers are, after an orthogonal change of coordinates, the polynomial generalized power functions of the cone of positive-definite matrices \citep{farautkoranyi1994}.
Related invariant constructions include the basic relative invariants of homogeneous cones, the associated Wishart families and the triangular-group characterization of Riesz exponential families \citep{ishi2001,anderssonwojnar2004,letacmassam2007,nakashima2014,ishikolodziejek2019}, where the classified objects are distribution families.
The moment calculations start from specified minors, and the invariant constructions start from a prescribed group action or relative invariance.
Our converse derives the fixed flag from a scalar first-order variance identity, and thereby exhausts the nonconstant covariance polynomials with a relative pivot.
Covariance stabilization has a separate history \citep{holland1973,song1982}.
Fisher's correlation transformation illustrates the distinction between first-order stabilization and exact sampling laws \citep{fisher1921,hotelling1953}:
the smooth nonpolynomial functional $\exp(\operatorname{arctanh}\rho)$ of the correlation $\rho$ has constant Gaussian first-order relative variance, but its sample-to-population law depends on $\rho$ (Supplementary Example~\ref{ex:fisher-exponential}), whereas for homogeneous polynomials scalar variance stabilization is forced to be logarithmic (Corollary~\ref{cor:log-variance-stabilization}).
The geometric ingredients connect to positive-definite matrix geometry \citep{bhatia2007} and eikonal rigidity \citep{caffarellicrandall2010}.

Anderson--Rubin and Fieller inversion address noisy denominators \citep{andersonrubin1949,fieller1954}, while impossibility results and weak identification theory explain the associated confidence-set and limit phenomena \citep{gleserhwang1987,dufour1997,staigerstock1997,stockwright2000,andrewscheng2012,andrewsstocksun2019}.
Uniform delta methods study approximation errors along drifting parameters \citep{kasy2019};
our stability theorem identifies the relevant uniformity condition for relative errors of covariance polynomials on the whole positive-definite cone.
Products of coefficients are nonregular at singular null hypotheses \citep{drtonxiao2016,mileschambaz2026}.
Our denominator-removal result and the mediation analysis separate this source of nonregularity from denominator instability.

Graphical identification supplies rational formulas through instruments, front-door adjustment and half-trek criteria \citep{pearl1995,britopearl2002,foygeldraismadrton2012,vanderzander2015,henckel2024}, and proxy identification and proximal inference yield further covariance ratios \citep{kurokipearl2014,miao2018,tchetgen2024}.
Ancestor decomposition studies generic identification \citep{drtonweihs2016}, and regular semiparametric inference for half-trek estimators has recently been derived \citep{mareis2026}.
Identification, estimation and model restrictions have been studied in the Napkin graph, and for trapdoor variables that must be held fixed in an identifying formula although the target does not depend on them \citep{helske2021,guo2025}.
Our analysis starts after an ambient rational formula has been selected and reduced, and determines the sampling behavior of its denominator under covariance perturbations, including perturbations outside the structural model.

\subsection{Organization}

Section~\ref{sec:prelim} fixes the notation, sets out the sampling framework and gives the formal definitions of the notions introduced above.
Section~\ref{sec:main} gives the classification, its proof outline and calibration consequences.
Section~\ref{sec:stability} treats relative stability and denominator removal.
Section~\ref{sec:examples} develops the examples and mediation inference, and Section~\ref{sec:numerics} gives numerical illustrations.
Section~\ref{sec:discussion} discusses further directions.
The Supplementary Material contains complete proofs and computational details.

\section{Statistical formulation}
\label{sec:prelim}

\subsection{Notation and sampling}
\label{sec:notation}

Fix an integer $d\ge1$.
We use $\RR$ and $\CC$ for the real and complex fields, $\ZZ$ for the integers, and $\ZZ_{\ge0}$ and $\ZZ_{\ge1}$ for the nonnegative and positive integers.
Vectors are columns, $A^\top$ denotes transpose, $I_r$ is the $r\times r$ identity matrix, $\1_r$ is the all-ones vector in $\RR^r$, and $e_j$ is the $j$th coordinate vector in the dimension indicated by context.
We write $\tr$, $\det$, $\rank$ and $\spec$ for trace, determinant, rank and the eigenvalue multiset, $|I|$ for the cardinality of a finite set, and $\|x\|$ and $\|A\|_F=\{\tr(A^\top A)\}^{1/2}$ for the Euclidean and Frobenius norms.
Let $[d]=\{1,\ldots,d\}$, let $\Sym_d$ be the space of real symmetric $d\times d$ matrices, and let $\Smat^d_{++}$ be its positive-definite cone;
$A\succ0$ and $A\succeq0$ mean that $A$ is positive definite and positive semidefinite, respectively.
For index sets $I,J\subseteq[d]$, $A_{I,J}$ is the submatrix of $A$ with rows in $I$ and columns in $J$, both in increasing order, $A_I=A_{I,I}$ and $A_{1:r}=A_{\{1,\ldots,r\}}$;
an ordered index list retains the displayed order, and a single index denotes the corresponding singleton, so that $A_{i,J}$ is a row vector and $A_{I,j}$ a column vector.
Variable names such as $X,M,Y$ may also label coordinates, as in $\sigma_{XM}=\cov(X,M)$.
The empty determinant is one.
The symbols $\ker A$, $\operatorname{Range}A$ and $\operatorname{span}(S)$ denote kernel, column space and linear span, $V^\perp$ is the Euclidean orthogonal complement, and $\diag(x_1,\ldots,x_d)$ is a diagonal matrix.
For $A\succeq0$, $A^{1/2}$ is the positive-semidefinite square root;
for $A\succ0$, $A^{-1/2}$ and $\log A$ are defined by spectral calculus;
and $\exp(H)$ is the matrix exponential.

We order the pairs $(i,j)$ with $1\le i\le j\le d$ lexicographically.
The vector $\vech(A)=(a_{11},a_{12},\ldots,a_{1d},a_{22},\ldots,a_{dd})^\top$ and all coordinate derivatives use this order.
The ring $\RR[\Sym_d]$ consists of polynomials in these $d(d+1)/2$ entries with real coefficients, and $\CC[\Sym_d]$ is its complex counterpart.
For polynomials, $P\equiv Q$ denotes polynomial equality, and for functions on a common domain it denotes equality throughout that domain;
$\deg P$ is total degree, $P\mid Q$ means that $P$ divides $Q$ in the polynomial ring under consideration, and two polynomials are coprime when their common divisors are constants.
A nonzero polynomial is homogeneous of degree $K$ if $P(tA)=t^KP(A)$ for every real $t$ and every $A\in\Sym_d$.
For a differentiable scalar function $P$ on an open subset of $\Sym_d$, its coordinate gradient with respect to $\vech(\Sigma)$ is $\nabla P$ and its symmetric gradient is determined by
\begin{equation}
    \label{eq:gradient-convention}
    \dd P_\Sigma(H)
    =
    \left.
        \frac{\dd}{\dd t}P(\Sigma+tH)
    \right|_{t=0}
    =
    \tr\{G_P(\Sigma)H\}
    ,\quad
    H\in\Sym_d
    .
\end{equation}
Thus, writing $\Sigma=(\sigma_{ij})$, $(G_P)_{ii}=\partial P/\partial\sigma_{ii}$ and $(G_P)_{ij}=2^{-1}\partial P/\partial\sigma_{ij}$ for $i<j$.
For functions of vector arguments, $\nabla$ is the Euclidean gradient, and a dot over a path denotes differentiation with respect to its parameter.

Expectation, variance, covariance and probability under a law $Q$, or under independent sampling from $Q$, are $\EE_Q$, $\var_Q$, $\cov_Q$ and $\Pr_Q$, with the subscript suppressed when clear;
in Gaussian statements the subscript $\Sigma$ specifies the population covariance.
We write $X\sim Q$ when $X$ has law $Q$, $U\indep V$ for independence, $\stackrel{d}{=}$ for equality in distribution, $\pto$ and $\dto$ for convergence in probability and distribution, and $\1\{E\}$ for the indicator of $E$.
For deterministic $b_n>0$, $Y_n=O_p(b_n)$ means that $Y_n/b_n$ is tight and $Y_n=o_p(b_n)$ means that $Y_n/b_n\pto0$, in Euclidean or Frobenius norm for vectors or matrices;
$f\sim g$ means that $f/g$ tends to one.
Limits are as $n\to\infty$, with dimension and specified functions fixed, unless stated otherwise, and along a parameter sequence the sample quantities $\Sighat$ and $\hat{\tau}$ retain their notation.

The reference Gaussian experiment consists of independent $X_1,\ldots,X_n\sim N_d(0,\Sigma)$ with unknown $\Sigma\succ0$ and
\begin{equation}
    \label{eq:sample-covariance}
    \Sighat
    =
    n^{-1}\sum_{i=1}^nX_iX_i^\top
    ,
    \quad
    n\Sighat\sim\cW_d(n,\Sigma)
    .
\end{equation}
Here $N_d(\mu,\Omega)$ is the Gaussian law with mean $\mu\in\RR^d$ and covariance $\Omega\succeq0$, and $\cW_d(\nu,\Omega)$ is the law of the sum of $\nu$ independent outer products of vectors with law $N_d(0,\Omega)$;
$\chi^2_\nu=\cW_1(\nu,1)$ is the chi-squared law with $\nu$ degrees of freedom, and the same symbol denotes a variable with that law.
The covariance $\Sighat$ is positive definite almost surely for $n\ge d$, and for a fixed $B\in\RR^{d\times r}$ of full column rank the compression $B^\top\Sighat B$ is positive definite almost surely for $n\ge r$.
Exact results state their required block size.
With an unknown Gaussian mean, the exact sampling laws hold for the unbiased centered covariance with $n-1$ in place of $n$ (Supplementary Section~\ref{supp:unknown-mean}).
Non-Gaussian sampling laws are introduced in Section~\ref{sec:fixed-law}, where they are needed.

\subsection{Rational covariance estimators}
\label{sec:rational-estimators}

We first fix the class of estimators and the role of the denominator.

\begin{definition}[Rational covariance estimator]
    \label{def:strategy}
    Let $N,D\in\RR[\Sym_d]$ be coprime, with $D\not\equiv0$, and set $\cU_D=\{\Sigma\succ0:D(\Sigma)\ne0\}$.
    The associated rational covariance functional is $\tau(\Sigma)=N(\Sigma)/D(\Sigma)$ for $\Sigma\in\cU_D$, and its plug-in estimator is $\hat{\tau}=N(\Sighat)/D(\Sighat)$ on $\{D(\Sighat)\ne0\}$.
    We call $\tau$ the ambient rational function, since its domain is the open set $\cU_D$ rather than a structural submodel, and $D$ its reduced denominator.
\end{definition}

The pair $(N,D)$ is determined by $\tau$ up to a common nonzero scalar (Section~\ref{supp:sample-domains}), so the reduced denominator is well defined, and a constant $D$ corresponds to a polynomial functional.
For $n\ge d$, any nonzero covariance polynomial is nonzero at $\Sighat$ almost surely (Section~\ref{supp:sample-domains}), so $\hat{\tau}$ is defined almost surely.
A structural covariance model is a subset of $\Smat^d_{++}$.
Distinct ambient rational functions that agree with the same target on such a model specify distinct plug-in rules, since the sample covariance is evaluated in the ambient space.
Reduction removes a common factor that would change the behavior of the denominator while leaving the ratio unchanged on its common domain.
The reduced denominator is likewise attached to the coordinates in which $N/D$ is written (Remark~\ref{rem:denominator-representation}).

\subsection{First-order relative variance}
\label{sec:relative-variance}

We next define the first-order relative variance, the central quantity of the paper, and the denominator statistic.
In the Gaussian experiment, the covariance matrix $\Gamma(\Sigma)$ of $\sqrt{n}\vech(\Sighat-\Sigma)$ has entries $\Gamma_{(ij),(kl)}(\Sigma)=\sigma_{ik}\sigma_{jl}+\sigma_{il}\sigma_{jk}$, $1\le i\le j\le d$, $1\le k\le l\le d$, in the $\vech$ order;
the same formula defines $\Gamma$ on $\Sym_d$, and for $\Sigma\succeq0$ it is the covariance of $\vech(XX^\top)$ for $X\sim N_d(0,\Sigma)$.
For a differentiable scalar function $P$, define its Gaussian first-order variance on its differentiability domain by
\begin{equation}
    \label{eq:sD}
    s_P^2(\Sigma)
    =
    \nabla P(\Sigma)^\top\Gamma(\Sigma)\nabla P(\Sigma)
    =
    2\tr\{(G_P(\Sigma)\Sigma)^2\}
    .
\end{equation}
At a fixed $\Sigma\succ0$ this is the variance of the delta-method limit of $\sqrt{n}\{P(\Sighat)-P(\Sigma)\}$.
For a polynomial, both expressions are polynomials in $\Sigma$, the identity extends to $\Sym_d$, and $s_P$ denotes the nonnegative square root wherever the polynomial expression is used on the positive-semidefinite cone.

Normalizing by the value of $D$ gives the relative quantities of interest.
For any nonzero $D\in\RR[\Sym_d]$, put $\cU_D=\{\Sigma\succ0:D(\Sigma)\ne0\}$ as in Definition~\ref{def:strategy}, and on this set define
\begin{equation*}
    R_D(\Sigma)
    =
    \frac{G_D(\Sigma)\Sigma}{D(\Sigma)}
    ,\quad
    \kappa_D(\Sigma)
    =
    \frac{s_D^2(\Sigma)}{D(\Sigma)^2}
    =
    2\tr\{R_D(\Sigma)^2\}
    .
\end{equation*}
We call $R_D$ the \emph{relative gradient} and $\kappa_D$ the \emph{first-order relative variance} of $D$;
the same formulas define $R_F$ and $\kappa_F$ for any differentiable scalar covariance functional $F$, at positive-definite points of its domain where $F\ne0$ and in the dimension of its argument.
The matrix $R_D$ is similar to the symmetric matrix $\Sigma^{1/2}G_D\Sigma^{1/2}/D$, so its eigenvalues are real and $\kappa_D\ge0$.
At each $\Sigma\in\cU_D$, $\kappa_D(\Sigma)$ is the limiting variance of both $\sqrt{n}\{D(\Sighat)/D(\Sigma)-1\}$ and $\sqrt{n}\{\log|D(\Sighat)|-\log|D(\Sigma)|\}$.
For the per-observation Fisher--Rao metric of the centered Gaussian family \citep{skovgaard1984}, with tangent directions identified with $\Sym_d$,
\begin{equation}
    \label{eq:fisher-metric}
    g^{\mathrm{FR}}_\Sigma(H,K)
    =
    \frac{1}{2}\tr(\Sigma^{-1}H\Sigma^{-1}K)
    ,\quad
    H,K\in\Sym_d
    ,
\end{equation}
the gradient of $\log|D|$ at $\Sigma\in\cU_D$ is $2\Sigma G_D(\Sigma)\Sigma/D(\Sigma)$ and its squared norm is $\kappa_D(\Sigma)$.
Consequently $\kappa_D(\Sigma)/n$ is the Fisher-information lower bound for finite-variance estimators of $\log|D(\Sigma)|$ that are unbiased in a neighborhood of $\Sigma$;
Supplementary Lemma~\ref{lem:trace-main} derives the score, metric and bound with this normalization.

Two transformation properties of $\kappa_D$ are used throughout, both proved in Supplementary Lemma~\ref{lem:trace-main}.
For nonsingular $A\in\RR^{d\times d}$ and $D_A(S)=D(ASA^\top)$, $\kappa_{D_A}(S)=\kappa_D(ASA^\top)$ at every $S\succ0$ with $D_A(S)\ne0$.
If $M\in\RR^{r\times d}$ has full row rank and $D(S)=\bar{D}(MSM^\top)$ on $\Sym_d$ for a nonzero $\bar{D}\in\RR[\Sym_r]$, then $\kappa_D(S)=\kappa_{\bar{D}}(MSM^\top)$ on $\cU_D$, with $\kappa_{\bar{D}}$ computed in dimension $r$.
Thus first-order relative variance is preserved by nonsingular changes of observation coordinates and by adding unused coordinates.

\begin{remark}[Dependence on the ambient representation]
\label{rem:denominator-representation}
    A nonsingular linear change of the observation coordinates, $S\mapsto ASA^\top$, induces an automorphism of $\RR[\Sym_d]$, so it preserves coprimality, carries the reduced denominator to $D_A$ and leaves the first-order relative variance unchanged.
    Residual-covariance coordinates are instead a rational reparametrization in a different polynomial ring, so they can change the reduced denominator and its first-order relative variance;
    Supplementary Section~\ref{supp:adjusted-mediation} illustrates this with the covariate-adjusted mediation product.
\end{remark}

The squared signal-to-noise ratio of the denominator under the Gaussian reference is the \emph{denominator statistic} $\hat{F}_D=nD(\Sighat)^2/s_D^2(\Sighat)$, defined when $s_D^2(\Sighat)>0$;
it equals $n/\kappa_D(\Sighat)$ when also $D(\Sighat)\ne0$.
It is the analogue of a first-stage $F$ statistic, and Section~\ref{sec:fixed-law} replaces $\Gamma$ by an empirical quadratic-score covariance under other sampling laws.

\subsection{Self-normalization and relative pivots}
\label{sec:sn-formulation}

We now define the two denominator properties that Theorem~\ref{thm:sn-pivot-classification} relates.

\begin{definition}[Self-normalizing denominator]
    \label{def:self-normalization}
    A nonconstant $D\in\RR[\Sym_d]$ is \emph{self-normalizing} (SN) if, for some $c>0$,
    \begin{equation}
        \label{eq:SN}
        2\tr\{(G_D(\Sigma)\Sigma)^2\}
        =
        cD(\Sigma)^2
        \quad
        (\Sigma\in\Sym_d)
        .
    \end{equation}
    Equivalently, $\kappa_D\equiv c$ on $\cU_D$.
\end{definition}

The constant is unique, and equality on the positive-definite cone suffices to establish the polynomial identity.
The definition ranges over all nonconstant real covariance polynomials;
their homogeneity and factorization will follow from Theorem~\ref{thm:sn-pivot-classification}.
The term names the deterministic identity \eqref{eq:SN} and is unrelated to self-normalized sums and martingales \citep{delapenalaishao2009}, in which a statistic is divided by a random scale computed from the same data.

\begin{definition}[Relative pivot]
    \label{def:relative-pivot}
    Let $F$ be a Borel-measurable real-valued function on a specified Borel subset of $\Sym_d$.
    Let $\cU_F$ be a nonempty open subset of $\Smat^d_{++}$ contained in that domain, on which $F$ is nonzero.
    We say that $F$ has a \emph{relative pivot} on $\cU_F$ if, for every sufficiently large integer $n$ and every $\Sigma\in\cU_F$,
    $\Sighat$ belongs to the domain of $F$ almost surely and the law of $F(\Sighat)/F(\Sigma)$ does not depend on $\Sigma$.
\end{definition}

That law may depend on $n$ and $F$.
The sample-to-population ratio supplies the normalization, and its difference from one is the \emph{relative error}.
For a nonzero polynomial $F=D$, we use its polynomial extension to $\Sym_d$ and take $\cU_F=\cU_D$.

\section{Classification of self-normalizing denominators}
\label{sec:main}

\subsection{Flag powers}
\label{sec:flags}

This subsection introduces the polynomials singled out by the classification and records their structure.
For an $r$-dimensional subspace $V\subseteq\RR^d$, choose a matrix $B_V\in\RR^{d\times r}$ with orthonormal columns spanning $V$, and put $\Delta_V(\Sigma)=\det(B_V^\top\Sigma B_V)$ and $\Delta_{\{0\}}=1$.
The \emph{compression determinant} $\Delta_V$ is basis-independent and equals the generalized variance of the compressed observation $B_V^\top X_1$.
For an index set $J\subseteq[d]$, we also write $\Delta_J(\Sigma)=\det\Sigma_J$; this equals $\Delta_{\operatorname{span}\{e_j:j\in J\}}(\Sigma)$.
In particular, $\Delta_\emptyset=1$.
For $r\ge1$, $\Delta_V$ is homogeneous of degree $r$ and irreducible in $\RR[\Sym_d]$ and $\CC[\Sym_d]$ (Supplementary Lemma~\ref{lem:symdet-irreducible}).

\begin{definition}[Flag power]
    \label{def:flag-power}
    Let $k\ge1$, let $\{0\}\subsetneq V_1\subsetneq\cdots\subsetneq V_k\subseteq\RR^d$, put $r_j=\dim V_j$, and take $\gamma\in\RR\setminus\{0\}$ and $a_j\in\ZZ_{\ge1}$.
    A \emph{flag power} and its multiplicities are
    \begin{equation}
        \label{eq:flag-power}
        D
        =
        \gamma\prod_{j=1}^k\Delta_{V_j}^{a_j}
        ,\quad
        m_i
        =
        \sum_{j=1}^k a_j\1\{i\le r_j\}
        ,\quad
        1\le i\le d
        .
    \end{equation}
\end{definition}

Then $m_1\ge\cdots\ge m_d\ge0$, $m_1>0$ and $\deg D=\sum_im_i$;
we set $m_{d+1}=0$.
Coordinates are \emph{adapted} to the flag if they are given by an orthogonal matrix $O=(u_1,\ldots,u_d)$ with $V_j=\operatorname{span}(u_1,\ldots,u_{r_j})$ for every $j$.
Write $S=O^\top\Sigma O$.
The residual variance of the $i$th coordinate of $O^\top X_1$ regressed on the preceding coordinates is $v_i=\det(S_{1:i})/\det(S_{1:i-1})$, with $v_1=S_{11}$, and $D(\Sigma)=\gamma\prod_{i=1}^d v_i^{m_i}$.
When working entirely in adapted coordinates, we use $D(S)$ for the re-expressed polynomial $S\mapsto D(OSO^\top)$.
Conversely, a monomial $\prod_iv_i^{m_i}$ with integer exponents is a polynomial precisely when every $m_i-m_{i+1}$ is nonnegative (Section~\ref{supp:determinant}).

The proportionality constant is visible directly:
for $S\succ0$ and $1\le r\le d$, with $\Delta_r(S)=\det S_{1:r}$ and block sizes $r$ and $d-r$, the second block row and column being empty when $r=d$,
\begin{equation}
    \label{eq:flag-gradient}
    \frac{G_{\Delta_r}(S)S}{\Delta_r(S)}
    =
    \begin{pmatrix}
    I_r&S_{1:r}^{-1}S_{\{1,\ldots,r\},\{r+1,\ldots,d\}}\\
    0&0
    \end{pmatrix}
    ,
\end{equation}
and summing over the factors of \eqref{eq:flag-power} gives the following structure.

\begin{lemma}[Structure of flag powers]
    \label{lem:flag-structure}
    Let $D$ be a flag power with multiplicities $m_1\ge\cdots\ge m_d$.
    The subspaces with positive exponents in \eqref{eq:flag-power}, their exponents and $\gamma$ are determined by $D$.
    The polynomial $D$ is homogeneous of degree $\sum_im_i$, has the sign of $\gamma$ on $\Smat^d_{++}$, and $(\det\Sigma)^{m_d}$ is the largest power of $\det\Sigma$ dividing it.
    In adapted coordinates the relative gradient is upper triangular with diagonal $(m_1,\ldots,m_d)$;
    hence $\spec R_D(\Sigma)=\{m_1,\ldots,m_d\}$ and $\kappa_D\equiv2\sum_im_i^2$ on $\Smat^d_{++}$.
\end{lemma}

The proof is given in Section~\ref{supp:determinant}.
For example, in dimension three with coordinates $(X,M,Y)$, $\sigma_{XX}\det\Sigma_{\{X,M\}}$ has multiplicities $(2,1,0)$ and constant $10$, whereas $\det\Sigma$ has constant $2d$.

\subsection{The classification theorem}
\label{sec:classification}

The following theorem, the main result of the paper, classifies the covariance polynomials with a relative pivot.

\begin{theorem}[Self-normalization, relative pivots and flag powers]
    \label{thm:sn-pivot-classification}
    Let $D\in\RR[\Sym_d]$ be nonconstant.
    Under the Gaussian experiment \eqref{eq:sample-covariance}, the following are equivalent:
    \begin{enumerate}
        \item[(a)] $D$ is self-normalizing;
        \item[(b)] $D$ has a relative pivot;
        \item[(c)] $D$ is a flag power.
    \end{enumerate}
    In that case let $m=(m_1,\ldots,m_d)$ be the multiplicity vector of $D$ in Definition~\ref{def:flag-power}.
    The constant in \eqref{eq:SN} is $c=2\sum_{i=1}^d m_i^2$.
    For $r=\max\{i\in[d]:m_i>0\}$ and every $n\ge r$,
    \begin{equation}
        \label{eq:wishart-pivot}
        \frac{D(\Sighat)}{D(\Sigma)}
        \stackrel{d}{=}
        W_{n,m}
        :=
        \prod_{i=1}^r\Biggl(
            \frac{U_{n,i}}n
        \Biggr)^{m_i}
        ,
        \quad
        U_{n,i}\sim\chi^2_{n-i+1}
    \end{equation}
    with $U_{n,1},\ldots,U_{n,r}$ independent.
\end{theorem}
In words, within covariance polynomials a covariance-free first-order relative variance already forces a covariance-free finite-sample law, and nested compression determinants are the only way this can happen.
Two implications are classical.
Lemma~\ref{lem:flag-structure} gives (c)$\Rightarrow$(a), and Bartlett's decomposition of the leading blocks of $n\Sighat$ in adapted coordinates gives (c)$\Rightarrow$(b) together with \eqref{eq:wishart-pivot} \citep[Chapter~3]{muirhead1982}.
For (b)$\Rightarrow$(a), the delta method gives $\sqrt{n}\{D(\Sighat)/D(\Sigma)-1\}\dto N_1(0,\kappa_D(\Sigma))$ at each $\Sigma\in\cU_D$;
a common finite-sample law forces a common limit, so $\kappa_D$ is constant, positive because $D$ is nonconstant, and clearing denominators gives \eqref{eq:SN}.
Section~\ref{supp:pivots} details these steps.
The deterministic implication (a)$\Rightarrow$(c) is the main step;
Section~\ref{sec:classification-proof-overview} outlines it and Section~\ref{supp:classification-necessity} proves it.

\begin{remark}[Local form of condition (b)]
    \label{rem:local-pivot}
    The proof of (b)$\Rightarrow$(a) uses only that the law of $D(\Sighat)/D(\Sigma)$ is the same for all $\Sigma$ in one nonempty open subset of $\cU_D$, along one unbounded sequence of sample sizes, so (b) may be weakened accordingly.
    The open set cannot be dropped:
    on the diagonal bivariate submodel, $\sigma_{11}\sigma_{22}$ has a covariance-free sample-to-population law without being self-normalizing (Section~\ref{supp:pivots}).
\end{remark}

In adapted coordinates, a lower-triangular $T$ with positive diagonal satisfies $D(TST^\top)=\{\prod_iT_{ii}^{2m_i}\}D(S)$, the relative invariance underlying generalized power functions \citep{farautkoranyi1994,ishi2001}.
The converse derives the fixed flag from \eqref{eq:SN} rather than assuming this action.
Polynomiality is essential:
the smooth function in Supplementary Example~\ref{ex:fisher-exponential} has constant first-order relative variance but no relative pivot.

\begin{remark}[Samplewise constancy of the denominator statistic]
\label{rem:sn-diagnostic}
    Under Gaussian sampling with $n\ge d$, self-normalization is equivalent to almost-sure constancy of $\hat{F}_D$.
    Indeed, \eqref{eq:SN} gives $\hat{F}_D=n/c$ whenever $D(\Sighat)\ne0$, under any sampling law.
    Conversely, the Gaussian sample covariance has full support on the positive-definite cone, so almost-sure constancy gives constancy of $nD(S)^2/s_D^2(S)$ on $\{S\succ0:D(S)\ne0,\ s_D^2(S)>0\}$ by continuity, and hence \eqref{eq:SN} by clearing denominators and polynomial continuation.
    The samplewise identity uses only the Gaussian-reference variance formula;
    the law \eqref{eq:wishart-pivot} additionally uses Gaussian sampling, and Section~\ref{sec:fixed-law} treats other laws.
\end{remark}

\subsection{Statistical consequences}
\label{sec:consequences}

We now derive the statistical consequences of Theorem~\ref{thm:sn-pivot-classification}: exact calibration, the extension to precision polynomials, variance stabilization and a finite certification test.

\begin{corollary}[Exact calibration]
    \label{cor:exact-calibration}
    Let $D$ be self-normalizing, with positive scalar factor $\gamma$ in its flag representation and self-normalization constant $c$.
    Let $r$ and $W_{n,m}$ be as in \eqref{eq:wishart-pivot}.
    For $0<p<1$, let $r_{n,p}$ be the $p$ quantile of $W_{n,m}$.
    For $0<\alpha<1$ and $n\ge r$, the interval
    \begin{equation}
        \label{eq:sn-exact-calibration}
        \left[
            \frac{D(\Sighat)}{r_{n,1-\alpha/2}},
            \,
            \frac{D(\Sighat)}{r_{n,\alpha/2}}
        \right]
    \end{equation}
    has coverage exactly $1-\alpha$ for $D(\Sigma)$ at every $\Sigma\succ0$.
    Moreover $\sqrt{n/c}\{D(\Sighat)/D(\Sigma)-1\}\dto N_1(0,1)$ and $\sqrt{n/c}\log\{D(\Sighat)/D(\Sigma)\}\dto N_1(0,1)$, uniformly in $\Sigma\succ0$.
\end{corollary}

The proof is given in Section~\ref{supp:exact-calibration}.
The interval requires no estimated variance, and the uniformity holds because the law in \eqref{eq:wishart-pivot} does not depend on $\Sigma$.

\begin{corollary}[Polynomial functionals of the precision matrix]
    \label{cor:bcd-precision}
    Let $q\in\RR[\Sym_d]$ be nonconstant and put $F(\Sigma)=q(\Sigma^{-1})^\varepsilon$, $\varepsilon\in\{-1,1\}$, on $\cV_q=\{\Sigma\succ0:q(\Sigma^{-1})\ne0\}$.
    The following are equivalent:
    $F$ has a relative pivot;
    $\kappa_F$ is a positive constant on $\cV_q$;
    and $q$ is a flag power as a polynomial in $\Sigma^{-1}$.
    In that case, with the multiplicities $m_1\ge\cdots\ge m_d$ of $q$, the constant is $c=2\sum_im_i^2$, $\cV_q=\Smat^d_{++}$ and, for $n\ge d$,
    \begin{equation}
        \label{eq:precision-pivot}
        \frac{F(\Sighat)}{F(\Sigma)}
        \stackrel{d}{=}
        \prod_{i=1}^d\Biggl(
            \frac{U_{n,i}}n
        \Biggr)^{-\varepsilon m_{d+1-i}}
        ,
        \quad
        U_{n,i}\sim\chi^2_{n-i+1}
    \end{equation}
    with $U_{n,1},\ldots,U_{n,d}$ independent.
\end{corollary}

The proof is given in Section~\ref{supp:precision}, and the local form in Remark~\ref{rem:local-pivot} applies to $F$ as well.
Inversion preserves the Fisher--Rao metric and gives $\kappa_F(\Sigma)=\kappa_q(\Sigma^{-1})$.
Supplementary Lemma~\ref{lem:bcd-inversion-isometry} uses complementary minors to reverse the flag and its variance exponents.
The corollary covers precision quadratic forms $u^\top\Sigma^{-1}u$ and their reciprocals, which include the residual variance $1/(\Sigma^{-1})_{jj}=\var(X_j\mid X_{-j})$, where $X_{-j}$ omits coordinate $j$ of $X\sim N_d(0,\Sigma)$, and the minimum-variance portfolio variance $1/(\1_d^\top\Sigma^{-1}\1_d)=\min\{w^\top\Sigma w:\1_d^\top w=1\}$.

The classification also settles scalar variance stabilization for homogeneous polynomials.

\begin{corollary}[Logarithmic variance stabilization]
    \label{cor:log-variance-stabilization}
    Let $D\in\RR[\Sym_d]$ be nonconstant and homogeneous of degree $K\ge1$.
    There exist an open interval $J$ containing $D(\Smat^d_{++})$, a continuously differentiable $g:J\to\RR$ and $v_0>0$ such that
    \begin{equation}
        \label{eq:vst}
        \{g'(D(\Sigma))\}^2s_D^2(\Sigma)
        =
        v_0
        \quad
        (\Sigma\succ0)
    \end{equation}
    if and only if $D$ is a flag power.
    In that case its range is one open half-line and, on that range, $g(x)=a\log|x|+b$ for constants $a,b\in\RR$ with $a\ne0$, and $v_0=a^2c$, where $c$ is the self-normalization constant of $D$.
\end{corollary}
The proof is given in Section~\ref{supp:vst}.
Homogeneity makes each scaling ray attain the entire nonzero half-line.
Along that ray, \eqref{eq:vst} determines $x^2g'(x)^2$, and hence $g$, on the full range.
The nonhomogeneous polynomial $1+\det\Sigma$ instead admits $g(x)=\log(x-1)$.

The final corollary turns the classification into a finite test.

\begin{corollary}[Reconstruction and certification]
    \label{cor:one-point-diagnostic}
    Let $D\in\RR[\Sym_d]$ be nonconstant.
    If $D(I_d)=0$, then $D$ is not self-normalizing.
    Otherwise let $B=G_D(I_d)/D(I_d)$ have ordered eigenvalues $m_1\ge\cdots\ge m_d$, and put $m_{d+1}=0$.
    If these eigenvalues are not all nonnegative integers, or if $m_1=0$, then $D$ is not self-normalizing.
    When they are nonnegative integers with $m_1>0$, choose an orthonormal eigenbasis $u_1,\ldots,u_d$ in this order and put $O_r=(u_1,\ldots,u_r)$ for $1\le r\le d$.
    The right-hand side of
    \begin{equation}
        \label{eq:one-point-reconstruction}
        D(\Sigma)
        \equiv
        D(I_d)\prod_{r=1}^d
        \det(O_r^\top\Sigma O_r)^{m_r-m_{r+1}}
    \end{equation}
    is then a polynomial independent of the choice of eigenbasis.
    Under this spectral condition, $D$ is self-normalizing if and only if \eqref{eq:one-point-reconstruction} holds as a polynomial identity on $\Sym_d$.
\end{corollary}

The proof is given in Section~\ref{supp:reconstruction}.
Certification thus consists of a spectral test at $I_d$ and one exact coefficient comparison;
for $d\ge2$, $(\tr\Sigma)^d$ passes the spectral test with $B=I_d$, but its candidate is $d^d\det\Sigma$.
Section~\ref{supp:reconstruction} also gives an exact-arithmetic construction from spectral subspaces.

\subsection{Outline of the proof of necessity}
\label{sec:classification-proof-overview}

Assume \eqref{eq:SN} and put $\alpha=c/2$.
The bounded norm of the logarithmic gradient prevents $\log|D|$ from tending to $-\infty$ along a finite-length path within the positive cone.
Hence $D$ is nowhere zero there (Lemma~\ref{lem:bcd-nozero}).
At any $A\succ0$, define the polynomial $P(S)=D(A^{1/2}SA^{1/2})/D(A)$ for $S\in\Sym_d$.
Then $P(I_d)=1$ and $P>0$ on $\Smat^d_{++}$.

On this cone write $J=G_P/P$ and $B=J(I_d)$.
The differentiated identity makes $J(S(t))S(t)$ constant along the flow $\dot{S}(t)=S(t)J(S(t))S(t)$ with $S(0)=I_d$.
The resulting line in logarithmic coordinates gives $P(\exp(tB))=\exp(\alpha t)$ for all real $t$.
Choose an orthogonal matrix $O$ whose columns are eigenvectors of $B$, and write $B=O\diag(\lambda_1,\ldots,\lambda_d)O^\top$ and $\lambda=(\lambda_1,\ldots,\lambda_d)^\top$.
For $u=(u_1,\ldots,u_d)^\top\in\RR^d$, define $h(u)=\log P\{O\diag(\exp(u_1),\ldots,\exp(u_d))O^\top\}$.
Then $\|\nabla h(u)\|^2\le\alpha$, $h(t\lambda)=\alpha t$ and $\|\lambda\|^2=\alpha$.
A two-sided Lipschitz comparison with this entire line gives $h(u)=\lambda^\top u$ (Lemma~\ref{lem:calibrated-flat}).
Its exponential is a polynomial on the diagonal, forcing $\lambda_i\in\ZZ_{\ge0}$.
The base point was arbitrary;
continuity and discreteness make the relative spectrum constant.
Its trace yields Euler's identity and homogeneity (Proposition~\ref{prop:integral-spectrum}).
This step fixes the eigenvalues, while the common subspaces remain to be determined.

Let $\ell\in\ZZ_{\ge0}$ be the largest integer for which $(\det\Sigma)^\ell$ divides $D$, and write $D=(\det\Sigma)^\ell E$.
If $E$ is constant, the flag representation is already obtained.
Otherwise, determinant stripping preserves the variance identity with a positive constant (Lemma~\ref{lem:strip-supp}), and the identity $\det G_E\det\Sigma=\{\prod_{i=1}^d\lambda_i(E)\}E^d$ for the fixed relative eigenvalues $\lambda_i(E)$ of $R_E$, together with irreducibility of $\det\Sigma$, forces a zero relative eigenvalue of $E$.
After multiplying $E$ by its sign on the positive cone, its gradient is therefore positive semidefinite of fixed rank $r=\rank G_E(I_d)$ with $1\le r<d$.

The elementary symmetric polynomial of order $r+1$ vanishes identically on $G_E$.
Let $U=\ker G_E(I_d)$, and let $C_0$ be the $r$th Newton transformation of $G_E(I_d)$ (Section~\ref{supp:fixed-kernel}), a positive multiple of the orthogonal projection onto $U$ by \eqref{eq:S1-newton-projection}.
Differentiating the vanishing identity through the symmetric Hessian makes $E$ constant along the line through $I_d$ in direction $C_0$ (Lemma~\ref{lem:Newton-characteristics}).
For any $A_0\succ0$, choose $\lambda_0>0$ with $A_0\preceq\lambda_0I_d$, where $H\preceq K$ means $K-H\succeq0$.
Loewner monotonicity of $E$, which follows from $G_E\succeq0$, and homogeneity give $E(A_0)\le E(A_0+tC_0)\le E(\lambda_0I_d+tC_0)=\lambda_0^{\deg E}E(I_d)$ for $t\ge0$, so the polynomial $t\mapsto E(A_0+tC_0)$ is bounded on a half-line and hence constant.
Differentiation gives $\tr\{G_E(A_0)C_0\}=0$;
positivity and the fixed rank imply $\ker G_E(A_0)=U$ for every $A_0\succ0$.
Thus $E$ depends only on the compression to $U^\perp$ (Proposition~\ref{prop:fixed-kernel-compression}).

Compression preserves the variance identity.
Induction on dimension produces a flag in $U^\perp$;
lifting it and restoring the full determinant produces a nested flag in $\RR^d$.
The determinant-free residue has minimum relative eigenvalue zero, so restoration gives $\ell=m_d$.
Supplementary Section~\ref{supp:classification} supplies all intermediate identities and completes this induction.

\section{Relatively stable denominators}
\label{sec:stability}

Self-normalization is rigid, whereas $\sigma_{11}\sigma_{22}$ in \eqref{eq:intro-three-denominators} has bounded but nonconstant first-order relative variance.
This section characterizes bounded first-order relative variance, derives what it yields for the ratio, locates its boundary, and extends it beyond Gaussian sampling.

\subsection{Criteria for relative stability}
\label{sec:stability-criteria}

\begin{definition}[Relative stability]
    \label{def:bcd-stable}
    For $D\not\equiv0$ in $\RR[\Sym_d]$, put $\bar{\kappa}_D=\sup_{\Sigma\in\cU_D}\kappa_D(\Sigma)$.
    We call $D$ \emph{relatively stable} if $\bar{\kappa}_D<\infty$.
\end{definition}

The criteria below use the polynomial normalized at the population covariance.
Write $\operatorname{GL}_d(\RR)$ for the set of invertible real $d\times d$ matrices and define
\begin{equation}
    \label{eq:normalized-orbit}
    \cA_D
    =
    \{A\in\operatorname{GL}_d(\RR):D(AA^\top)\ne0\}
    ,\quad
    P_A(H)
    =
    \frac{D(AHA^\top)}{D(AA^\top)}
    ,\quad
    B_A
    =
    \frac{A^\top G_D(AA^\top)A}{D(AA^\top)}
    .
\end{equation}
Here $A\in\cA_D$ and $H\in\Sym_d$;
for each such $A$, $P_A\in\RR[\Sym_d]$, $P_A(I_d)=1$ and $G_{P_A}(I_d)=B_A$.
Coefficient norms refer to the coefficients in the monomial basis of polynomials of degree at most $\deg D$ in $\vech(H)$;
this space is finite-dimensional, so boundedness does not depend on the norm.
For $S_n^\circ=n^{-1}\sum_iZ_iZ_i^\top$ with independent $Z_i\sim N_d(0,I_d)$, the sample-to-population ratio at $\Sigma=AA^\top$ has law $P_A(S_n^\circ)$.
The randomness is thus evaluated under one fixed law while the normalized polynomial varies with $A$.

The following theorem shows that relative stability is a deterministic property of $D$, checkable through (ii) or (iii), and that it is exactly uniform parametric-rate control of the relative error.

\begin{theorem}[Equivalent criteria for relative stability]
    \label{thm:uniform-denominator}
    Let $D\in\RR[\Sym_d]$ be nonzero.
    The following conditions are equivalent.
    \begin{enumerate}
        \item[(i)] $\bar{\kappa}_D<\infty$.
        \item[(ii)] $D$ is nowhere zero on $\Smat^d_{++}$ and there is a finite $\Lambda\ge0$ such that, for all $\Sigma,\Theta\succ0$,
        \begin{equation}
            \label{eq:log-Lipschitz}
            \bigl|
                \log|D(\Theta)|-\log|D(\Sigma)|
            \bigr|
            \le
            \Lambda\Bigl\|
                \log(\Sigma^{-1/2}\Theta\Sigma^{-1/2})
            \Bigr\|_F
            .
        \end{equation}
        \item[(iii)] $\{P_A:A\in\cA_D\}$ is bounded in coefficient norm.
        \item[(iv)] Under Gaussian sampling, the relative errors are uniformly root-$n$ tight:
        \begin{equation}
            \label{eq:uniform-tightness}
            \lim_{M\to\infty}\limsup_{n\to\infty}
            \sup_{\Sigma\in\cU_D}
            \Pr_\Sigma\Biggl\{
                \sqrt{n}\Biggl|
                    \frac{D(\Sighat)}{D(\Sigma)}-1
                \Biggr|>M
            \Biggr\}
            =
            0
            .
        \end{equation}
    \end{enumerate}
    When these conditions hold, the least constant in \eqref{eq:log-Lipschitz} is $\sqrt{\bar{\kappa}_D/2}$.
\end{theorem}
The proof is given in Section~\ref{supp:stability-proof}.
For (i)$\Rightarrow$(ii), the bound on the gradient of $\log|D|$ gives nonvanishing (Supplementary Lemma~\ref{lem:bcd-nozero}) and \eqref{eq:log-Lipschitz} by integration along the geodesic from $\Sigma$ to $\Theta$ in the metric $2g^{\mathrm{FR}}$, whose length is the norm in \eqref{eq:log-Lipschitz}, $\sqrt{2}$ times the Fisher--Rao distance.
For (ii)$\Rightarrow$(iii), on a fixed ball about $I_d$ the bound controls $P_A$ uniformly above and away from zero, and equivalence of polynomial norms then controls all its coefficients.
For (iii)$\Rightarrow$(iv), a uniform Taylor expansion gives
\begin{equation}
    \label{eq:uniform-expansion}
    \sqrt{n}\{P_A(S_n^\circ)-1\}
    =
    \frac{1}{\sqrt{n}}\sum_{i=1}^n
    \{Z_i^\top B_AZ_i-\tr B_A\}
    +
    o_p(1)
\end{equation}
uniformly in $A$, and the leading term has variance $2\|B_A\|_F^2=\kappa_D(AA^\top)$.
For (iv)$\Rightarrow$(i), an unbounded collection of pointwise limiting variances contradicts \eqref{eq:uniform-tightness}.
The Taylor remainder is where uniform delta-method arguments enter \citep{kasy2019}.

\begin{corollary}[Uniform normal approximation]
    \label{cor:uniform-normal}
    Let $D$ be relatively stable, and let $v_0>0$ be such that $\cC_D(v_0)=\{\Sigma\succ0:\kappa_D(\Sigma)\ge v_0\}$ is nonempty.
    Then
    \begin{equation}
        \label{eq:Gaussian-uniform-normal}
        \sup_{\Sigma\in\cC_D(v_0)}\sup_{x\in\RR}
        \Biggl|
            \Pr_\Sigma\Biggl\{
                \sqrt{\frac{n}{\kappa_D(\Sigma)}}
                \Biggl(
                    \frac{D(\Sighat)}{D(\Sigma)}-1
                \Biggr)\le x
            \Biggr\}
            -\Phi(x)
        \Biggr|
        \longrightarrow
        0
        ,
    \end{equation}
    where $\Phi$ is the standard normal distribution function, and the same holds with $D(\Sighat)/D(\Sigma)-1$ replaced by $\log\{D(\Sighat)/D(\Sigma)\}$.
    Both conclusions remain valid with $\kappa_D(\Sighat)$ in place of $\kappa_D(\Sigma)$.
\end{corollary}

The proof is given in Section~\ref{supp:stability-proof}.
In particular, if $D$ is positive on the cone, $0<\alpha<1$ and $z_p=\Phi^{-1}(p)$, the interval with endpoints $D(\Sighat)\exp\{\pm z_{1-\alpha/2}\sqrt{\kappa_D(\Sighat)/n}\}$ has coverage for $D(\Sigma)$ converging to $1-\alpha$ uniformly on $\cC_D(v_0)$;
this is the plug-in interval used in Section~\ref{sec:numerics}.
We say that a family of statistics is \emph{asymptotically standard normal uniformly} over a set of parameters when the corresponding supremum of distribution-function differences tends to zero, as in \eqref{eq:Gaussian-uniform-normal}.

Positive relatively stable polynomials are closed under products, positive linear combinations and nonsingular congruences (Supplementary Lemma~\ref{lem:bcd-closure}), and each compression determinant $\Delta_V$ has first-order relative variance $2\dim V$, so products of compression determinants are relatively stable, nested or not.
Euler's identity gives $\kappa_D\ge2K^2/d$ for homogeneous $D$ of degree $K\ge1$, while $\kappa_{\tr(\Sigma^r)}\le2r^2$ for every integer $r\ge1$;
in dimension two, $\tr(\Sigma^2)$ is relatively stable but not a finite positive sum of products of compression determinants (Supplementary Example~\ref{ex:bcd-positive-unstable}).

\subsection{Consequences for the ratio}
\label{sec:ratio-removal}

We now show what relative stability buys for the ratio estimator.
For $\tau=N(\Sigma)/D(\Sigma)$, the exact identity $\hat{\tau}-\tau=[\{N(\Sighat)-\tau D(\Sighat)\}/D(\Sigma)]\{D(\Sigma)/D(\Sighat)\}$ separates the normalized moment from the relative denominator error.

\begin{corollary}[Removal of denominator randomness]
    \label{cor:denominator-removal}
    Let $D$ be relatively stable, let $N\in\RR[\Sym_d]$, and let $\Sigma_n\succ0$ and $\ell_n>0$ be deterministic sequences.
    Put $\tau_n=N(\Sigma_n)/D(\Sigma_n)$ and $\hat{\tau}=N(\Sighat)/D(\Sighat)$.
    Under Gaussian sampling with covariance $\Sigma_n$, if $\zeta_n=\ell_n\{N(\Sighat)-\tau_nD(\Sighat)\}/D(\Sigma_n)=O_p(1)$, then $\ell_n(\hat{\tau}-\tau_n)=\zeta_n+o_p(1)$.
\end{corollary}

The proof (Section~\ref{supp:denominator-removal}) uses that the difference is $\zeta_n\{D(\Sigma_n)/D(\Sighat)-1\}$:
relative stability makes the second factor converge to zero and transfers every limit of the tight normalized moment to the ratio.
The moment $N-\tau_nD$ therefore determines the remaining precision, limit shape and studentization;
Section~\ref{sec:mediation} treats one such problem under covariance degeneration.

Self-normalization additionally gives the relative pivot in Theorem~\ref{thm:sn-pivot-classification} and $\hat{F}_D=n/c$.
More generally, for nonconstant relatively stable $D$, $\hat{F}_D=n/\kappa_D(\Sighat)\ge n/\bar{\kappa}_D$ wherever it is defined, so every fixed diagnostic threshold is eventually exceeded throughout the sample domain.
If $D$ is a flag power and $\tau$ depends only on full-predecessor regression coefficients in coordinates adapted to its flag, Gaussian-reference score orthogonality reduces moment inversion to the Wald interval widened by a deterministic factor (Supplementary Corollary~\ref{cor:flag-fieller-enlargement}).

\subsection{Weak sequences, minors and positive polynomials}
\label{sec:weak-minors}

To distinguish the orders at which relative stability can fail, we first specify a nondegenerate first-order denominator limit.

\begin{definition}[First-order weak denominator sequence]
    \label{def:bcd-weak-sequence}
    For a nonzero $D\in\RR[\Sym_d]$, a sequence $\Sigma_n\in\cU_D$ is a \emph{first-order weak denominator sequence} for $D$ if $\Sigma_n\to\Sigma_\ast\succeq0$, $\sqrt{n} D(\Sigma_n)\to\delta\in\RR$, and $s_D(\Sigma_\ast)>0$.
\end{definition}

Along such a sequence, Supplementary Proposition~\ref{prop:transfer} gives $\sqrt{n}D(\Sighat)\dto N_1(\delta,s_D^2(\Sigma_{\ast}))$ and
\begin{equation}
    \label{eq:weak-diagnostic}
    \hat{F}_D
    \dto
    (\mu+Z)^2
    ,
    \quad
    \mu
    =
    \delta/s_D(\Sigma_{\ast})
    ,\quad
    Z\sim N_1(0,1)
    .
\end{equation}
The limiting law is noncentral $\chi^2_1$ with noncentrality $\mu^2$, and the same proposition gives the accompanying ratio limit under joint-moment conditions, in the classical ratio setting \citep{fieller1954,staigerstock1997}.
For relatively stable $D$, the inequality $s_D^2\le\bar{\kappa}_DD^2$ extends to the closed cone by continuity, so $s_D(\Sigma_{\ast})>0$ forces $D(\Sigma_{\ast})\ne0$ and $\sqrt{n}D(\Sigma_n)\to\pm\infty$;
no first-order weak denominator sequence exists.

For equal-size nonempty index sets $I,J$, write $M_{I,J}=\det\Sigma_{I,J}$ in increasing orders, call it principal when $I=J$, and call it almost-principal when $I$ and $J$ differ in exactly one element.
Factors that differ by a nonzero scalar are associates;
we collect associates, including $M_{I,J}=M_{J,I}$, before assigning multiplicities.

\begin{proposition}[Stability and first-order weakness of minor products]
    \label{prop:bcd-minor-products}
    Let $D=\gamma\prod_{j=1}^k f_j^{a_j}$, where $\gamma\ne0$, $a_j\in\ZZ_{\ge1}$ and the $f_j$ are pairwise nonassociate covariance minors.
    Then $D$ is relatively stable if and only if every $f_j$ is principal.
    It admits a first-order weak denominator sequence if and only if some nonprincipal $f_j$ has multiplicity $a_j=1$.
    If all nonprincipal factors have multiplicity at least two, then $s_D(\Sigma_{\ast})=0$ at every $\Sigma_{\ast}\succeq0$ satisfying $D(\Sigma_{\ast})=0$.
\end{proposition}

The proof is given in Section~\ref{supp:minor-products-proof}.
Each nonprincipal minor has a smooth positive-definite zero.
Irreducibility gives a nearby zero at which the other factors remain nonzero.
This yields the simple-factor direction.
Repeated factors and principal-minor variance identities give the converse.

\begin{proposition}[First-order degeneracy at zeros of positive polynomials]
    \label{prop:positive-boundary-degeneracy}
    Suppose $D\in\RR[\Sym_d]$ is positive on $\Smat^d_{++}$.
    If $\Sigma_{\ast}\succeq0$ and $D(\Sigma_{\ast})=0$, then $s_D(\Sigma_{\ast})=0$.
    Consequently $D$ admits no first-order weak denominator sequence.
\end{proposition}

Positivity excludes first-order weak denominator sequences, but not relative instability.
The proof (Section~\ref{supp:positive-boundary-proof}) writes $\Sigma_\ast=AA^\top$ with $A$ of full column rank;
the nonnegative polynomial $H\mapsto D(AHA^\top)$ then has an interior minimum at the identity, so $A^\top G_D(\Sigma_\ast)A=0$, which gives $s_D(\Sigma_\ast)=0$ by \eqref{eq:sD}.

\begin{example}[Repeated zeros and higher-order instability]
    \label{ex:squared-covariance}
    For $d=2$ and $D=\sigma_{12}^2$, $s_D^2=4\sigma_{12}^2(\sigma_{11}\sigma_{22}+\sigma_{12}^2)$.
    Thus $s_D$ vanishes at every zero while $\kappa_D$ is unbounded.
    For a fixed $h\in\RR\setminus\{0\}$ and $n>h^2$, take $\Sigma_n=\left(\begin{smallmatrix}1&h/\sqrt{n}\\h/\sqrt{n}&1\end{smallmatrix}\right)$.
    Under Gaussian sampling with covariance $\Sigma_n$, $D(\Sighat)/D(\Sigma_n)\dto(h+Z)^2/h^2$ and $\hat{F}_D\dto(h+Z)^2/4$, where $Z\sim N_1(0,1)$.
    Supplementary Example~\ref{ex:bcd-positive-unstable} and Example~\ref{ex:bcd-napkin} give positive, relatively unstable polynomials, the latter from an identifying formula.
\end{example}

\begin{remark}[Self-normalization, stability and orders of weakness]
    \label{rem:stability-hierarchy}
    Among nonconstant polynomials in dimension $d\ge2$, the implications SN $\Rightarrow$ relative stability $\Rightarrow$ absence of first-order weak denominator sequences are both strict.
    The respective counterexamples are $\sigma_{11}\sigma_{22}$ and $\sigma_{12}^2$, with unused coordinates added as needed.
    First-order weak sequences have a Gaussian denominator limit with nonvanishing noise;
    repeated zeros and positive boundary zeros can instead produce higher-order relative limits.
    The implication for a full ratio additionally uses its moment conditions, as specified in Supplementary Proposition~\ref{prop:transfer}.
\end{remark}

\subsection{Fixed standardized sampling laws}
\label{sec:fixed-law}

We now show that relative stability, unlike the exact law of Theorem~\ref{thm:sn-pivot-classification}, survives a change of sampling law.
Fix a law $Q$ for $Z\in\RR^d$ such that $\EE_Q(Z)=0$, $\EE_Q(ZZ^\top)=I_d$ and $\EE_Q(\|Z\|^4)<\infty$.
Let $Z_1,\ldots,Z_n$ be independent vectors with law $Q$.
For $A\in\operatorname{GL}_d(\RR)$, put $X_i=AZ_i$, and let $\Pr_{A,Q}$ denote probability under the joint law of the independent observations $X_1,\ldots,X_n$.
With $S_n=n^{-1}\sum_{i=1}^n Z_iZ_i^\top$, we have $\Sighat=AS_nA^\top$, and the sample-to-population denominator ratio equals $P_A(S_n)$ for $A\in\cA_D$.
With the quadratic scores $y_i=\vech(X_iX_i^\top)$ and $\bar y=n^{-1}\sum_i y_i$, put
\begin{equation}
\label{eq:empirical-scores}
    \hat{\Gamma}_{\mathrm{emp}}
    =
    \frac{1}{n}\sum_i(y_i-\bar y)(y_i-\bar y)^\top
    ,
    \quad
    \hat{v}_D
    =
    \frac{\nabla D(\Sighat)^\top\hat{\Gamma}_{\mathrm{emp}}
    \nabla D(\Sighat)}{D(\Sighat)^2}
    .
\end{equation}
Statistics whose defining formulas are undefined, because a denominator or an estimated variance vanishes or a logarithm has a nonpositive argument, are set to zero;
these events have probability tending to zero in all results below.

\begin{proposition}[Relative stability and empirical studentization under a fixed law]
    \label{prop:fixed-law-stability}
    Let $Q$ be as above, let $D\in\RR[\Sym_d]$ be nonzero, and for $A\in\cA_D$ and $v_0>0$ put $v_Q(A)=\var_Q(Z^\top B_AZ)$ and $\cA_Q(v_0)=\{A\in\cA_D:v_Q(A)\ge v_0\}$.
    In (i) and (ii), $D$ is relatively stable, so $\cA_D=\operatorname{GL}_d(\RR)$ and suprema and infima over $A$ range over this set.
    \begin{enumerate}
        \item[(i)]
        $D(\Sighat)/D(AA^\top)=1+O_p(n^{-1/2})$ uniformly in $A\in\operatorname{GL}_d(\RR)$, and $\inf_A\Pr_{A,Q}(\Sighat\succ0)\to1$.
        
        \item[(ii)]
        For every $\epsilon>0$, $\sup_A\Pr_{A,Q}(|\hat{v}_D-v_Q(A)|>\epsilon)\to0$, and for every $v_0>0$ with $\cA_Q(v_0)$ nonempty, both $\sqrt{n/v_Q(A)}\{D(\Sighat)/D(AA^\top)-1\}$ and $\sqrt{n/\hat{v}_D}\{D(\Sighat)/D(AA^\top)-1\}$ are asymptotically standard normal uniformly over $A\in\cA_Q(v_0)$.
        
        \item[(iii)]
        Conversely, if $\var_Q(Z^\top HZ)>0$ for every nonzero $H\in\Sym_d$ and the relative error of a polynomial $D\not\equiv0$ is uniformly root-$n$ tight over $A\in\cA_D$ in the sense of \eqref{eq:uniform-tightness} with $\Pr_{A,Q}$, then $\bar{\kappa}_D<\infty$.
    \end{enumerate}
\end{proposition}

The proof is given in Section~\ref{supp:fixed-law}, together with the fourth-moment and centering calculations;
the same conclusions hold for the logarithmic relative error.
Thus the deterministic criteria of Theorem~\ref{thm:uniform-denominator} transfer denominator control to any fixed sampling law with finite fourth moments, whereas the exact law of Theorem~\ref{thm:sn-pivot-classification} does not transfer.
The quantity $v_Q(A)$ is the actual first-order relative variance and can vary even for a flag power;
if $Q$ is invariant under orthogonal transformations, flag powers retain a relative pivot, with a law depending on $Q$ (Supplementary Lemma~\ref{lem:rotational-flag-pivot}).
Supplementary Example~\ref{ex:fixed-law-degeneracy} shows why (iii) needs the nondegeneracy assumption.
Corollary~\ref{cor:denominator-removal} holds under $\Pr_{A_n,Q}$ along any sequence $A_n$ (Section~\ref{supp:denominator-removal}).

\section{Examples and inference for a ratio}
\label{sec:examples}

This section locates the classes of Sections~\ref{sec:main} and~\ref{sec:stability} in reduced ambient rational functions used in practice.

\subsection{Regression along one order}
\label{sec:recursive}

We first show that regression along one order produces only constant or flag-power denominators.
Fix an order of the coordinates, and for $1\le i<j\le d$ let $\beta_{ji}(\Sigma)=\{\Sigma_{1:j-1}^{-1}\Sigma_{\{1,\ldots,j-1\},j}\}_i$ be the coefficient of coordinate $i$ in the regression of coordinate $j$ on all its predecessors;
put $\Delta_r=\det\Sigma_{1:r}$ and $\Delta_0=1$ as in Section~\ref{sec:flags}.
Set $\beta_{ji}=0$ for $i\ge j$, so that $\beta=(\beta_{ji})_{j,i=1}^d$ is strictly lower triangular.
Then $L=(I_d-\beta)^{-1}$ is unit lower triangular and $\Sigma=L\diag(v_1,\ldots,v_d)L^\top$, with residual variances $v_j=\Delta_j/\Delta_{j-1}$ for $1\le j\le d$.

\begin{proposition}[Flag denominators from recursive regression]
    \label{prop:bcd-regression-source}
    Every real polynomial in the functions $\beta_{ji}$, $1\le i<j\le d$, has a reduced denominator that is constant or of the form $\gamma\prod_{r=1}^{d-1}\Delta_r^{a_r}$ with $\gamma\ne0$ and $a_r\in\ZZ_{\ge0}$, hence a flag power.
    The same holds for polynomials in the entries of $L$.
\end{proposition}

Thus regression along one order never produces a relatively unstable denominator, whatever polynomial of the coefficients is targeted.
The proof is given in Section~\ref{supp:recursive}.
Cramer's rule gives denominator $\Delta_{j-1}$ for each coefficient, reduction preserves products of these irreducible leading minors, and $L=I_d+\beta+\cdots+\beta^{d-1}$.
The next proposition identifies the reduced denominator exactly for the products that arise in mediation.

\begin{proposition}[Adjacent path products]
    \label{prop:path-products}
    For integers $q\ge0$, $\ell\ge1$ and $d\ge q+\ell+1$, the product $T_{q,\ell}=\prod_{j=q+1}^{q+\ell}\beta_{j+1,j}$ has reduced denominator and self-normalization constant
    \begin{equation*}
        D_{q,\ell}=\prod_{j=q+1}^{q+\ell}\Delta_j
        ,
        \quad
        c_{q,\ell}=2q\ell^2+\frac{\ell(\ell+1)(2\ell+1)}3
        .
    \end{equation*}
\end{proposition}

The proof is given in Section~\ref{supp:recursive}.
The irreducible nonprincipal numerator factors are coprime to the denominator, and the multiplicities are $\ell$ in the first $q+1$ positions, followed by $\ell-1,\ldots,1$ and zeros.
In particular, the mediation product ($q=0$, $\ell=2$) has constant $10$, and its version adjusted for $q$ covariates has constant $8q+10$.

Summation over paths can cancel factors:
in the recursive linear structural model of Supplementary Proposition~\ref{prop:bcd-gformula}, the total effect of a unit intervention at coordinate $t$ on coordinate $y>t$ is $L_{yt}=M_{\{1,\ldots,t\},\{1,\ldots,t-1,y\}}/\Delta_t$, a reduced fraction with self-normalization constant $2t$.

\subsection{A catalogue of denominators}
\label{sec:catalogue}

Table~\ref{tab:catalogue} collects representative reduced denominators.
For a nonempty predictor index set $\cJ\subseteq[d]$ with $p_0=|\cJ|$ and a response coordinate $Y\in[d]\setminus\cJ$, let $R^2_{Y\cdot\cJ}=\Sigma_{Y,\cJ}\Sigma_{\cJ}^{-1}\Sigma_{\cJ,Y}/\sigma_{YY}$ be the coefficient of determination.
For $C\subseteq[d]$ with $q=|C|$ and $a,b\in[d]\setminus C$, let $\sigma_{ab\cdot C}=\sigma_{ab}-\Sigma_{a,C}\Sigma_C^{-1}\Sigma_{C,b}$ be the residual covariance, equal to $\sigma_{ab}$ when $C=\emptyset$, and for $a\ne b$ let $\rho_{ab\cdot C}=\sigma_{ab\cdot C}/(\sigma_{aa\cdot C}\sigma_{bb\cdot C})^{1/2}$ be the partial correlation given $C$.
For generalized least-squares (GLS) weights, fix a design $\cX\in\RR^{d\times p}$ of full column rank with $1\le p<d$, and put $W_{\cX}(\Sigma)=(\cX^\top\Sigma^{-1}\cX)^{-1}\cX^\top\Sigma^{-1}$.

\begin{table}[tb]
\centering
\caption{
Reduced ambient denominators.
An entry $c$ denotes self-normalization, and $\kappa_D$ is the first-order relative variance.
Here $\Delta_J=\det\Sigma_J$, $q=|C|$, and $u\in\RR^d\setminus\{0\}$ is fixed.
The notation $\Delta_r$ denotes the leading principal minor of order $r$ in the coordinate order used for the corresponding row.
Supplementary Section~\ref{supp:minors} verifies the entries.
}
\label{tab:catalogue}
\resizebox{\textwidth}{!}{%
    \begin{tabular}{@{}ccc@{}}
        \toprule
            Functional & Reduced denominator & Classification and first-order relative variance\\
        \midrule
            Regression coefficient, $p_0$ predictors
            & $\Delta_{\cJ}$ & $c=2p_0$\\
            Precision entry or $u^\top\Sigma^{-1}u$
            & $\det\Sigma$ & $c=2d$\\
            Nonconstant GLS weight
            & $\Delta_{\operatorname{Range}(\cX)^\perp}$ & $c=2(d-p)$\\
            Mediation or front-door product
            & $\Delta_1\Delta_2$ & $c=10$\\
            Adjusted two-edge product
            & $\Delta_{q+1}\Delta_{q+2}$ & $c=8q+10$\\
            Adjacent $\ell$-edge product
            & $\prod_{j=1}^{\ell}\Delta_j$ & $c=\ell(\ell+1)(2\ell+1)/3$\\
            Intervention at $t$, response $y>t$
            & $\Delta_t$ & $c=2t$\\
            Squared partial correlation given $C$
            & $\Delta_{C\cup\{a\}}\Delta_{C\cup\{b\}}$
            & \makecell{Relatively stable,\\ $\kappa_D=8q+4+4\rho_{ab\cdot C}^2$}\\
            Coefficient of determination
            & $\sigma_{YY}\Delta_{\cJ}$
            & \makecell{Relatively stable,\\ $\kappa_D=2(p_0+1)+4R^2_{Y\cdot\cJ}$}\\
            Conditional IV or scalar linear proximal formula
            & $M_{C\cup\{a\},C\cup\{b\}}$ & \makecell{Relatively unstable,\\ first-order weak sequence}\\
            Gaussian Napkin functional
            & $D_{\mathrm{nap}}$ in \eqref{eq:napkin-polynomials}
            & Positive and relatively unstable\\
        \bottomrule
    \end{tabular}
}
\end{table}

Three groups of rows deserve comment.
Precision entries have the irreducible denominator $\det\Sigma$, of higher degree than their cofactor numerators, and every nonconstant scalar GLS weight has reduced denominator $\Delta_{\operatorname{Range}(\cX)^\perp}$, so the design determines the orientation of the compression determinant (Supplementary Proposition~\ref{prop:bcd-gls}).
For subspaces $V$ and $W$, $\Sigma^{1/2}G_{\Delta_V}\Sigma^{1/2}/\Delta_V$ is the orthogonal projection onto $\Sigma^{1/2}V$, and the overlap of the two projections controls the first-order relative variance of $\Delta_V\Delta_W$ (Supplementary Lemma~\ref{lem:two-subspace}).
Hence the squared-partial-correlation denominator satisfies $n/(8q+8)\le\hat{F}_D\le n/(8q+4)$, although its numerator $M_{C\cup\{a\},C\cup\{b\}}^2$ has a repeated zero at $\rho_{ab\cdot C}=0$ and makes the ratio nonregular there.
For the conditional IV (marginal IV when $C=\emptyset$) and scalar linear proximal formulas, whose derivations and reduced nonprincipal-minor denominators are given in Supplementary Section~\ref{supp:iv-proximal-formulas}, the almost-principal denominator of multiplicity one places them in the first-order weak regime of Proposition~\ref{prop:bcd-minor-products}, and their denominator statistics are monotone transforms of one-coefficient regression $F$ statistics (Supplementary Proposition~\ref{prop:bcd-partial-minor}).

\subsection{A positive but unstable denominator}
\label{sec:napkin}

We now show that nonparametric identifiability does not guarantee a relatively stable denominator.
The Napkin graph has directed edges $W\to Z\to X\to Y$ and bidirected edges $W\leftrightarrow X$ and $W\leftrightarrow Y$, and the interventional mean of $Y$ is identified by a ratio of two $g$-formula integrals \citep{helske2021,guo2025}.
Use the coordinate order $(Z,W,X,Y)$.
For a centered Gaussian law, write $\EE(X\mid Z,W)=a_{\mathrm{nap}}Z+b_{\mathrm{nap}}W$, $s=\var(X\mid Z,W)$ and $\EE(Y\mid Z,W,X)=\theta X+\lambda Z+\eta W$.
The coefficient of the intervention value in the identifying ratio is then $T_{\mathrm{nap}}(\Sigma)=\theta+\eta b_{\mathrm{nap}}\sigma_{WW}/(s+b_{\mathrm{nap}}^2\sigma_{WW})$ (Section~\ref{supp:napkin}).
This fixes the ambient covariance formula, which agrees with the causal coefficient on the Gaussian structural model of the graph.

\begin{example}[A positive unstable denominator]
\label{ex:bcd-napkin}
    Put $\delta=\det\Sigma_{\{Z,W\}}$, $r_X=\Sigma_{(Z,W),X}$, $r_Y=\Sigma_{(Z,W),Y}$, $m=\sigma_{ZZ}\sigma_{WX}-\sigma_{ZW}\sigma_{ZX}$ and $m_Y=\sigma_{ZZ}\sigma_{WY}-\sigma_{ZW}\sigma_{ZY}$, and let $\operatorname{adj}$ denote the classical adjugate.
    The reduced fraction $T_{\mathrm{nap}}=N_{\mathrm{nap}}/D_{\mathrm{nap}}$ is
    \begin{equation}
        \label{eq:napkin-polynomials}
        \begin{split}
            D_{\mathrm{nap}}
            &=
            \delta\det\Sigma_{\{Z,W,X\}}+\sigma_{WW}m^2
            ,\\
            N_{\mathrm{nap}}
            &=
            \delta^2\sigma_{XY}
            -\delta r_X^\top\operatorname{adj}(\Sigma_{\{Z,W\}})r_Y+\sigma_{WW}m m_Y
            .
        \end{split}
    \end{equation}
    The denominator is irreducible and positive on the positive-definite cone.
    Let $U,V,R$ be independent standard normals and, for fixed $\alpha\ne0$ and $t>0$, set $Z=U$, $W=U+\sqrt{t}V$ and $X=\alpha Z+W+R$.
    Let $\Sigma_t$ be the covariance of any positive-definite Gaussian extension of this triple to $(Z,W,X,Y)$, in that order.
    Since $D_{\mathrm{nap}}$ depends only on the $(Z,W,X)$ covariance block, the following identities hold for every such extension, with the asymptotic equivalence taken as $t\downarrow0$:
    \begin{equation}
        \label{eq:napkin-path}
        D_{\mathrm{nap}}(\Sigma_t)=t^2(t+2)
        ,\quad
        \kappa_{D_{\mathrm{nap}}}(\Sigma_t)
        =
        \frac{2(13t^3+48t^2+46t+2)}{t(t+2)^2}\sim t^{-1}
        .
    \end{equation}
    The path has positive-definite extensions to the full Napkin model with a nonzero $Z\to X$ coefficient.
    Along $t_n=1/n$,
    \begin{equation}
        \label{eq:napkin-relative-limit}
        \frac{D_{\mathrm{nap}}(\Sighat)}{D_{\mathrm{nap}}(\Sigma_{t_n})}
        \dto
        \frac{1+(1+Z_0)^2}{2}
        ,\quad
        Z_0\sim N_1(0,1)
        .
    \end{equation}
\end{example}

The proof is given in Section~\ref{supp:napkin}.
Thus an identified effect can have a denominator that is positive everywhere yet relatively unstable, with a non-Gaussian relative limit driven by the residual variance $\var(W\mid Z)=t$;
by Proposition~\ref{prop:positive-boundary-degeneracy} this is higher-order instability rather than a first-order weak denominator sequence.
The same Gaussian model admits a conditional IV formula with reduced denominator $\sigma_{WW}\sigma_{ZX}-\sigma_{WZ}\sigma_{WX}=\alpha t$ on this path (Supplementary Section~\ref{supp:napkin-iv}).
At fixed $t$, this denominator admits first-order weakness as $\alpha\to0$, whereas the sample-to-population law of $D_{\mathrm{nap}}$ does not depend on $\alpha$.
Both first-order relative variances diverge as $t\to0$, and the difference belongs to the selected formulas, not to the effect.

\subsection{Mediation and front-door inference}
\label{sec:mediation}

We now derive the effect expansion and studentization that complement denominator control for mediation.
Let $\xi=(\xi_X,\xi_M,\xi_Y)^\top$ have a fixed law $Q$ with mean zero, covariance $I_3$ and finite fourth moments.
For coefficients $a,b,\gamma\in\RR$ and scales $\nu_X,\sigma_\varepsilon,v>0$, consider
\begin{equation}
    \label{eq:mediation-model}
    X
    =
    \sqrt{\nu_X}\xi_X
    ,\quad
    M
    =
    aX+\sqrt{v}\xi_M
    ,\quad
    Y
    =
    bM+\gamma X+\sigma_\varepsilon\xi_Y
    .
\end{equation}
The covariance functional
\begin{equation}
    \label{eq:mediation-functional}
    \tau_{\mathrm{med}}(\Sigma)
    =
    \frac{\sigma_{XM}(\sigma_{XX}\sigma_{MY}-\sigma_{XM}\sigma_{XY})}
    {\sigma_{XX}(\sigma_{XX}\sigma_{MM}-\sigma_{XM}^2)}
    =
    a(\Sigma)b(\Sigma)
\end{equation}
has value $ab$ at $\Sigma=\cov\{(X,M,Y)^\top\}$ under \eqref{eq:mediation-model}, where $a(\Sigma)=\sigma_{XM}/\sigma_{XX}$ and $b(\Sigma)=(\sigma_{XX}\sigma_{MY}-\sigma_{XM}\sigma_{XY})/(\sigma_{XX}\sigma_{MM}-\sigma_{XM}^2)$ are the coefficients of $X$ in the regression of $M$ on $X$ and of $M$ in the regression of $Y$ on $(X,M)$, respectively.
For $i=1,\ldots,n$, let $\xi_i$ be independent copies of $\xi$ and let $(X_i,M_i,Y_i)$ be obtained from \eqref{eq:mediation-model} with $\xi_i$ in place of $\xi$.
Under mediation or front-door structural assumptions this product has the corresponding causal interpretation \citep{pearl1995}, and in the front-door representation $\gamma$ records the treatment--outcome association induced by confounding.
The reduced denominator of $\tau_{\mathrm{med}}$ is $D_{\mathrm{med}}=\sigma_{XX}(\sigma_{XX}\sigma_{MM}-\sigma_{XM}^2)$, which equals $\nu_X^2v$ under \eqref{eq:mediation-model}, with self-normalization constant $10$ (Proposition~\ref{prop:path-products}).

For $f\in\{a,b,\tau_{\mathrm{med}}\}$, we use either of the delta-method standard errors
\begin{equation}
    \label{eq:effect-se}
    \widehat{\operatorname{se}}_{f,J}^{\,2}
    =
    \frac{1}{n}\nabla f(\Sighat)^\top\hat{\Gamma}_J\nabla f(\Sighat)
    ,\quad
    J
    \in
    \{\mathrm{G},\mathrm{emp}\}
    ,\quad
    \hat{\Gamma}_{\mathrm{G}}
    =
    \Gamma(\Sighat)
    ,
\end{equation}
where $\mathrm{G}$ denotes the Gaussian-reference version and $\mathrm{emp}$ the empirical version in \eqref{eq:empirical-scores}.
Put $\hat{a}=a(\Sighat)$, $\hat{b}=b(\Sighat)$ and $\hat{\tau}=\tau_{\mathrm{med}}(\Sighat)=\hat{a}\hat{b}$.
Write $t_{a,J}=\hat{a}/\widehat{\operatorname{se}}_{a,J}$ and $t_{b,J}=\hat{b}/\widehat{\operatorname{se}}_{b,J}$.
For the standardized sample covariance $S=n^{-1}\sum_i\xi_i\xi_i^\top=(s_{ij})$, in the order $(\xi_X,\xi_M,\xi_Y)$, define $g_a(S)=s_{12}/s_{11}$ and $g_b(S)=(s_{11}s_{23}-s_{12}s_{13})/(s_{11}s_{22}-s_{12}^2)$.
Then, exactly, $\hat{a}=a+\sqrt{v/\nu_X}\,g_a(S)$ and $\hat{b}=b+(\sigma_\varepsilon/\sqrt{v})g_b(S)$.

These identities isolate the residual variance in deterministic scale factors.
The next proposition establishes studentized normality for fixed $a\ne0$ along every bounded positive sequence of residual variances.

\begin{proposition}[Canonical reduction and studentized mediation under variance collapse]
    \label{prop:mediation-boundary-inference}
    Under \eqref{eq:mediation-model}, the following statements hold.
    \begin{enumerate}
        \item[(i)]
        For each $J\in\{\mathrm{G},\mathrm{emp}\}$ there are nonnegative functions $e_{a,n,J},e_{b,n,J}$ of the standardized sample $\xi_1,\ldots,\xi_n$ such that $\widehat{\operatorname{se}}_{a,J}=\sqrt{v/\nu_X}\,e_{a,n,J}$ and $\widehat{\operatorname{se}}_{b,J}=(\sigma_\varepsilon/\sqrt{v})e_{b,n,J}$.
        Hence, for fixed $Q$ and $n$, the joint law of $(t_{a,J},t_{b,J})$ depends on the parameters only through $\tilde{a}=a\sqrt{\nu_X/v}$ and $\tilde{b}=b\sqrt{v}/\sigma_\varepsilon$.

        \item[(ii)]
        Let $\Omega=\cov_Q\{(\xi_X\xi_M,\xi_M\xi_Y)^\top\}\succ0$, and fix $a\ne0$, $b$, $\gamma$, $\nu_X$, $\sigma_\varepsilon$ and $\bar{v}>0$.
        Along every sequence $v=v_n\in(0,\bar{v}]$, $(\hat{\tau}-ab)/\widehat{\operatorname{se}}_{\tau_{\mathrm{med}},\mathrm{emp}}\dto N_1(0,1)$.

        \item[(iii)]
        If moreover $\Omega=I_2$, for which $\xi_M\indep(\xi_X,\xi_Y)$ is sufficient, then (ii) also holds with $\widehat{\operatorname{se}}_{\tau_{\mathrm{med}},\mathrm{G}}$.
    \end{enumerate}
\end{proposition}

The proofs of (i) and of (ii)--(iii) are given in Sections~\ref{supp:canonical} and~\ref{supp:scaled-mediation}, where it is also shown that, under $\Omega=I_2$, the moment-inversion interval of Supplementary Corollary~\ref{cor:flag-fieller-enlargement} has limiting coverage $1-\alpha$ along the sequences in (ii).
The proof uses the exact identity
\begin{equation}
    \label{eq:scaled-Fv}
    \sqrt{v}(\hat{\tau}-ab)
    =
    F_v(S)
    ,\quad
    F_v(S)
    =
    a\sigma_\varepsilon g_b(S)
    +\frac{bv}{\sqrt{\nu_X}}g_a(S)+\frac{\sigma_\varepsilon\sqrt{v}}{\sqrt{\nu_X}}g_a(S)g_b(S)
    .
\end{equation}
Its derivatives with respect to $S$ are bounded uniformly for $v\in[0,\bar{v}]$ near $I_3$, with $F_0(S)=a\sigma_\varepsilon g_b(S)$, and under the assumptions of (ii) its linear score $(bv/\sqrt{\nu_X})\xi_X\xi_M+a\sigma_\varepsilon\xi_M\xi_Y$ has variance at least $\lambda_{\min}(\Omega)a^2\sigma_\varepsilon^2$, where $\lambda_{\min}$ denotes the smallest eigenvalue;
a uniform Taylor expansion and consistent empirical fourth moments complete the argument.

The canonical coordinates satisfy $\tilde{a}\tilde{b}=ab\sqrt{\nu_X}/\sigma_\varepsilon$, and for fixed $a\ne0$ the collapse of $v$ moves toward $|\tilde{a}|\to\infty$ and $\tilde{b}\to0$.
Corollary~\ref{cor:denominator-removal} controls the relatively stable denominator, and the proposition supplies the expansion of the effect and its standard error.
In particular, under $\Omega=I_2$ and $v_n\to0$, $\widehat{\operatorname{se}}_{\tau_{\mathrm{med}},J}=|a|\sigma_\varepsilon(nv_n)^{-1/2}\{1+o_p(1)\}$ for $J\in\{\mathrm{G},\mathrm{emp}\}$.
The denominator statistic remains $n/10$ while this precision changes with the rate of $nv_n$.

\begin{remark}[Local coefficient regimes]
    \label{rem:mediation-local-regimes}
    Assume $\xi_M\indep(\xi_X,\xi_Y)$ and fix $\nu_X,\sigma_\varepsilon>0$ and $\gamma\in\RR$.
    Consider \eqref{eq:mediation-model} with $a=a_n$, $b=b_n$ and $v=v_n>0$, where $a_n\sqrt{n\nu_X/v_n}\to a_0\in\RR$ and $b_n\sqrt{nv_n}/\sigma_\varepsilon\to b_0\in\RR$.
    Then $n\sqrt{\nu_X}(\hat{\tau}-a_nb_n)/\sigma_\varepsilon\dto(a_0+Z_1)(b_0+Z_2)-a_0b_0$, with independent standard normals $Z_1,Z_2$.
    For either standard-error version $J\in\{\mathrm{G},\mathrm{emp}\}$, the signed Wald statistic $(\hat{\tau}-a_nb_n)/\widehat{\operatorname{se}}_{\tau_{\mathrm{med}},J}$ converges to this limit divided by $\{(a_0+Z_1)^2+(b_0+Z_2)^2\}^{1/2}$.
    At $a_0=b_0=0$, the resulting Wald limit has law $N_1(0,1/4)$.
    The joint limits, including the empirical studentizer, are given in Supplementary Proposition~\ref{prop:bcd-local-product}.
    This is the singular-hypothesis phenomenon for coefficient products \citep{drtonxiao2016}, and it occurs despite relative stability of the denominator.
\end{remark}

Empirical centering and intercepts preserve these asymptotic conclusions.
Supplementary Section~\ref{supp:mediation-centering} gives the centered score covariances and regression normalization.

\section{Numerical illustrations}
\label{sec:numerics}

Three experiments illustrate the theory: covariance-free calibration of a self-normalizing denominator against a relatively stable one, inference for the mediated effect under collapse of the mediator's residual variance, and two non-Gaussian error laws that separate miscalibration of the denominator law from miscalibration of the effect variance.
Each design cell uses $B=50000$ independent samples and the known-mean sample covariance \eqref{eq:sample-covariance}.
All intervals have nominal level $0.95$, and coverage is computed over all replications.
For a coverage $\hat{p}$, the Monte Carlo standard error (MCSE) is $\{\hat{p}(1-\hat{p})/B\}^{1/2}$, approximately $0.0010$ at $\hat{p}=0.95$, and is reported alongside the results \citep{morris2019simulation}.
Supplementary Section~\ref{supp:numerics} gives the design details, reference-quantile checks and additional numerical results.

\subsection{Covariance-free and transported calibration}
\label{sec:calibration}

The first experiment contrasts calibration by a single reference law for a self-normalizing denominator with calibration for a relatively stable one.
For an integer $r\ge1$, let $d=2r$, $I_1=\{1,\ldots,r\}$ and $I_2=\{r+1,\ldots,2r\}$.
For $0\le\rho<1$, generate $X=(U^\top,\{\rho U+\sqrt{1-\rho^2}V\}^\top)^\top$ with independent $U,V\sim N_r(0,I_r)$, and put $\Sigma=\cov(X)$.
All $r$ canonical correlations between the two blocks then equal $\rho$.
We compare $D_{\mathrm{SN}}=\det\Sigma$ with the nonnested product $D_{\mathrm{st}}=\det\Sigma_{I_1}\det\Sigma_{I_2}$.
Both are positive, relatively stable polynomials of degree $2r$, with $\kappa_{D_{\mathrm{SN}}}=4r$ and $\kappa_{D_{\mathrm{st}}}=4r(1+\rho^2)$.
The grid is $r\in\{1,3\}$, $n\in\{50,200,1000\}$ and $\rho\in\{0,0.5,0.9,0.99\}$.
For each polynomial, the quantiles of its sample-to-population ratio at $\rho=0$ are used at every $\rho$;
we call this transported-reference calibration, which is exact for $D_{\mathrm{SN}}$ by Theorem~\ref{thm:sn-pivot-classification} and exact for $D_{\mathrm{st}}$ only at $\rho=0$.

Write $R_{n,D}=D(\Sighat)/D(\Sigma)$ and let $q^{(0)}_{n,D,p}$ be the $p$ quantile of $\log R_{n,D}$ at $\rho=0$.
Transported-reference coverage is $\Pr\{q^{(0)}_{n,D,0.025}\le\log R_{n,D}\le q^{(0)}_{n,D,0.975}\}$, and Figure~\ref{fig:calibration} (a) displays the quantiles of $\sqrt{n}\log R_{n,D}$ across $\rho$ to show which law depends on the covariance.
For $D_{\mathrm{st}}$ we also report the plug-in interval with endpoints $D_{\mathrm{st}}(\Sighat)\exp\{\pm z_{0.975}\sqrt{\kappa_{D_{\mathrm{st}}}(\Sighat)/n}\}$, which adapts to the covariance through $\kappa_{D_{\mathrm{st}}}(\Sighat)$ and is asymptotically valid uniformly in $\rho$ by Corollary~\ref{cor:uniform-normal}, since $\kappa_{D_{\mathrm{st}}}\ge4r$.

For $D_{\mathrm{SN}}$ the transported-reference coverage is exactly $0.95$, reproduced up to Monte Carlo error at every $\rho$.
For $D_{\mathrm{st}}$ it converges to
\begin{equation}
\label{eq:transport-limit}
    2\Phi\Biggl\{
        \frac{z_{0.975}}{\sqrt{1+\rho^2}}
    \Biggr\}-1
    ,
\end{equation}
which equals $0.8363$ at $\rho=0.99$ for both block sizes.
At $r=3$, $n=1000$ and $\rho=0.99$, the simulated coverages are $0.9490$ for $D_{\mathrm{SN}}$ and $0.8360$ for $D_{\mathrm{st}}$, while the plug-in interval for $D_{\mathrm{st}}$ attains $0.9492$.
Thus relative stability supports covariance-adapted calibration, and self-normalization additionally supports calibration by one reference law for all covariances.

\begin{figure}[tb]
\centering
\includegraphics[width=.485\linewidth]{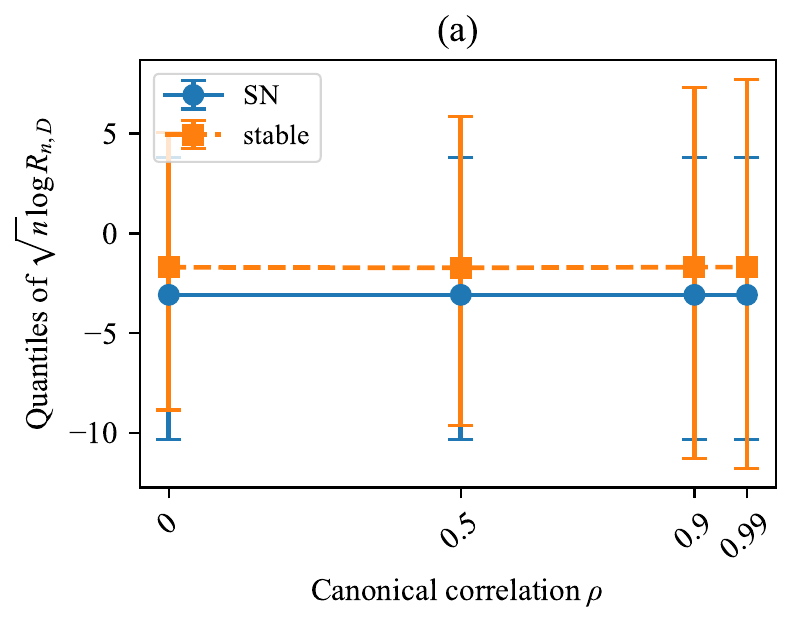}\hfill
\includegraphics[width=.485\linewidth]{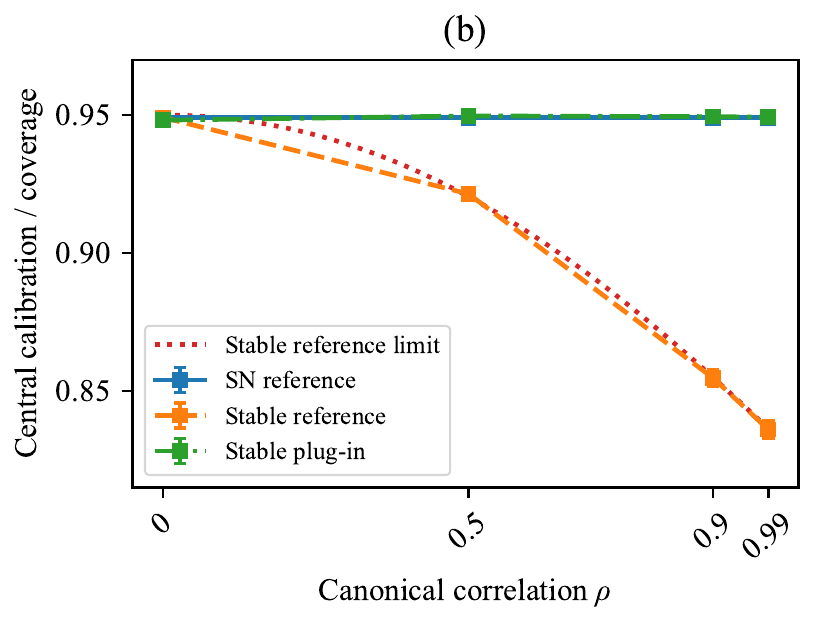}
\caption{
Calibration for $r=3$.
(a) Markers are empirical medians of $\sqrt{n}\log R_{n,D}$ at $n=50$, and whiskers extend to its $0.025$ and $0.975$ quantiles.
(b) Transported-reference and covariance-adapted coverage at $n=1000$, with limit \eqref{eq:transport-limit}.
Error bars in (b) show $\pm1.96$ MCSEs.
Common random numbers make the $D_{\mathrm{SN}}$ summaries coincide across $\rho$.
}
\label{fig:calibration}
\end{figure}

Supplementary Table~\ref{tab:corrections} compares the uncorrected log-delta interval with a normal interval using exact logarithmic moments.
Supplementary Sections~\ref{supp:numerical-adjusted} and~\ref{supp:numerical-precision} treat covariate-adjusted products and a precision reciprocal, respectively.

\subsection{Denominator protection and mediated-effect inference}
\label{sec:illustration-mediation}

The second experiment varies the mediator's residual variance while keeping the mediation coefficients fixed.
The Gaussian design uses \eqref{eq:mediation-model} with $\nu_X=1$, $a=0.8$, $b=1$, $\gamma=0.5$ and $\sigma_\varepsilon^2=0.75$, and varies $n\in\{100,300,1000\}$ and $v\in\{1,n^{-1/2},n^{-1},n^{-2}\}$.
By Theorem~\ref{thm:sn-pivot-classification}, the sample-to-population law of $D_{\mathrm{med}}$ is the same for every $v$, and $\hat{F}_D=n/10$ in every sample.
Figure~\ref{fig:mediation} (a) plots the coverage of the Wald interval with the empirical standard error $\widehat{\operatorname{se}}_{\tau_{\mathrm{med}},\mathrm{emp}}$ against the median of that standard error.
At $n=1000$, coverage ranges from $0.9493$ to $0.9500$, while the median standard error increases from $0.0384$ at $v=1$ to $21.8537$ at $v=n^{-2}$;
at $n=100$, coverage is $0.9373$--$0.9381$.
The standard errors reflect the loss of precision under residual-variance collapse.
Proposition~\ref{prop:mediation-boundary-inference} (ii) gives a standard normal studentized limit along each of these variance sequences.

The second part of the experiment maps the local coefficient regimes of Remark~\ref{rem:mediation-local-regimes} in the canonical coordinates $h_a=\sqrt{n}\tilde{a}=a\sqrt{n\nu_X/v}$ and $h_b=\sqrt{n}\tilde{b}=b\sqrt{nv}/\sigma_\varepsilon$.
By Proposition~\ref{prop:mediation-boundary-inference} (i), every $(h_a,h_b)$ can be evaluated from the same standardized observations.
Figure~\ref{fig:mediation} (b) shows the coverage of the Wald interval with empirical standard error at $n=300$ on a prespecified nonnegative $24\times24$ grid.
At $(0,0)$ the observed coverage is $0.9998$, against the limit $0.9999$ from the $N_1(0,1/4)$ law, and at $(2,2)$ it is $0.9059$, against $0.9091$ from the local product limit.
This intermediate local regime persists although the denominator law is unchanged, and the overlaid fixed-coefficient collapse path, on which $h_ah_b=nab\sqrt{\nu_X}/\sigma_\varepsilon$, passes through the well-calibrated region of the grid.
Supplementary Section~\ref{supp:illustration-canonical} reports the Gaussian-reference studentization, the local-limit calculations and the moment-inversion results.

\begin{figure}[tb]
\centering
\includegraphics[width=.485\linewidth]{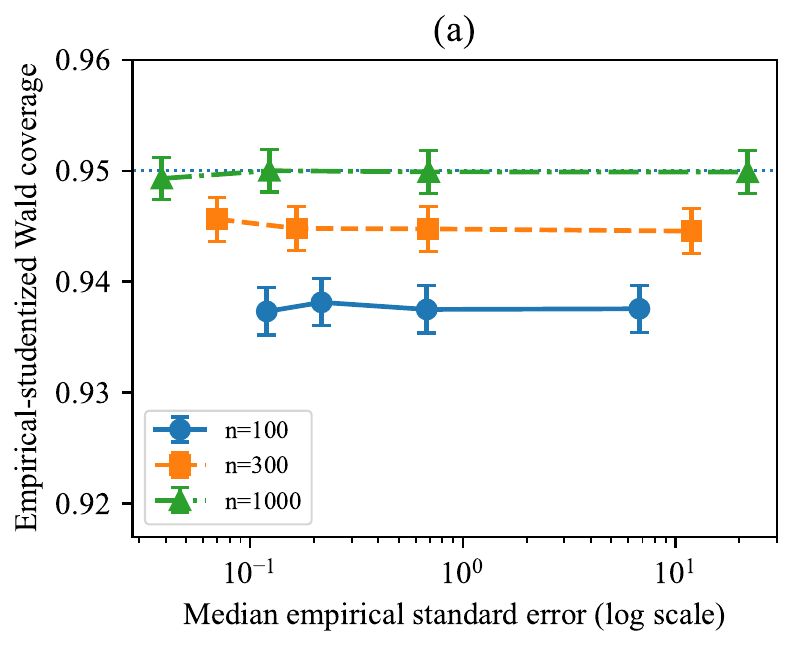}\hfill
\includegraphics[width=.485\linewidth]{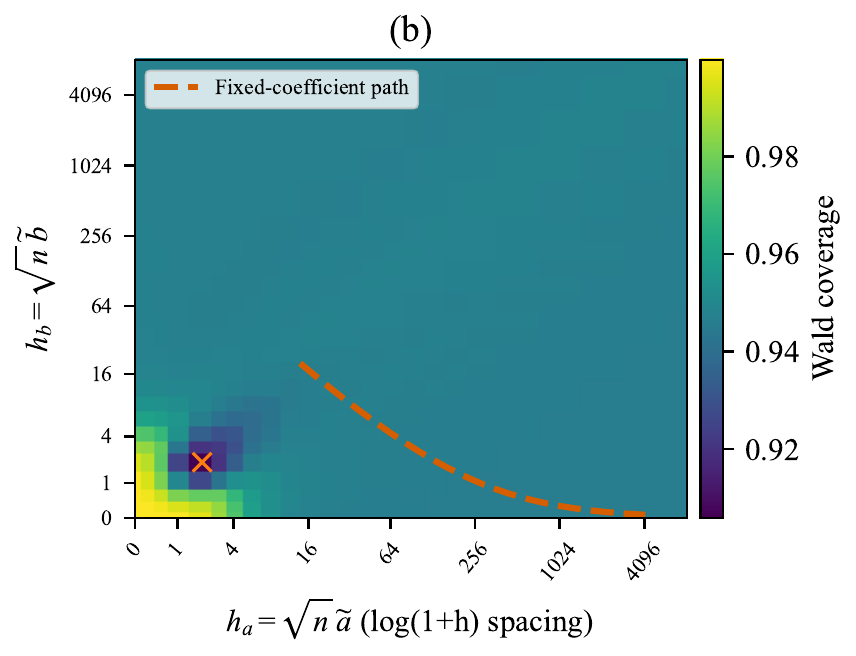}
\caption{
Precision and local coefficients in mediation.
(a) Coverage using empirical standard errors, with $\pm1.96$ Monte Carlo standard-error bars.
Points follow $v=1,n^{-1/2},n^{-1},n^{-2}$ with fixed $a=0.8$.
(b) Empirical-Wald coverage at $n=300$, with axis spacings given by $\log(1+h_a)$ and $\log(1+h_b)$.
The cross marks $(2,2)$.
The dashed curve follows fixed coefficients from $v=1$ to $n^{-2}$.
Each point uses $50000$ replications.
}
\label{fig:mediation}
\end{figure}

\subsection{Sampling laws: two distinct calibration failures}
\label{sec:sampling-laws}

The third experiment changes the error law in the mediation design of Section~\ref{sec:illustration-mediation} to separate two failures, a denominator law that differs from the Gaussian reference and an effect variance that differs from the Gaussian-reference formula.
We compare Gaussian error terms (G) with two fixed standardized laws for $(\xi_X,\xi_M,\xi_Y)$ in \eqref{eq:mediation-model}.
In (E), the error terms are independent $\operatorname{Exp}(1)-1$ variables, where $\operatorname{Exp}(1)$ is exponential with rate one, each with fourth moment $9$.
In (D), independent standard normals $U,V,R$ give $(\xi_X,\xi_M,\xi_Y)=(U,V,VR)$.
All three designs target \eqref{eq:mediation-functional}.

Under (E), the actual first-order relative variance of $D_{\mathrm{med}}$ is $40$, compared with the Gaussian-reference value $10$, so the denominator interval based on the chi-squared product law of Theorem~\ref{thm:sn-pivot-classification} has limiting coverage $0.6729$.
The effect-score covariance is $\Omega=I_2$, so both Wald intervals for the effect are asymptotically valid by Proposition~\ref{prop:mediation-boundary-inference} (ii)--(iii).
Under (D), $D_{\mathrm{med}}$ depends only on the Gaussian pair $(X,M)$ and retains the Gaussian law, but $\Omega=\diag(1,3)$.
The ratio of the actual to the Gaussian-reference effect variance is $\lambda(v)=(v^2+1.44)/(v^2+0.48)$, so the Gaussian-reference Wald interval has limiting coverage $0.8731$ at $v=1$ and $0.7422$ as $v_n\to0$.

Supplementary Tables~\ref{tab:mediation} and~\ref{tab:denominator} confirm that the two failures are distinct.
Within each design the sample-to-population ratio of $D_{\mathrm{med}}$ has the same law for every $v$ (Supplementary Section~\ref{supp:rotational-sampling}), and only under (E) does that law differ from the Gaussian reference, whereas only under (D) does the Gaussian-reference effect variance fail.
The empirical standard error corrects the latter at the finite-sample cost of estimating fourth moments;
under (E) with $v=n^{-1}$, its coverage increases from $0.9362$ at $n=300$ to $0.9438$ at $n=1000$.
The weak-covariance and Napkin comparisons are given in Supplementary Sections~\ref{supp:numerical-weak} and~\ref{supp:numerical-napkin}, respectively.

\section{Discussion}
\label{sec:discussion}

Under Gaussian sampling, constant first-order relative variance characterizes flag powers and forces a covariance-independent sample-to-population law.
Relative stability characterizes uniform control of scaled relative errors, and the moment $N-\tau D$ then determines the precision and limiting distribution of the ratio.
These results separate exact calibration from denominator control.

The results concern fixed ambient formulas evaluated at the unrestricted sample covariance, in fixed dimension.
Submodels with empty interior can carry additional pivots (Remark~\ref{rem:local-pivot}), and under a law with isotropic fourth moments constancy of the actual first-order relative variance characterizes flag powers among homogeneous polynomials but not among all polynomials (Supplementary Section~\ref{supp:rotational-sampling}), so self-normalization is a Gaussian-reference notion, whereas relative stability transfers to any fixed standardized law with finite fourth moments, with a converse under nondegeneracy of the quadratic-score covariance (Proposition~\ref{prop:fixed-law-stability}).
Data-dependent choice among formulas introduces an additional selection stage into this fixed-formula framework.

A further question concerns rational functionals with nowhere-zero reduced numerator and denominator on the positive-definite cone.
Does constant positive first-order relative variance force a nonzero constant times a product of signed integer powers of nested compression determinants?
Such products have relative pivots by Bartlett decomposition.
Theorem~\ref{thm:sn-pivot-classification} and Corollary~\ref{cor:bcd-precision} cover covariance polynomials and the stated precision-polynomial class.
The latter includes the bivariate residual variance $\det\Sigma/\sigma_{11}=1/(\Sigma^{-1})_{22}$.
Its covariance-gradient kernel nevertheless varies with $\Sigma$, showing why the polynomial fixed-kernel argument requires modification for general rational functions (Supplementary Section~\ref{supp:precision}).

A second question is a structural description of the relatively stable class.
Theorem~\ref{thm:uniform-denominator} supplies analytic and normalized-orbit criteria, and Proposition~\ref{prop:bcd-minor-products} gives a complete test within products of covariance minors, but positive sums of products of compression determinants do not exhaust the class (Section~\ref{sec:stability-criteria}).
Turning the general criteria into practical symbolic certificates would extend denominator assessment to a broader range of identification algorithms.

The Napkin example motivates graphical criteria for the denominator structure of a selected identifying formula, distinguishing identification from relative stability.
Its non-Gaussian relative limit also motivates joint local analysis of numerator and denominator near boundary zeros of positive denominators.
Such an analysis would clarify how their shared leading terms affect cancellation and the resulting ratio experiment.

\section*{Code availability}
The Python scripts reproducing the numerical illustrations in Section~\ref{sec:numerics} are available at \url{https://github.com/shutech2001/self-normalizing-denominators-experiments}.

\section*{Acknowledgments}
Shu Tamano was supported by JSPS KAKENHI Grant Number 25K24203.

\section*{Supplementary material}
\label{SM}
The Supplementary Material contains complete proofs, the unknown-mean and fixed-law extensions, the calculations for the examples, and the design, reference calculations and additional numerical results.

\bibliographystyle{apalike}
\bibliography{bibliography}

\clearpage
\setcounter{section}{0}
\renewcommand{\thesection}{\Alph{section}}
\makeatletter
\@addtoreset{equation}{section}
\@addtoreset{table}{section}
\@addtoreset{figure}{section}
\setcounter{equation}{0}
\setcounter{table}{0}
\setcounter{figure}{0}
\renewcommand{\theequation}{\thesection.\arabic{equation}}
\renewcommand{\thetable}{\thesection.\arabic{table}}
\renewcommand{\thefigure}{\thesection.\arabic{figure}}
\def\supp@numbercounter#1{%
    \@ifundefined{c@#1}{}{%
        \@addtoreset{#1}{section}%
        \setcounter{#1}{0}%
        \expandafter\gdef\csname the#1\endcsname{\thesection.\arabic{#1}}%
    }%
}
\supp@numbercounter{theorem}
\supp@numbercounter{lemma}
\supp@numbercounter{proposition}
\supp@numbercounter{corollary}
\supp@numbercounter{definition}
\supp@numbercounter{remark}
\supp@numbercounter{example}
\supp@numbercounter{assumption}
\let\supp@numbercounter\relax
\makeatother

\begin{center}
{\Large\bfseries Supplementary Material for}\par
\medskip
{\Large\bfseries ``Self-normalizing denominators in rational covariance estimators''}
\end{center}

\section{Proofs for the statistical formulation and the classification theorem}
\label{supp:classification}

This section proves the identities of Section~\ref{sec:prelim} and the deterministic implication (a)$\Rightarrow$(c) of Theorem~\ref{thm:sn-pivot-classification}.
Section~\ref{supp:formulation} collects conventions and proves the Gaussian variance and information identities, Section~\ref{supp:nonvanishing-flags} proves nonvanishing and the sufficiency direction, Sections~\ref{supp:integral-spectrum} and~\ref{supp:fixed-kernel} contain the two structural steps, Section~\ref{supp:determinant} treats determinant factors and the structure of flag powers, and Section~\ref{supp:classification-necessity} completes the induction.

\subsection{Conventions and Gaussian variance identities}
\label{supp:formulation}
Throughout the Supplementary Material, $d$ is fixed and $p_d=d(d+1)/2$.
Symmetric gradients use the trace inner product and the coordinate order in Section~\ref{sec:notation}.
The matrix $\Pi_V$ is the Euclidean orthogonal projection onto a subspace $V$, $\sgn(x)$ is the sign of a real number $x$, and the Loewner order is $H\preceq K$ when $K-H\succeq0$.
Polynomial rings here are unique factorization domains, so irreducible elements are prime, and equality of polynomials on a nonempty Euclidean-open set implies polynomial equality.
A polynomial over a unique factorization domain is called primitive when its coefficients have no common noninvertible factor.
In polynomial-ring notation, a matrix or vector in the list of variables stands for its independent entries.
For $\Sigma\succ0$ and $H,K\in\Sym_d$, the affine-invariant metric is $g^{\mathrm{AI}}_\Sigma(H,K)=\tr(\Sigma^{-1}H\Sigma^{-1}K)=2g^{\mathrm{FR}}_\Sigma(H,K)$.
It has the same geodesics as the Fisher--Rao metric in \eqref{eq:fisher-metric}.
Tangent directions to the covariance cone are identified with $\Sym_d$, $\|H\|_{\mathrm{FR},\Sigma}=\{g^{\mathrm{FR}}_\Sigma(H,H)\}^{1/2}$, and for a differentiable scalar function $f$ the Fisher--Rao gradient $\operatorname{grad}_{\mathrm{FR}}f(\Sigma)\in\Sym_d$ is the unique matrix satisfying $g^{\mathrm{FR}}_\Sigma(\operatorname{grad}_{\mathrm{FR}}f(\Sigma),H)=\dd f_\Sigma(H)$ for every $H\in\Sym_d$, so that $\operatorname{grad}_{\mathrm{FR}}f(\Sigma)=2\Sigma G_f(\Sigma)\Sigma$.
The Hessian operators and Newton transformations are defined where used.

The following conventions supplement Section~\ref{sec:notation}.
We allow $I_0$ as an empty identity block, and $\#I$ also denotes the cardinality of a finite set $I$.
The index range $a:b$ denotes $a,a+1,\ldots,b$ and is empty when $b<a$;
a dot in an index denotes all indices in that position, as in $A_{\cdot,j}$;
for a vector $x$, $x_I$ and $x_{1:r}$ denote the corresponding subvectors;
empty sums and products are zero and one;
$\deg0=-\infty$;
and $P\nmid Q$ means that $P$ does not divide $Q$.
We write $\Sighat=(\hat{\sigma}_{ij})$.
Directional derivatives of vector- or matrix-valued functions are taken componentwise, with the notation $\dd$ of \eqref{eq:gradient-convention}.
For deterministic quantities, $f=O(g)$ means that $|f|/|g|$ is bounded in the stated limit.
Standard deviations and standard errors are the nonnegative square roots of the corresponding variances.
The symbol $\chi^2_\nu$ is also used for real $\nu>0$, and $\chi_\nu$ denotes the positive square root of a $\chi^2_\nu$ variable.

The first lemma collects the Gaussian variance identity of Section~\ref{sec:relative-variance}, its information-theoretic meaning and its transformation rules.

\begin{lemma}[Trace form, information metric, congruence and compression]
    \label{lem:trace-main}
    Let $\Sigma\succ0$, and let $P$ be continuously differentiable near $\Sigma$.
    For $X\sim N_d(0,\Sigma)$, the first-order influence function is
    \begin{equation}
        \label{eq:S1-linear-score}
        \phi_P(X;\Sigma)
        :=
        \dd P_\Sigma(XX^\top-\Sigma)
        =
        X^\top G_P(\Sigma)X-\tr\{G_P(\Sigma)\Sigma\}
        .
    \end{equation}
    It has variance $s_P^2(\Sigma)=2\tr\{(G_P(\Sigma)\Sigma)^2\}$, equal to the coordinate quadratic form in \eqref{eq:sD}, and $s_P^2(\Sigma)$ is the limiting variance of $\sqrt{n}\{P(\Sighat)-P(\Sigma)\}$.
    The same influence-function notation is used for any differentiable scalar covariance functional, with arguments suppressed when clear.
    If $P(\Sigma)\ne0$ and $f=\log|P|$, then $\operatorname{grad}_{\mathrm{FR}}f(\Sigma)=2\Sigma G_P(\Sigma)\Sigma/P(\Sigma)$ and $\|\operatorname{grad}_{\mathrm{FR}}f(\Sigma)\|_{\mathrm{FR},\Sigma}^2=\kappa_P(\Sigma)$.
    Moreover, every estimator of $f(\Sigma)$ from $n$ observations that is unbiased in a neighborhood of $\Sigma$ and has finite variance has variance at least $\kappa_P(\Sigma)/n$.
    For an invertible $A\in\RR^{d\times d}$, define $P_A(S)=P(ASA^\top)$ wherever the argument lies in the domain of $P$.
    At every $S\succ0$ for which $P$ is differentiable near $ASA^\top$, $G_{P_A}(S)=A^\top G_P(ASA^\top)A$ and $s_{P_A}^2(S)=s_P^2(ASA^\top)$.
    For $1\le r\le d$, let $M\in\RR^{r\times d}$ have full row rank, and let $\bar{P}$ be a continuously differentiable scalar function on an open subset of $\Sym_r$.
    If $P(S)=\bar{P}(MSM^\top)$ on the corresponding open domain in $\Sym_d$, then $G_P(S)=M^\top G_{\bar{P}}(MSM^\top)M$ and $s_P^2(S)=s_{\bar{P}}^2(MSM^\top)$ at its positive-definite points.
    Here the gradient and variance of $\bar{P}$ use its $r$-dimensional argument.
    Division by the squared function values gives the first-order relative-variance identities, and for polynomials the variance identities extend to every symmetric argument.
\end{lemma}

\begin{proof}[Proof of Lemma~\ref{lem:trace-main}]
    Write $G=G_P(\Sigma)$.
    The linear perturbation of $P$ evaluated at $XX^\top-\Sigma$ is exactly \eqref{eq:S1-linear-score}, because $\nabla P^\top\vech(H)=\tr(GH)$ with the stated off-diagonal weights.
    The Gaussian fourth-moment formula gives, for all $H,K\in\Sym_d$,
    \begin{equation}
    \label{eq:S1-quadratic-cov}
        \cov(X^\top HX,X^\top KX)
        =
        2\tr(H\Sigma K\Sigma)
        .
    \end{equation}
    Taking $H=K=G$ proves the variance and coordinate formulas;
    equivalently, diagonalize $\Sigma^{1/2}G\Sigma^{1/2}$ and use independent standard normals with $\var(Z_i^2)=2$.
    The multivariate central limit theorem applies to $\vech(X_iX_i^\top-\Sigma)$, which has finite second moments, and a first-order Taylor expansion gives the limiting variance $s_P^2(\Sigma)$;
    the same expansion with gradient $G_P/P$ gives the limiting variances $\kappa_P(\Sigma)$ of $\sqrt{n}\{P(\Sighat)/P(\Sigma)-1\}$ and of $\sqrt{n}\{\log|P(\Sighat)|-\log|P(\Sigma)|\}$ stated in Section~\ref{sec:relative-variance}.
    
    For the information bound, the directional likelihood score of one observation $X$ is $\dot{\ell}_\Sigma(H;X)=2^{-1}\{X^\top\Sigma^{-1}H\Sigma^{-1}X-\tr(\Sigma^{-1}H)\}$.
    Equation~\eqref{eq:S1-quadratic-cov} gives $\EE\{\dot{\ell}_\Sigma(H;X)\dot{\ell}_\Sigma(K;X)\}=g^{\mathrm{FR}}_\Sigma(H,K)$.
    The displayed gradient is the unique matrix satisfying $g^{\mathrm{FR}}_\Sigma(\operatorname{grad}_{\mathrm{FR}}f(\Sigma),H)=\dd f_\Sigma(H)$ for every $H\in\Sym_d$.
    Its squared norm at $\Sigma$ is $2\tr\{(G_P(\Sigma)\Sigma/P(\Sigma))^2\}=\kappa_P(\Sigma)$, and $\dot{\ell}_\Sigma(\operatorname{grad}_{\mathrm{FR}}f(\Sigma);X)=\phi_P(X;\Sigma)/P(\Sigma)$ is the influence function of $\log|P(\Sighat)|$.
    Let $T_n$ have finite variance and $\EE_{\Sigma'}(T_n)=f(\Sigma')$ for all $\Sigma'$ in a neighborhood of $\Sigma$;
    differentiation under the integral, justified by the Gaussian likelihood, gives $\dd\{\EE_\Sigma(T_n)\}(H)=\dd f_\Sigma(H)$ for every $H$.
    Put $V=\operatorname{grad}_{\mathrm{FR}}f$ at $\Sigma$.
    Then $\cov_\Sigma\{T_n,\sum_{i=1}^n\dot{\ell}_\Sigma(V;X_i)\}=\dd f_\Sigma(V)=\kappa_P$ and $\var_\Sigma\{\sum_{i=1}^n\dot{\ell}_\Sigma(V;X_i)\}=n\kappa_P$, so Cauchy--Schwarz gives $\var_\Sigma(T_n)\ge\kappa_P/n$ when $\kappa_P>0$;
    the zero case is immediate.
    
    For congruence, apply the chain rule to $\dd P_{ASA^\top}(AHA^\top)$;
    cyclicity of trace gives the variance identity, since $G_{P_A}(S)S$ is similar to $G_P(ASA^\top)ASA^\top$.
    For compression, the chain rule gives $G_P=M^\top G_{\bar{P}}M$ and $\tr\{(M^\top G_{\bar{P}}MS)^2\}=\tr\{(G_{\bar{P}}MSM^\top)^2\}$.
    These prove the assertions on positive-definite arguments, and polynomial continuation proves the ambient identities for polynomials.
\end{proof}

\subsubsection{Reduced representations and sample domains}
\label{supp:sample-domains}

\begin{proof}[Proof of the representation and sample-domain assertions in Section~\ref{sec:rational-estimators}]
    For coprime pairs $N,D$ and $\tilde{N},\tilde{D}$ representing the same nonzero rational function, $N\tilde{D}=\tilde{N} D$ implies $D\mid\tilde{D}$ and $\tilde{D}\mid D$.
    Thus $\tilde{D}=bD$ and $\tilde{N}=bN$ for a nonzero constant $b$.
    For the zero function, coprimality forces a constant denominator.
    This proves the uniqueness used in Definition~\ref{def:strategy}.
    
    A nonzero polynomial in finitely many real variables has a Lebesgue-null zero set:
    induction on the number of variables follows by viewing it as a univariate polynomial with polynomial coefficients and applying Fubini outside the common zero set of its coefficients.
    For $n\ge d$ and a nonzero $D\in\RR[\Sym_d]$, the polynomial $(x_1,\ldots,x_n)\mapsto D(n^{-1}\sum_{i=1}^nx_ix_i^\top)$ in the entries of $x_1,\ldots,x_n\in\RR^d$ is nonzero.
    Indeed, any $S\succ0$ with $D(S)\ne0$ can be attained by taking the first $d$ columns of the data matrix to be $\sqrt{n} S^{1/2}$ and the remaining columns zero.
    The joint Gaussian data law has a density, so $D(\Sighat)=0$ has probability zero.
    The same argument applied to a square data minor gives positive definiteness of the sample covariance for $n\ge d$, and to a projected sample gives the $n\ge r$ block assertion.
    For a nonconstant $D$, $G_D$ is nonzero on a nonempty part of the positive cone, and \eqref{eq:sD} is positive there.
    Hence $s_D^2$ is a nonzero polynomial and its sample zero set is also null for $n\ge d$.
\end{proof}

\subsection{Nonvanishing and sufficiency of flag powers}
\label{supp:nonvanishing-flags}

This subsection shows that bounded first-order relative variance forces nonvanishing on the cone, and proves the sufficiency direction (c)$\Rightarrow$(a) of Theorem~\ref{thm:sn-pivot-classification}.

\begin{lemma}[Nonvanishing under bounded first-order relative variance]
\label{lem:bcd-nozero}
    Let $P\in\RR[\Sym_d]$ be nonzero and suppose that $\kappa_P(\Sigma)\le C<\infty$ whenever $\Sigma\succ0$ and $P(\Sigma)\ne0$.
    Then $P$ is nowhere zero on the positive-definite cone and has constant sign there.
\end{lemma}

\begin{proof}[Proof of Lemma~\ref{lem:bcd-nozero}]
    Choose $A\succ0$ with $P(A)\ne0$, possible since the positive cone is open and $P$ is a nonzero polynomial.
    On each nonzero component put $f=\log|P|$.
    For a differentiable positive-definite path $S(t)$, the Cauchy--Schwarz inequality for the Frobenius inner product gives
    \begin{equation}
        \label{eq:S1-log-path}
        |(f\circ S)'(t)|
        \le
        \sqrt{C/2}\,
        \|S(t)^{-1/2}\dot S(t)S(t)^{-1/2}\|_F
        .
    \end{equation}
    If $P(B)=0$ at some $B\succ0$, the segment $S(t)=(1-t)A+tB$, $0\le t\le1$, lies in a compact subset of the positive cone.
    Let $t_{\ast}>0$ be its first zero of $P$.
    The right side of \eqref{eq:S1-log-path} is bounded and hence integrable up to $t_{\ast}$.
    Integration makes $f(S(t))$ bounded below for $t<t_{\ast}$.
    Continuity gives $P(S(t))\to0$ as $t\uparrow t_{\ast}$, contradicting that lower bound.
    Connectedness of the cone then fixes the sign.
\end{proof}

The next lemma is the sufficiency direction of Theorem~\ref{thm:sn-pivot-classification}.

\begin{lemma}[Relative gradient of a flag power]
\label{lem:flag-relative-gradient}
    A flag power with multiplicities $m_i$ satisfies \eqref{eq:SN} with $c=2\sum_i m_i^2$.
    In orthonormal coordinates adapted to the flag, its relative gradient is upper triangular with diagonal $(m_1,\ldots,m_d)$.
    Its degree is $\sum_i m_i$ and it has the sign of $\gamma$ on the positive-definite cone.
\end{lemma}

\begin{proof}[Proof of Lemma~\ref{lem:flag-relative-gradient}]
    Let $E_r$ be the $d\times r$ matrix of the first $r$ coordinate vectors.
    Differentiating $\Delta_r(S)=\det S_{1:r}$ at $S\succ0$ gives $G_{\Delta_r}(S)=\Delta_r(S)E_rS_{1:r}^{-1}E_r^\top$, and multiplication by $S/\Delta_r(S)$ gives \eqref{eq:flag-gradient}.
    The product rule for logarithmic derivatives sums these matrices with their factor exponents.
    The diagonal entry in position $i$ is $m_i$.
    The trace of the square is $\sum_i m_i^2$, giving the identity on the cone and everywhere by polynomial continuation.
    Orthogonal congruence treats an arbitrary flag by Lemma~\ref{lem:trace-main}.
    Degrees add to $\sum_j a_j r_j=\sum_i m_i$, and all compression determinants are positive on that cone.
\end{proof}

\subsection{Calibrated flats and an integral spectrum}
\label{supp:integral-spectrum}

This subsection carries out the first structural step of the necessity proof.
A Lipschitz function that attains its Lipschitz bound along an entire line is affine, and on a flat of commuting matrices this rigidity forces an integral spectrum.

\begin{lemma}[A calibrated line determines a Lipschitz function]
\label{lem:calibrated-flat}
    Let $h:\RR^d\to\RR$ be globally $L$-Lipschitz with $L>0$.
    If there exists $v\in\RR^d$ with $\|v\|=L$ and $h(tv)=h(0)+tL^2$ for every real $t$, then $h(u)=h(0)+v^\top u$ for every $u\in\RR^d$.
\end{lemma}

\begin{proof}[Proof of Lemma~\ref{lem:calibrated-flat}]
    For every real $t$,
    \begin{equation*}
        h(0)+tL^2-L\|u-tv\|\le h(u)
        \le
        h(0)+tL^2+L\|u-tv\|
        .
    \end{equation*}
    For fixed $u$, expansion of the squared norm gives $\|u-tv\|=|t|L-\sgn(t)v^\top u/L+O(|t|^{-1})$.
    Let $t\to+\infty$ in the lower bound and $t\to-\infty$ in the upper bound.
    Both approach $h(0)+v^\top u$.
\end{proof}

\begin{proposition}[Integral spectrum and automatic homogeneity]
\label{prop:integral-spectrum}
    Suppose that a nonzero polynomial $D\in\RR[\Sym_d]$ satisfies \eqref{eq:SN} with $c>0$.
    Then there are fixed integers $m_1\ge\cdots\ge m_d\ge0$ such that $\spec R_D(\Sigma)=\{m_1,\ldots,m_d\}$ and $\sum_im_i^2=c/2$ for every $\Sigma\succ0$.
    The polynomial is homogeneous of degree $K=\sum_im_i\ge1$, and after multiplication by its constant sign its symmetric gradient is positive semidefinite of constant rank $\#\{i:m_i>0\}$ throughout the positive cone.
\end{proposition}

\begin{proof}[Proof of Proposition~\ref{prop:integral-spectrum}]
    Set $\alpha=c/2$.
    Lemma~\ref{lem:bcd-nozero} gives nonvanishing.
    Fix an arbitrary $A\succ0$ and define $P(S)=D(A^{1/2}SA^{1/2})/D(A)$ for $S\succ0$, $f=\log P$, $J=G_P/P$ and $B=J(I_d)$.
    Then $P>0$, $P(I_d)=1$ and $\tr\{(J(S)S)^2\}=\alpha$.
    
    \emph{Logarithmic characteristics.}
    For $S\succ0$, let $\cL_S:\Sym_d\to\Sym_d$ be the derivative of the matrix-valued function $J$, so that $\cL_S(H)=\dd J_S(H)$ for $H\in\Sym_d$.
    Because $J$ is the trace gradient of $f$, equality of its mixed second derivatives gives $\tr\{\cL_S(H)K\}=\tr\{H\cL_S(K)\}$.
    Differentiating $\tr\{(J(S)S)^2\}=\alpha$ in direction $H$ yields $0=2\tr\{\cL_S(H)SJ(S)S\}+2\tr\{HJ(S)SJ(S)\}$.
    Nondegeneracy of the trace pairing and self-adjointness imply
    \begin{equation*}
        \cL_S\{SJ(S)S\}
        =
        -J(S)SJ(S)
        .
    \end{equation*}
    The smooth vector field $S\mapsto SJ(S)S$ has a local solution $\dot S=SJ(S)S$, $S(0)=I_d$, in the positive cone, and along it $\dd\{J(S(t))S(t)\}/\dd t=-JSJS+JSJS=0$.
    Thus $J(S(t))S(t)=B$, $\dot S=SB$ and $S(t)=\exp(tB)$ locally.
    Also $\dd f(S(t))/\dd t=\tr\{(J(S(t))S(t))^2\}=\alpha$, so $P(\exp(tB))=\exp(\alpha t)$ locally.
    Both sides are entire functions of the scalar complex variable $t$, since $P$ is polynomial, and the identity theorem gives
    \begin{equation}
        \label{eq:S1-entire-characteristic}
        P(\exp(tB))
        =
        \exp(\alpha t)
        \quad
        (t\in\RR)
        .
    \end{equation}
    
    \emph{A full diagonal flat.}
    Choose an orthogonal eigenbasis $O$ of $B$ and write $B=O\diag(\lambda_1,\ldots,\lambda_d)O^\top$ and $\lambda=(\lambda_1,\ldots,\lambda_d)^\top$.
    Put $\tilde{P}(S)=P(OSO^\top)$ and $h(u)=\log\tilde{P}\{\diag(\exp(u_1),\ldots,\exp(u_d))\}$ for $u\in\RR^d$.
    At $S=\diag(\exp(u_1),\ldots,\exp(u_d))$, write $\tilde{J}=G_{\tilde{P}}(S)/\tilde{P}(S)$.
    Differentiation and the variance identity give
    \begin{equation*}
        \|\nabla h(u)\|^2
        =
        \sum_i \exp(2u_i)\tilde{J}_{ii}^2
        \le
        \sum_{i,j}\exp(u_i+u_j)\tilde{J}_{ij}^2
        =
        \alpha
        .
    \end{equation*}
    The inequality makes $h$ globally $\sqrt{\alpha}$-Lipschitz by integration along straight segments.
    Equation~\eqref{eq:S1-entire-characteristic} gives $h(t\lambda)=\alpha t$, while $h(0)=0$ and $\|\lambda\|^2=\tr(B^2)=\alpha$.
    Lemma~\ref{lem:calibrated-flat} therefore yields $h(u)=\lambda^\top u$, and consequently
    \begin{equation}
        \label{eq:S1-diagonal-monomial}
        \tilde{P}\{\diag(x_1,\ldots,x_d)\}
        =
        \prod_i x_i^{\lambda_i}
        \quad
        (x_i>0)
        .
    \end{equation}
    
    \emph{Polynomiality and discreteness.}
    Set every variable but $x_i$ to one in \eqref{eq:S1-diagonal-monomial}.
    The resulting nonzero polynomial $p_i$ has $p_i(1)=1$ and satisfies $xp_i'(x)=\lambda_i p_i(x)$ for $x>0$, hence identically.
    Writing $p_i(x)=\sum_j b_jx^j$ gives $(j-\lambda_i)b_j=0$ for every $j$, so $\lambda_i$ is a nonnegative integer.
    At the original base point, $B=A^{1/2}G_D(A)A^{1/2}/D(A)$ is similar to $R_D(A)$.
    Since $A$ was arbitrary, all relative eigenvalues are nonnegative integers with squared sum $\alpha$, and only finitely many ordered integer tuples satisfy this constraint.
    The ordered eigenvalues of the symmetric matrix $A^{1/2}G_D(A)A^{1/2}/D(A)$ vary continuously with $A$, so connectedness of the positive cone makes the tuple constant.
    
    \emph{Homogeneity and positivity.}
    Let $K=\sum_i m_i>0$.
    The trace identity $\tr\{G_D(A)A\}=KD(A)$ holds on the cone and hence polynomially.
    Decompose $D=\sum_jD_j$ into homogeneous components;
    Euler's identity for each component gives $\sum_j(j-K)D_j=0$, so only $D_K$ is nonzero.
    After multiplying $D$ by its sign, $D(A)>0$ and $A^{1/2}G_D(A)A^{1/2}/D(A)$ is positive semidefinite with the stated constant rank, and congruence gives the same rank and positivity for $G_D(A)$.
\end{proof}

\subsection{Irreducibility, determinant stripping and the structure of flag powers}
\label{supp:determinant}

This subsection supplies the algebraic facts used throughout: irreducibility of compression determinants, the effect of removing a power of $\det\Sigma$ on the variance identity, and the proof of Lemma~\ref{lem:flag-structure}.

\begin{lemma}[Irreducibility of the generic symmetric determinant]
\label{lem:symdet-irreducible}
    For every $r\ge1$, the determinant of a generic symmetric $r\times r$ matrix is irreducible over $\RR$ and $\CC$.
    Every normalized compression determinant $\Delta_V$, $V\ne\{0\}$, is irreducible in the ambient symmetric-matrix polynomial ring.
    For nonzero subspaces $V,W$, these normalized polynomials are associates if and only if $V=W$.
\end{lemma}

\begin{proof}[Proof of Lemma~\ref{lem:symdet-irreducible}]
    Work over either field, denoted by $\KK$.
    The assertion for $r=1$ is immediate.
    For $r\ge2$, write the generic matrix as $\left(\begin{smallmatrix}A&y\\y^\top&z\end{smallmatrix}\right)$.
    Its determinant is $(\det A)z-y^\top\operatorname{adj}(A)y$.
    Inductively $\det A$ is irreducible, and remains irreducible after adjoining the coordinates of $y$.
    In the unique factorization domain $\KK[A,y]$, it cannot divide $y^\top\operatorname{adj}(A)y$:
    the coefficient of the square of the last coordinate of $y$ is a nonzero minor of $A$ of degree $r-2$, whereas $\deg\det A=r-1$.
    For $r=2$ this coefficient is one.
    The two coefficients of the polynomial in $z$ are therefore coprime.
    It is primitive and has degree one over the fraction field, so Gauss' lemma gives irreducibility in $\KK[A,y,z]$.
    
    Adjoining additional independent symmetric coordinates preserves irreducibility, and an orthogonal change of coordinates is a polynomial-ring automorphism.
    This proves the assertion for $\Delta_V$.
    Finally $\Delta_V(I_d)=1$ and $G_{\Delta_V}(I_d)=\Pi_V$.
    Associate normalized polynomials are equal, so their gradients at $I_d$ give $\Pi_V=\Pi_W$ and hence $V=W$.
\end{proof}

The polynomiality criterion of Section~\ref{sec:flags} follows.
For integer weights $m_i$, with $m_{d+1}=0$, Schur complementation gives $\prod_i v_i^{m_i}=\prod_{r=1}^d\Delta_r^{m_r-m_{r+1}}$ in the fraction field.
The leading minors are pairwise nonassociate irreducibles, so unique factorization makes this expression polynomial exactly when all displayed exponents are nonnegative.

\begin{lemma}[Determinant stripping]
\label{lem:strip-supp}
    Let $E\in\RR[\Sym_d]$ be nonzero and homogeneous of degree $K_E\ge0$, let $\ell\in\ZZ_{\ge0}$, and put $D=(\det\Sigma)^\ell E$.
    For a real constant $c_D$, set $c_E=c_D-4\ell K_E-2d\ell^2$.
    Then the polynomial identity $2\tr\{(G_D\Sigma)^2\}=c_DD^2$ holds if and only if $2\tr\{(G_E\Sigma)^2\}=c_EE^2$ holds.
    When $E$ is constant, either identity forces $c_E=0$.
    For $K_E>0$, either identity forces $c_E\ge2K_E^2/d>0$.
\end{lemma}

\begin{proof}[Proof of Lemma~\ref{lem:strip-supp}]
    On the nonempty part of the positive cone where $E\ne0$, logarithmic differentiation gives $R_D=\ell I_d+R_E$, and Euler's identity gives $\tr R_E=K_E$.
    Hence $2\tr(R_D^2)=2d\ell^2+4\ell K_E+2\tr(R_E^2)$, which proves the equivalence on that set;
    clearing denominators and polynomial continuation give the identities everywhere.
    For $K_E>0$, $R_E$ is similar to a symmetric matrix of trace $K_E$, so the sum-of-squares inequality gives $2\tr(R_E^2)\ge2K_E^2/d$.
    A degree-zero nonzero $E$ is constant and has zero gradient.
\end{proof}

\begin{proof}[Proof of Lemma~\ref{lem:flag-structure}]
    Lemma~\ref{lem:symdet-irreducible} makes compression determinants irreducible and pairwise nonassociate unless their subspaces coincide.
    Unique factorization in $\RR[\Sym_d]$, together with $\Delta_V(I_d)=1$, therefore identifies the positive-exponent subspaces, their exponents and $\gamma$ from $D$, and the exponent of $\det\Sigma=\Delta_{\RR^d}$ is $m_d$.
    The homogeneity, sign and relative gradient are Lemma~\ref{lem:flag-relative-gradient}.
\end{proof}

\subsection{Newton characteristics and a fixed kernel}
\label{supp:fixed-kernel}

This subsection carries out the second structural step: a homogeneous polynomial whose gradient is positive semidefinite of constant rank on the cone depends only on a fixed compression.
Define $\mathrm{e}_j(C)$ and $\cT_r(C)$, for $C\in\Sym_d$, by
\begin{equation*}
    \det(tI_d+C)=\sum_{j=0}^d\mathrm{e}_j(C)t^{d-j}
    ,\quad
    \mathrm{e}_0=1
    ,\quad
    \cT_r(C)=G_{\mathrm{e}_{r+1}}(C)
    ,\quad
    0\le r<d
    .
\end{equation*}
The gradient here is with respect to $C$ in the trace inner product.
Diagonalizing $C$ gives the polynomial and trace identities
\begin{equation}
    \label{eq:S1-newton-formulas}
    \cT_r(C)=\sum_{j=0}^r(-1)^j\mathrm{e}_{r-j}(C)C^j
    ,
    \quad
    \tr\{C\cT_r(C)\}=(r+1)\mathrm{e}_{r+1}(C)
    .
\end{equation}
Indeed the $i$th diagonal derivative of $\mathrm{e}_{r+1}$ at a diagonal matrix is the elementary symmetric polynomial of order $r$ in all eigenvalues except the $i$th;
off-diagonal derivatives vanish by invariance under diagonal sign changes.
Orthogonal equivariance then gives the first identity.
The second is Euler's identity.
In particular, if $C\succeq0$ has rank $r\ge1$ with positive eigenvalues $\theta_1,\ldots,\theta_r$, then
\begin{equation}
    \label{eq:S1-newton-projection}
    \cT_r(C)
    =
    \Biggl(
        \prod_{i=1}^r\theta_i
    \Biggr)\Pi_{\ker C}
    .
\end{equation}

\begin{lemma}[Newton characteristic lines]
\label{lem:Newton-characteristics}
    Let $P\in\RR[\Sym_d]$ and $0\le r<d$ satisfy $\mathrm{e}_{r+1}(G_P)\equiv0$.
    For $A\in\Sym_d$, set $C_A=\cT_r(G_P(A))$.
    Then $G_P(A+tC_A)=G_P(A)$ and $P(A+tC_A)=P(A)$ for every real $t$.
\end{lemma}

\begin{proof}[Proof of Lemma~\ref{lem:Newton-characteristics}]
    For $S\in\Sym_d$, let $\cH_S:\Sym_d\to\Sym_d$ be the Hessian operator associated with the symmetric gradient, defined by $\cH_S(H)=\dd(G_P)_S(H)$ for $H\in\Sym_d$.
    Equality of the second differentials of $P$ makes $\cH_S$ self-adjoint for the trace pairing.
    Differentiating the assumed identity in direction $H$ gives $0=\tr\{\cT_r(G_P(S))\cH_S(H)\}=\tr\{H\cH_S(\cT_r(G_P(S)))\}$, so $\cH_S(\cT_r(G_P(S)))=0$.
    The polynomial vector field $\dot S=\cT_r(G_P(S))$ has a local solution through $A$.
    Along it, $G_P(S(t))$ is constant, so the velocity equals $C_A$ and $S(t)=A+tC_A$ locally.
    The second identity in \eqref{eq:S1-newton-formulas} gives $\dd P(S(t))/\dd t=(r+1)\mathrm{e}_{r+1}(G_P(A))=0$.
    Both conclusions hold locally in $t$ and are polynomial identities in $t$, so they hold for every real $t$.
\end{proof}

\begin{proposition}[Fixed kernel and compression for monotone homogeneous polynomials]
    \label{prop:fixed-kernel-compression}
    Let $P\in\RR[\Sym_d]$ be homogeneous of degree $K\ge1$ and satisfy $G_P(S)\succeq0$ and $\rank G_P(S)=r$ for every $S\succ0$, where $1\le r<d$.
    Then $U=\ker G_P(I_d)$ is the kernel of $G_P(S)$ at every such $S$, and $G_P(S)u\equiv0$ for every $u\in U$.
    If $M\in\RR^{r\times d}$ has orthonormal rows spanning $U^\perp$, there exists a nonconstant homogeneous polynomial $\bar{P}$ of degree $K$ on $\Sym_r$ such that
    \begin{equation}
        \label{eq:S1-fixed-compression}
        P(S)
        =
        \bar{P}(MSM^\top)
        \quad
        (S\in\Sym_d)
        .
    \end{equation}
\end{proposition}

\begin{proof}[Proof of Proposition~\ref{prop:fixed-kernel-compression}]
    The rank condition implies $\mathrm{e}_{r+1}(G_P)\equiv0$ by polynomial continuation.
    Set $C_0=\cT_r(G_P(I_d))$.
    Equation~\eqref{eq:S1-newton-projection} gives $C_0\succeq0$ and $\operatorname{Range}C_0=U$, and Lemma~\ref{lem:Newton-characteristics} yields $P(I_d+tC_0)=P(I_d)$ for all real $t$.
    
    The gradient condition implies Loewner monotonicity within the cone:
    if $0\prec S\preceq S'$, then $P(S')-P(S)=\int_0^1\tr\{G_P(S+s(S'-S))(S'-S)\}\dd s\ge0$.
    For any $A_0\succ0$, choose $\lambda_0>0$ with $A_0\preceq\lambda_0 I_d$.
    Every matrix in the following comparison is positive definite for $t\ge0$:
    \begin{equation*}
        P(A_0)
        \le
        P(A_0+tC_0)
        \le
        P(\lambda_0 I_d+tC_0)
        =
        \lambda_0^K P\{I_d+(t/\lambda_0)C_0\}
        =
        \lambda_0^K P(I_d)
        .
    \end{equation*}
    A real polynomial bounded above and below on $[0,\infty)$ is constant, since any nonzero leading term of positive degree is unbounded there.
    Thus $t\mapsto P(A_0+tC_0)$ is constant and $\tr\{G_P(A_0)C_0\}=0$.
    For positive-semidefinite matrices $H,C$, $\tr(HC)=\|H^{1/2}C^{1/2}\|_F^2$, so zero trace implies $HC^{1/2}=0$ and therefore $H\operatorname{Range}C=0$.
    Hence $G_P(A_0)$ annihilates $U$, and since $U$ and the kernel both have dimension $d-r$, they coincide.
    For fixed $u\in U$, every entry of $G_P(S)u$ is polynomial and vanishes on the positive cone, hence vanishes identically.
    
    Complete the rows of $M$ to an orthogonal coordinate system whose last $d-r$ directions span $U$.
    In these coordinates the corresponding rows and columns of $G_P$ vanish identically, and \eqref{eq:gradient-convention} then makes every partial derivative involving a coordinate outside the leading $r\times r$ symmetric block zero.
    In characteristic zero this means that $P$ depends only on that block, yielding \eqref{eq:S1-fixed-compression}.
    The compression is surjective on $\Sym_r$, so nonconstancy and homogeneous degree are preserved.
\end{proof}

\subsection{Proof of the necessity in Theorem~\ref{thm:sn-pivot-classification}}
\label{supp:classification-necessity}

The two structural steps are now combined by induction on the dimension.

\begin{proof}[Proof of Theorem~\ref{thm:sn-pivot-classification}:
(a)$\Rightarrow$(c)]
    Proposition~\ref{prop:integral-spectrum} first gives homogeneity of degree $K=\deg D\ge1$ and a fixed nonnegative integral relative spectrum.
    We prove the flag form by strong induction on $d$.
    For $d=1$, homogeneity gives $D(\sigma_{11})=\gamma\sigma_{11}^K$.
    Suppose the assertion holds in every smaller dimension, and write $D=(\det\Sigma)^\ell E$ with $\ell\ge0$ maximal.
    Homogeneous-component comparison shows that $E$ is homogeneous of degree $K_E=\deg D-d\ell$.
    If $K_E=0$, the claim is a full-determinant power.
    Otherwise Lemma~\ref{lem:strip-supp} gives a positive-constant proportionality identity for $E$.
    Apply Proposition~\ref{prop:integral-spectrum} to $E$ and denote its relative eigenvalues by $\lambda_1(E),\ldots,\lambda_d(E)$.
    Taking determinants in $G_E(\Sigma)\Sigma=E(\Sigma)R_E(\Sigma)$ and clearing denominators gives $\det G_E(\Sigma)\det\Sigma=\{\prod_i\lambda_i(E)\}E(\Sigma)^d$.
    If the scalar product on the right were nonzero, primality of $\det\Sigma$ would imply $\det\Sigma\mid E$, contrary to maximality of $\ell$.
    Thus the product is zero.
    After multiplying $E$ by its sign on the positive cone, its gradient is positive semidefinite of constant rank $r$ with $1\le r<d$.
    
    Proposition~\ref{prop:fixed-kernel-compression} supplies $M\in\RR^{r\times d}$ with orthonormal rows and a nonconstant homogeneous polynomial $\bar{E}\in\RR[\Sym_r]$ such that $E(\Sigma)=\bar{E}(M\Sigma M^\top)$ for every $\Sigma\in\Sym_d$.
    Lemma~\ref{lem:trace-main} implies that $\bar{E}$ has the same proportionality constant.
    Every positive-definite $r\times r$ matrix can occur as the compression, by adjoining a positive identity block on $\ker M$, so this is an identity on the lower-dimensional cone and hence in its ambient polynomial ring.
    The induction hypothesis writes $\bar{E}$ as a flag power.
    The isometry $M^\top$ lifts its subspaces to a nested collection in $\RR^d$.
    Restoring $(\det\Sigma)^\ell$ adds only the full space to this collection;
    the sign is absorbed into $\gamma$.
    This proves (c).
    
    The relative eigenvalues of a nonconstant determinant-free residue are nonnegative and include zero, and a constant residue has zero relative gradient, so $R_D=\ell I_d+R_E$ has minimum eigenvalue $\ell$;
    this agrees with $\ell=m_d$ in Lemma~\ref{lem:flag-structure}.
    The constant $c$ is given by Lemma~\ref{lem:flag-structure}, and the sampling statements are proved in Section~\ref{supp:pivots}.
\end{proof}

\section{Proofs for the sampling laws and statistical consequences}
\label{supp:statistical-consequences}

This section proves the sampling assertions of Theorem~\ref{thm:sn-pivot-classification}, its corollaries in Section~\ref{sec:consequences}, the smooth counterexample cited in Section~\ref{sec:classification}, and the unknown-mean conventions of Section~\ref{sec:notation}.

\subsection{Proof of the sampling statements in Theorem~\ref{thm:sn-pivot-classification}}
\label{supp:pivots}

We prove the classical implication (c)$\Rightarrow$(b) with the explicit law \eqref{eq:wishart-pivot}, and the statistical converse (b)$\Rightarrow$(a) in the local form of Remark~\ref{rem:local-pivot};
the implications (c)$\Rightarrow$(a) and (a)$\Rightarrow$(c) are Lemma~\ref{lem:flag-structure} and Section~\ref{supp:classification-necessity}.

\begin{proof}[Proof of Theorem~\ref{thm:sn-pivot-classification}: (c)$\Rightarrow$(b), \eqref{eq:wishart-pivot} and (b)$\Rightarrow$(a)]
    Work in orthonormal coordinates adapted to the flag, and let $r=\max\{i:m_i>0\}$.
    Only the leading $r\times r$ covariance block enters $D$.
    Write this population block as $TT^\top$ with $T$ lower triangular and positive diagonal.
    If $\mathbf{Z}$ is an $n\times r$ matrix with independent standard normal entries, then $n\Sighat_{1:r}\stackrel{d}{=}T\mathbf{Z}^\top\mathbf{Z}T^\top$, and all leading blocks of this lower-triangular congruence factor in the same way.
    It remains to compute the leading determinants of $\mathbf{Z}^\top\mathbf{Z}$.
    
    Let $z_j$ be column $j$ of $\mathbf{Z}$, let $\Pi_{j-1}$ be the orthogonal projection in $\RR^n$ onto the span of $z_1,\ldots,z_{j-1}$, with $\Pi_0=0$, and define $U_{n,j}=\|(I_n-\Pi_{j-1})z_j\|^2$.
    For $n\ge r$ these columns have the required ranks almost surely.
    Conditionally on the earlier columns, $U_{n,j}$ has law $\chi^2_{n-j+1}$, independent of their values, so $U_{n,j}$ is independent of the sigma-field they generate, and induction makes $U_{n,1},\ldots,U_{n,r}$ mutually independent.
    Gram--Schmidt orthogonalization gives $\det\{(\mathbf{Z}^\top\mathbf{Z})_{1:s}\}=\prod_{j=1}^sU_{n,j}$ for $1\le s\le r$, and therefore $\det\Sighat_{1:s}/\det\Sigma_{1:s}\stackrel{d}{=}n^{-s}\prod_{j=1}^sU_{n,j}$ jointly over $s\le r$.
    Multiplying with the flag exponents proves \eqref{eq:wishart-pivot} and (c)$\Rightarrow$(b).
    This is the leading-minor part of Bartlett's decomposition \citep[Chapter~3]{muirhead1982}, established here directly at the block-level sample-size threshold.
    
    For (b)$\Rightarrow$(a), suppose that on one fixed nonempty Euclidean-open set $\cO\subseteq\cU_D$ the sample-to-population laws are common to all $\Sigma\in\cO$ along an unbounded sequence of integer sample sizes, and write these sizes as $n_j\to\infty$.
    For every fixed $\Sigma\in\cO$, the central limit theorem for $\vech(\Sighat)$ and the delta method give $\sqrt{n_j}\{D(\Sighat)/D(\Sigma)-1\}\dto N_1\{0,\kappa_D(\Sigma)\}$.
    The finite-sample laws on the left are identical for any two covariances in $\cO$, for each $j$, so uniqueness of weak limits makes $\kappa_D$ one constant on $\cO$.
    If that constant were zero, the identity $\kappa_D(\Sigma)=2\|\Sigma^{1/2}G_D(\Sigma)\Sigma^{1/2}/D(\Sigma)\|_F^2$ would imply $G_D=0$ on $\cO$;
    each coordinate derivative is a polynomial, and vanishing on a nonempty Euclidean-open set makes it zero identically, which would make $D$ constant.
    Hence the common constant is $c>0$, and the polynomial $s_D^2-cD^2$ vanishes on $\cO$ and hence identically on $\Sym_d$.
    This proves (b)$\Rightarrow$(a) in the local form of Remark~\ref{rem:local-pivot}, and hence for (b) itself.
\end{proof}

\subsubsection{A covariance submodel with an additional pivot}
This example, cited in Remark~\ref{rem:local-pivot}, shows that the open set there cannot be dropped.
For $d=2$ and $\Sigma=\diag(s_1,s_2)$ with $s_1,s_2>0$, the two Gaussian sample columns are independent.
With $D(\Sigma)=\sigma_{11}\sigma_{22}$ and the known-mean convention, $D(\Sighat)/D(\Sigma)\stackrel{d}{=}U_nV_n/n^2$ with independent $U_n,V_n\sim\chi_n^2$, so the law does not depend on $s_1,s_2$ on this two-dimensional family.
On the full three-dimensional cone $\Smat^2_{++}$, however, $\kappa_D=4+4\rho^2$ with $\rho=\sigma_{12}/\sqrt{\sigma_{11}\sigma_{22}}$.
Thus the submodel has a relative pivot although the ambient polynomial's first-order relative variance is nonconstant, and the submodel has empty interior in $\Sym_2$.

\subsubsection{Exact calibration and uniform relative limits}
\label{supp:exact-calibration}

\begin{proof}[Proof of Corollary~\ref{cor:exact-calibration}]
    Each $U_{n,i}$ has a strictly positive continuous density on $(0,\infty)$, so $\log W_{n,m}$ is a sum of independent variables with strictly positive densities on $\RR$.
    It therefore has a continuous strictly increasing distribution function on $\RR$, and $W_{n,m}$ has that property on $(0,\infty)$.
    Hence the $p$ quantile $r_{n,p}=\inf\{x>0:\Pr(W_{n,m}\le x)\ge p\}$ is unique and $\Pr(r_{n,\alpha/2}\le W_{n,m}\le r_{n,1-\alpha/2})=1-\alpha$.
    For $\gamma>0$, $D$ is positive on the cone by Lemma~\ref{lem:flag-structure}, and this event is exactly membership of $D(\Sigma)$ in \eqref{eq:sn-exact-calibration}.
    
    For fixed $r$, the independent coordinates satisfy $\sqrt{n}(U_{n,i}/n-1)\dto N_1(0,2)$ for $i=1,\ldots,r$, jointly with independent limits;
    this follows by applying the central limit theorem to $n-i+1$ independent squared standard normals, since the centering adjustment $(1-i)/\sqrt{n}$ tends to zero.
    The product map $x\mapsto\prod_ix_i^{m_i}$ and its logarithm have derivatives $(m_1,\ldots,m_r)$ at $(1,\ldots,1)$, so the multivariate delta method gives variance $2\sum_im_i^2=c$ and proves the two limits in Corollary~\ref{cor:exact-calibration}.
    Convergence in distribution to the continuous standard normal law implies uniform convergence of the distribution functions by P\'olya's theorem.
    For each $n$ the distribution function is the same at every $\Sigma\succ0$, so its supremum over covariances equals this one-dimensional discrepancy, which proves the uniformity;
    in particular, $\sqrt{n}\{D(\Sighat)/D(\Sigma)-1\}$ is tight uniformly over $\Sigma\succ0$.
\end{proof}

\subsection{Proof of Corollary~\ref{cor:bcd-precision} and precision-matrix identities}
\label{supp:precision}

Inversion is an isometry of the Fisher--Rao metric and reverses a flag through complementary minors;
these two facts transfer the classification to precision polynomials.

\begin{lemma}[Inversion and complementary compression determinants]
    \label{lem:bcd-inversion-isometry}
    Inversion is an isometry of $\Smat^d_{++}$ for $g^{\mathrm{FR}}$.
    Let $\Sigma\succ0$ and put $\Omega=\Sigma^{-1}$.
    For a scalar function $q$ differentiable on a neighborhood of $\Omega$ with $q(\Omega)\ne0$, and $F(\Sigma)=q(\Omega)^\varepsilon$ with $\varepsilon\in\{-1,1\}$,
    \begin{equation}
        \label{eq:inversion-gradient}
        \frac{G_F(\Sigma)}{F(\Sigma)}
        =
        -\varepsilon\Omega\frac{G_q(\Omega)}{q(\Omega)}\Omega
        ,\quad
        \kappa_F(\Sigma)
        =
        \kappa_q(\Omega)
        .
    \end{equation}
    If $0\le r\le d$, $O_1\in\RR^{d\times r}$ and $O_2\in\RR^{d\times(d-r)}$ are such that $(O_1,O_2)$ is orthogonal, then
    \begin{equation}
        \label{eq:complementary-minor}
        \det(O_1^\top\Sigma^{-1}O_1)
        =
        \frac{\det(O_2^\top\Sigma O_2)}{\det\Sigma}
        ,
    \end{equation}
    with the empty determinant equal to one when a block is empty.
\end{lemma}

\begin{proof}[Proof of Lemma~\ref{lem:bcd-inversion-isometry}]
    Differentiating $\Sigma\Sigma^{-1}=I_d$ gives $\dd(\Sigma^{-1})_\Sigma(H)=-\Omega H\Omega$, so
    \begin{equation*}
        g^{\mathrm{FR}}_\Omega(-\Omega H\Omega,-\Omega K\Omega)
        =
        \tr(H\Omega K\Omega)/2
        =
        g^{\mathrm{FR}}_\Sigma(H,K)
        .
    \end{equation*}
    The chain rule for $\log|F|=\varepsilon\log|q(\Sigma^{-1})|$ gives the gradient formula, and cyclicity of trace proves the first-order relative-variance identity.
    For \eqref{eq:complementary-minor}, express $\Sigma$ in the orthogonal basis $(O_1,O_2)$.
    The $O_1$ block of its inverse is the inverse of the Schur complement of $O_2^\top\Sigma O_2$, whose determinant is $\det\Sigma/\det(O_2^\top\Sigma O_2)$ by the block determinant formula;
    taking the reciprocal proves the assertion, and for an empty block the equation reduces to the determinant or inverse-determinant identity.
\end{proof}

\begin{proof}[Proof of Corollary~\ref{cor:bcd-precision}]
    On $\cV_q$, $F$ is a smooth nonzero rational function.
    For $n\ge d$ the sample covariance is positive definite almost surely, and $q(\Sighat^{-1})=0$ has probability zero:
    with $K=\deg q$, the function $\tilde{q}(S)=(\det S)^Kq\{\operatorname{adj}(S)/\det S\}$ is a polynomial, nonzero because inversion maps the positive cone onto itself and $q$ is nonzero on an open subset of that cone, and the sample zero-set argument of Section~\ref{supp:sample-domains} applies to $\tilde{q}$.
    
    If $F$ has common sample-to-population laws along an unbounded sample-size sequence, the pointwise rational delta method gives common normal limits, as in Section~\ref{supp:pivots}, so $\kappa_F$ is constant on $\cV_q$.
    Lemma~\ref{lem:bcd-inversion-isometry} transfers this constant to $\kappa_q$ on $\cU_q=\{\Omega\succ0:q(\Omega)\ne0\}$.
    The constant is positive, since a zero value would give $G_q=0$ on an open set and make $q$ constant, and Theorem~\ref{thm:sn-pivot-classification} classifies $q$ as a flag power.
    The same reasoning starts directly from constant positive $\kappa_F$.
    
    Conversely, let the flag of $q$ in its precision argument have dimensions $r_j$, exponents $a_j$ and multiplicities $m_i$.
    Choose an orthonormal basis adapted to it, and reverse the basis order for covariance coordinates.
    Denote the corresponding leading covariance minors by $\tilde{\Delta}_s$, with $\tilde{\Delta}_0=1$ and $\tilde{\Delta}_d=\det\Sigma$.
    Lemma~\ref{lem:bcd-inversion-isometry} gives
    \begin{equation*}
        q(\Sigma^{-1})
        =
        \gamma\prod_j
        \Biggl(
            \frac{\tilde{\Delta}_{d-r_j}(\Sigma)}
            {\tilde{\Delta}_d(\Sigma)}
        \Biggr)^{a_j}
        .
    \end{equation*}
    In the successive residual variances of these reversed coordinates, the exponent in position $i$ is $-\sum_ja_j+\sum_{j:\,i\le d-r_j}a_j=-\sum_{j:\,r_j\ge d+1-i}a_j=-m_{d+1-i}$, and raising to $\varepsilon$ multiplies all exponents by $\varepsilon$.
    For $n\ge d$, all leading sample minors and all Bartlett factors are strictly positive, so the simultaneous leading-minor factorization of Section~\ref{supp:pivots} applies with these signed integer exponents and gives \eqref{eq:precision-pivot}.
    Equation~\eqref{eq:inversion-gradient} gives the first-order relative-variance constant, and the flag form of $q$ gives $\cV_q=\Smat^d_{++}$.
\end{proof}

\subsubsection{Variance functionals in Corollary~\ref{cor:bcd-precision}}
For $q(\Omega)=u^\top\Omega u$ with $u\ne0$, the multiplicities are $(1,0,\ldots,0)$ and the constant is $2$, so $q(\Sighat^{-1})/q(\Sigma^{-1})$ has law $n/\chi^2_{n-d+1}$ and the reciprocal functional has sample-to-population law $\chi^2_{n-d+1}/n$.
For $u=e_j$, a one-coordinate Schur complement gives $1/(\Sigma^{-1})_{jj}=\var(X_j\mid X_{-j})$ under the Gaussian law.
For $u=\1_d$, Cauchy--Schwarz gives $(\1_d^\top w)^2\le(w^\top\Sigma w)(\1_d^\top\Sigma^{-1}\1_d)$ with equality at $w=\Sigma^{-1}\1_d/(\1_d^\top\Sigma^{-1}\1_d)$, which proves the minimum-variance characterization in Section~\ref{sec:consequences}.

\subsubsection{Signed flag products and the residual-variance example of Section~\ref{sec:discussion}}
Let $V_1\subsetneq\cdots\subsetneq V_k$ be nonzero subspaces, $r_j=\dim V_j$, $a_j\in\ZZ$, and $\gamma\ne0$.
The rational functional $F=\gamma\prod_{j=1}^k\Delta_{V_j}^{a_j}$ is finite and nonzero on $\Smat^d_{++}$;
omit factors with $a_j=0$ and suppose $F$ is nonconstant.
In adapted coordinates put $w_i=\sum_{j:r_j\ge i}a_j$.
The relative-gradient calculation of Lemma~\ref{lem:flag-relative-gradient} gives $\kappa_F=2\sum_iw_i^2>0$.
For $n\ge r_k$, all active sample blocks are positive definite almost surely, and the simultaneous leading-minor factorization of Section~\ref{supp:pivots} gives $F(\Sighat)/F(\Sigma)\stackrel{d}{=}\prod_{i=1}^{r_k}(U_{n,i}/n)^{w_i}$ with independent $U_{n,i}\sim\chi^2_{n-i+1}$, for positive, zero and negative $w_i$ alike.

For the bivariate residual variance $F(\Sigma)=\det\Sigma/\sigma_{11}=\sigma_{22}-\sigma_{12}^2/\sigma_{11}$, direct differentiation under the symmetric-gradient convention yields
\begin{equation*}
    G_F(\Sigma)=
    \begin{pmatrix}
        \sigma_{12}^2/\sigma_{11}^2&-\sigma_{12}/\sigma_{11}\\
        -\sigma_{12}/\sigma_{11}&1
    \end{pmatrix}
    ,\quad
    \ker G_F(\Sigma)
    =
    \operatorname{span}\{(1,\sigma_{12}/\sigma_{11})^\top\}
    .
\end{equation*}
Also $G_F(\Sigma)\Sigma/F(\Sigma)$ has diagonal entries $0,1$ and is upper triangular, so $\kappa_F=2$, while the kernel varies with $\sigma_{12}/\sigma_{11}$.
For a fixed $\Theta=(\theta_{ij})\succ0$ with $\theta_{12}\ne0$ and $t\ge0$, $F(\Theta+te_1e_1^\top)=\theta_{22}-\theta_{12}^2/(\theta_{11}+t)$ is bounded and strictly increasing.
This identifies the step in the proof of Proposition~\ref{prop:fixed-kernel-compression} whose rational analogue requires a new argument.

\subsection{Proof of Corollary~\ref{cor:log-variance-stabilization}}
\label{supp:vst}

Homogeneity confines the argument to a scaling ray, on which a variance-stabilizing transformation is forced to be logarithmic.

\begin{proof}[Proof of Corollary~\ref{cor:log-variance-stabilization}]
    Suppose first that \eqref{eq:vst} holds.
    Homogeneity of $D$ implies $s_D^2(t\Sigma)=t^{2K}s_D^2(\Sigma)$ for $t>0$.
    If $D(\Sigma_0)=0$ for a positive-definite $\Sigma_0$, then \eqref{eq:vst} evaluated at $t\Sigma_0$ becomes $\{g'(0)\}^2t^{2K}s_D^2(\Sigma_0)=v_0$ for all $t>0$, which is impossible since $K\ge1$ and $v_0>0$.
    Thus $D$ has constant nonzero sign on the connected cone.
    Fix $\Sigma_0\succ0$ and put $x_0=D(\Sigma_0)$.
    The scaling ray $t\Sigma_0$ has values $t^Kx_0$ covering the entire half-line of that sign, so $D(\Smat^d_{++})$ is precisely that half-line.
    
    At $\Sigma_0$, equation~\eqref{eq:vst} gives $s_D^2(\Sigma_0)>0$.
    Set $c=s_D^2(\Sigma_0)/x_0^2>0$.
    By homogeneity, $s_D^2(t\Sigma_0)=cD(t\Sigma_0)^2$ on this ray, and since the ray covers the range, \eqref{eq:vst} forces $x^2\{g'(x)\}^2=v_0/c$ throughout that range.
    The function $xg'(x)$ is continuous and takes only the two values $\pm\sqrt{v_0/c}$, so connectedness makes it a fixed nonzero constant $a$, and integration gives $g(x)=a\log|x|+b$ on the range.
    Substitution at any positive-definite covariance gives $s_D^2=cD^2$, polynomial continuation gives self-normalization, and Theorem~\ref{thm:sn-pivot-classification} yields the flag form.
    
    Conversely, a flag power has constant nonzero sign and is homogeneous of positive degree, so its range is one open half-line, on which each $g(x)=a\log|x|+b$ with $a\ne0$ is continuously differentiable and satisfies \eqref{eq:vst} with $v_0=a^2c$.
    For the nonhomogeneous comparison in Section~\ref{sec:consequences}, $D=1+\det\Sigma$ has range $(1,\infty)$ and $\log(D-1)=\log\det\Sigma$ has first-order variance $2d$, directly from its gradient $\Sigma^{-1}$.
\end{proof}

\subsection{Proof of Corollary~\ref{cor:one-point-diagnostic}}
\label{supp:reconstruction}

The value and gradient at $I_d$ determine the only possible flag, and one polynomial identity then decides self-normalization.

\begin{proof}[Proof of Corollary~\ref{cor:one-point-diagnostic}]
    If $D(I_d)=0$, Lemma~\ref{lem:bcd-nozero} excludes self-normalization.
    Suppose $D$ is self-normalizing, so that Theorem~\ref{thm:sn-pivot-classification} writes it in the form \eqref{eq:flag-power}.
    For every factor, $\Delta_V(I_d)=1$ and $G_{\Delta_V}(I_d)=\Pi_V$, so $D(I_d)=\gamma$ and $B=G_D(I_d)/D(I_d)=\sum_ja_j\Pi_{V_j}$.
    In coordinates adapted to the flag, this sum is diagonal with entries $m_1,\ldots,m_d$.
    A positive exponent occurs exactly at a strict drop $m_r>m_{r+1}$, and then the corresponding subspace is the direct sum of the eigenspaces of $B$ belonging to its $r$ largest eigenvalues, which identifies the subspace independently of the eigenbasis and proves \eqref{eq:one-point-reconstruction}.
    At positions with equal successive eigenvalues the exponent is zero, so choices inside a repeated eigenspace do not change the right-hand side.
    
    Conversely, suppose the spectral condition and the polynomial identity hold.
    The subspaces $\operatorname{Range}O_r$ are nested, the exponents $m_r-m_{r+1}$ are nonnegative integers, and at least one is positive, so the right-hand side of \eqref{eq:one-point-reconstruction} is a flag power and Lemma~\ref{lem:flag-structure} gives self-normalization.
    For $D=(\tr\Sigma)^d$, $D(I_d)=d^d$ and $G_D(I_d)=d^dI_d$, yielding the candidate $d^d\det\Sigma$, which differs from $D$ at any positive diagonal matrix with unequal entries by the strict arithmetic--geometric mean inequality.
\end{proof}

For exact input coefficients, the certification can be carried out without orthonormalizing algebraic eigenvectors.
Once the eigenvalues of the symmetric $B$ have been verified to be nonnegative integers, the spectral subspace at a strict drop $m_r>m_{r+1}$ is the direct sum of $\ker(B-jI_d)$ over the distinct eigenvalues $j$ among $m_1,\ldots,m_r$.
These kernels have rational bases for rational input.
If $C_r$ is any full-column-rank basis matrix of that subspace, then $\Delta_{\operatorname{Range}C_r}(\Sigma)=\det(C_r^\top\Sigma C_r)/\det(C_r^\top C_r)$, by replacing $C_r$ with an orthonormal basis of the same space.
Hence rational arithmetic suffices for constructing and comparing the candidate after the integral spectral test, which checks the characteristic polynomial of $B$ for integer roots within a computable spectral bound;
real algebraic input is treated in the same way with real algebraic arithmetic.

\subsection{A smooth first-order solution without a relative pivot}
\label{supp:fisher-exponential}

This example, cited in Section~\ref{sec:classification}, shows that polynomiality is essential in Theorem~\ref{thm:sn-pivot-classification}.

\begin{example}[First-order relative stabilization without a relative pivot]
\label{ex:fisher-exponential}
    For $\Sigma\in\Smat^2_{++}$, let $\rho(\Sigma)=\sigma_{12}/\sqrt{\sigma_{11}\sigma_{22}}$ and $F(\Sigma)=\exp\{\operatorname{arctanh}\rho(\Sigma)\}=\{(1+\rho)/(1-\rho)\}^{1/2}$.
    This positive smooth function satisfies $s_F^2=F^2$ and is nonpolynomial, and under known-mean Gaussian sampling $F(\Sighat)/F(\Sigma)$ has a covariance-dependent law for every $n\ge2$.
\end{example}

Thus, among smooth functions, constant first-order relative variance does not force a relative pivot.

\begin{proof}[Proof of Example~\ref{ex:fisher-exponential}]
    In coordinates $(x,z,y)=(\sigma_{11},\sigma_{12},\sigma_{22})$,
    \begin{equation*}
        \nabla\rho
        =
        \Biggl(
            -\frac{\rho}{2x},\frac{1}{\sqrt{xy}},
            -\frac{\rho}{2y}
        \Biggr)^\top
        ,
        \quad
        \Gamma
        =
        \begin{pmatrix}
        2x^2&2xz&2z^2\\
        2xz&xy+z^2&2yz\\
        2z^2&2yz&2y^2
        \end{pmatrix}
        .
    \end{equation*}
    Their contraction gives $s_\rho^2=(1-\rho^2)^2$, the chain rule gives $s_h^2=1$ for $h=\operatorname{arctanh}\rho$, and hence $s_F^2/F^2=1$ for $F=\exp(h)$.
    The function is nonconstant and satisfies $F(t\Sigma)=F(\Sigma)$ for every $t>0$, whereas a polynomial with this invariance on an open cone is constant by homogeneous-component comparison, so $F$ is nonpolynomial.
    
    For the sample-to-population law, it is enough to use $\Sigma_\rho=\left(\begin{smallmatrix}1&\rho\\\rho&1\end{smallmatrix}\right)$, $-1<\rho<1$.
    Let $x,y\in\RR^n$ be the two sample coordinate columns and put $\hat{\rho}=x^\top y/(\|x\|\|y\|)$, so that $F(\Sighat)/F(\Sigma_\rho)=\exp(\operatorname{arctanh}\hat{\rho}-\operatorname{arctanh}\rho)$;
    for $n\ge2$, $|\hat{\rho}|<1$ almost surely.
    Let $\EE_\rho$ denote expectation for the independent Gaussian observations with covariance $\Sigma_\rho$, and define $b_n(\rho)=\EE_\rho(\operatorname{arctanh}\hat\rho)-\operatorname{arctanh}\rho$.
    This expectation is finite and differentiable at zero:
    on any compact subinterval of $(-1,1)$ the Gaussian density and its derivative in $\rho$ are bounded by a polynomial in $\|x\|+\|y\|$ times $\exp\{-\epsilon_0(\|x\|^2+\|y\|^2)\}$ for some $\epsilon_0>0$, and in polar coordinates the angular singularities of $\operatorname{arctanh}\hat{\rho}$ are logarithmic and integrable for $n\ge2$.
    
    At $\rho=0$ the score with respect to $\rho$ is $x^\top y$.
    The radii $\|x\|$ and $\|y\|$ are independent $\chi_n$ variables, independent of the angle between $x$ and $y$, whose cosine $\hat{\rho}$ has density $f_n(u)=C_n(1-u^2)^{(n-3)/2}$ on $(-1,1)$ with normalizing constant $C_n$, by spherical coordinates for an independent uniform direction in $\RR^n$.
    Writing $\mu_n=\EE(\|x\|)$, integration by parts gives $\EE_0(\hat{\rho}\operatorname{arctanh}\hat{\rho})=\{C_n/(n-1)\}\int_{-1}^1(1-u^2)^{(n-3)/2}\dd u=1/(n-1)$, the boundary terms vanishing because $(1-u^2)^{(n-1)/2}\operatorname{arctanh}u\to0$ at both endpoints.
    Differentiating the expectation now yields $b_n'(0)=\mu_n^2/(n-1)-1$.
    Integration by parts in the $\chi_n$ density gives $\EE(\|x\|)=(n-1)\EE(\|x\|^{-1})$ for $n>1$, and strict Cauchy--Schwarz for this nondegenerate positive variable gives $\EE(\|x\|)\,\EE(\|x\|^{-1})>1$;
    hence $\mu_n^2>n-1$ and $b_n'(0)>0$.
    If the sample-to-population law were independent of $\rho$, its integrable logarithm would have a constant expectation, which is $b_n(\rho)$, a contradiction.
\end{proof}

\subsection{Unknown means and covariance normalization}
\label{supp:unknown-mean}

This subsection records the conventions for an unknown mean stated in Section~\ref{sec:notation};
the corresponding conventions for the fixed laws of Section~\ref{sec:fixed-law} are given in Section~\ref{supp:fixed-law-centering}.
For $n\ge2$ independent $N_d(\mu,\Sigma)$ observations, with $\mu\in\RR^d$ and $\Sigma\succ0$, set $\bar{X}=n^{-1}\sum_{i=1}^n X_i$, $\nu=n-1$, and $\tilde{\Sigma}=\nu^{-1}\sum_{i=1}^n (X_i-\bar{X})(X_i-\bar{X})^\top$.
Then $\nu\tilde{\Sigma}\sim\cW_d(\nu,\Sigma)$.
For a flag power with multiplicity vector $m=(m_1,\ldots,m_d)$, put $r=\max\{i\in[d]:m_i>0\}$ and $c=2\sum_{i=1}^d m_i^2$.
Its sample-to-population law is $\prod_{i=1}^r(\chi^2_{\nu-i+1}/\nu)^{m_i}$ for $\nu\ge r$, with independent factors.
For the precision functionals of Corollary~\ref{cor:bcd-precision}, the product instead runs from $i=1$ to $d$, uses exponents $-\varepsilon m_{d+1-i}$ with $\varepsilon\in\{-1,1\}$, and requires $\nu\ge d$.
With the matching definition $\nu D(\tilde{\Sigma})^2/s_D^2(\tilde{\Sigma})$, the denominator statistic of a self-normalizing polynomial equals $\nu/c$.
If the centered covariance is instead divided by $n$, the same exact products use $n$ in their denominators and retain $\nu-i+1$ as their chi-squared degrees of freedom.

Write the $n\times d$ data matrix as $\1_n\mu^\top+\mathbf{Z}T^\top$, with $TT^\top=\Sigma$ and independent standard normal entries in $\mathbf{Z}$, and let $C_n=I_n-\1_n\1_n^\top/n$ be the centering projection.
In an orthogonal basis of $\RR^n$ with first vector $\1_n/\sqrt{n}$, orthogonal invariance of the Gaussian entries transforms $\mathbf{Z}^\top C_n\mathbf{Z}$ into the Gram matrix of $n-1$ independent standard normal rows.
This proves the Wishart assertion, and the exact laws follow from the preceding proofs with $\nu$ degrees of freedom.
Changing the covariance divisor scales every sample residual variance by $\nu/n$, proving the last assertion.
The covariance block of the per-observation Fisher metric remains $g^{\mathrm{FR}}$ when the mean is unknown, because the mean score is linear in the centered observations and orthogonal to the centered quadratic covariance score;
thus the information calculation uses the sample size $n$, while the exact centered sampling law uses $\nu=n-1$.

\section{Proofs for relative stability, fixed-law sampling and weak denominator sequences}
\label{supp:stability}

This section proves Theorem~\ref{thm:uniform-denominator}, Corollaries~\ref{cor:uniform-normal} and~\ref{cor:denominator-removal} and Proposition~\ref{prop:fixed-law-stability}, records the closure properties and the fourth-moment and centering calculations cited in Section~\ref{sec:stability}, and establishes the limits along first-order weak denominator sequences stated in Section~\ref{sec:weak-minors}.
We use the coordinates, symmetric gradients and sampling conventions of Section~\ref{sec:prelim}, with $D$ and $d$ fixed.
For an integer $K\ge0$, put $\cP_K=\{P\in\RR[\Sym_d]:\deg P\le K\}$, with the Euclidean norm of the coefficients in the monomial basis.
For $H\in\Sym_d$, let $\iota(H)\in\RR^{p_d}$ have diagonal coordinates $H_{ii}$ and upper off-diagonal coordinates $2H_{ij}$, so that $\tr(HS)=\iota(H)^\top\vech(S)$ and $\nabla P=\iota(G_P)$ for every polynomial $P$.
For a nonzero $D\in\RR[\Sym_d]$ and $\Sigma\in\cU_D$, put $J_D(\Sigma)=\Sigma^{1/2}G_D(\Sigma)\Sigma^{1/2}/D(\Sigma)$.
This symmetric matrix is similar to $R_D(\Sigma)$ and satisfies $\kappa_D(\Sigma)=2\|J_D(\Sigma)\|_F^2$.
The delta method and the finite-dimensional central limit theorem are used with their usual conditions \citep[Chapters~2--3]{vandervaart1998}.

\subsection{Deterministic control of normalized polynomials}
\label{supp:normalized}

The following lemma converts a coefficient bound on polynomials normalized at $I_d$ into Taylor bounds near $I_d$ that are uniform over the family;
it is deterministic and serves both Gaussian and fixed-law sampling.

\begin{lemma}[Uniform control of normalized covariance polynomials]
    \label{lem:normalized-polynomial-control}
    Let $\cP\subseteq\cP_K$ with $P(I_d)=1$ for every $P\in\cP$.
    For any fixed $\varrho>0$, $\cP$ is bounded in coefficient norm if and only if $\sup_{P\in\cP}\sup_{\|H-I_d\|_F\le\varrho}|P(H)|<\infty$.
    In that case there exist $0<\epsilon<1/2$ and $C<\infty$ such that, uniformly in $P\in\cP$ and $\|H\|_F\le\epsilon$,
    \begin{equation}
        \label{eq:poly-bounds}
        \begin{split}
            &\frac{1}{2}
            \le
            P(I_d+H)
            \le
            \frac{3}{2}
            ,\quad
            |P(I_d+H)-1-\tr(B_PH)|
            \le
            C\|H\|_F^2
            ,\\
            &B_P=G_P(I_d)
            ,\quad
            \Biggl\|
                \frac{G_P(I_d+H)}{P(I_d+H)}-B_P
            \Biggr\|_F
            \le
            C\|H\|_F
            ,
        \end{split}
    \end{equation}
    and the matrices $B_P$ and all first and second coordinate derivatives of $P$ on this ball are uniformly bounded.
\end{lemma}

\begin{proof}[Proof of Lemma~\ref{lem:normalized-polynomial-control}]
    The supremum norm over a fixed ball of positive radius is a norm on $\cP_K$, because a polynomial vanishing on a nonempty open set vanishes identically, and all norms on a finite-dimensional real vector space are equivalent;
    this proves the equivalence.
    Differentiation and evaluation on a fixed compact ball are continuous linear operations on the coefficient vector, so all derivatives through order two, and in particular $B_P$, are bounded uniformly on $\|H-I_d\|_F\le1/2$.
    The mean-value inequality gives $|P(I_d+H)-1|\le C_1\|H\|_F$, and choosing $\epsilon$ with $C_1\epsilon\le1/2$ gives the bounds on $P(I_d+H)$.
    The integral second-order Taylor remainder gives the quadratic error bound, since coordinate and Frobenius norms are equivalent with constants depending only on $d$.
    Finally, $G_P(I_d+H)/P(I_d+H)-B_P=\{G_P(I_d+H)-B_P\}/P(I_d+H)+B_P\{1-P(I_d+H)\}/P(I_d+H)$, and the derivative bound with $P(I_d+H)\ge1/2$ gives the normalized-gradient bound.
\end{proof}

\subsection{Proof of Theorem~\ref{thm:uniform-denominator} and Corollary~\ref{cor:uniform-normal}}
\label{supp:stability-proof}

We prove the deterministic equivalences (i)$\Leftrightarrow$(ii)$\Leftrightarrow$(iii) first, then (iii)$\Rightarrow$(iv) by a uniform Taylor expansion and (iv)$\Rightarrow$(i) by a pointwise limit;
the same expansion with a Berry--Esseen bound gives Corollary~\ref{cor:uniform-normal}.

\begin{proof}[Proof of Theorem~\ref{thm:uniform-denominator}]
    Let $\cA_D$, $P_A$ and $B_A$ be as in \eqref{eq:normalized-orbit}.
    The matrix $B_A$ is symmetric and $\tr(B_A^2)=\tr\{R_D(AA^\top)^2\}$ by cyclicity of the trace, so
    \begin{equation}
    \label{eq:gradient-norm}
        \kappa_D(AA^\top)
        =
        2\|B_A\|_F^2
        ,\quad
        A\in\cA_D
        .
    \end{equation}
    
    \emph{(i)$\Rightarrow$(ii).}
    Lemma~\ref{lem:bcd-nozero} applies to the finite bound $\bar{\kappa}_D$ and gives constant nonzero sign on the positive cone.
    For a differentiable path $\Sigma(t)$ in that cone,
    \begin{equation*}
        \frac{\dd}{\dd t}\log|D(\Sigma(t))|
        =
        \tr\{J_D(\Sigma(t))\Sigma(t)^{-1/2}\dot{\Sigma}(t)\Sigma(t)^{-1/2}\}
        ,
    \end{equation*}
    so Cauchy--Schwarz gives
    \begin{equation*}
        \Biggl|
            \frac{\dd}{\dd t}\log|D(\Sigma(t))|
        \Biggr|
        \le
        \sqrt{\bar{\kappa}_D/2}\,\|\Sigma(t)^{-1/2}\dot{\Sigma}(t)\Sigma(t)^{-1/2}\|_F
        .
    \end{equation*}
    Given $\Sigma,\Theta\succ0$, take $H=\log(\Sigma^{-1/2}\Theta\Sigma^{-1/2})$ and $\Sigma(t)=\Sigma^{1/2}\exp(tH)\Sigma^{1/2}$ for $0\le t\le1$.
    Then $\Sigma(t)^{-1}\dot{\Sigma}(t)=\Sigma^{-1/2}H\Sigma^{1/2}$, so $\|\Sigma(t)^{-1/2}\dot{\Sigma}(t)\Sigma(t)^{-1/2}\|_F^2=\tr[\{\Sigma(t)^{-1}\dot{\Sigma}(t)\}^2]=\|H\|_F^2$, and integration over $[0,1]$ proves \eqref{eq:log-Lipschitz} with $\Lambda=\sqrt{\bar{\kappa}_D/2}$.
    
    \emph{(ii)$\Rightarrow$(i), with the least constant.}
    At a fixed $\Sigma\succ0$ and $H\in\Sym_d$, apply \eqref{eq:log-Lipschitz} to $\Theta=\Sigma^{1/2}\exp(tH)\Sigma^{1/2}$, divide by $|t|$ and let $t\to0$ to obtain $|\tr\{J_D(\Sigma)H\}|\le\Lambda\|H\|_F$.
    Duality of the Frobenius norm yields $\kappa_D(\Sigma)\le2\Lambda^2$, and the supremum over $\Sigma$ gives $\bar{\kappa}_D\le2\Lambda^2$.
    Together with the preceding direction, the least admissible constant is $\sqrt{\bar{\kappa}_D/2}$.
    
    \emph{(ii)$\Rightarrow$(iii).}
    For $A\in\operatorname{GL}_d(\RR)$, the matrix $O=(AA^\top)^{-1/2}A$ is orthogonal, and for $H\succ0$ the matrix $(AA^\top)^{-1/2}AHA^\top(AA^\top)^{-1/2}=OHO^\top$ has logarithm of Frobenius norm $\|\log H\|_F$.
    Since $D$ has constant sign on the cone, $P_A(H)>0$, and \eqref{eq:log-Lipschitz} gives $|\log P_A(H)|\le\Lambda\|\log H\|_F$.
    On $\|H-I_d\|_F\le1/2$ this bounds $P_A$ uniformly above and away from zero, and Lemma~\ref{lem:normalized-polynomial-control} with $\varrho=1/2$ gives (iii).
    
    \emph{(iii)$\Rightarrow$(i).}
    The linear map $P\mapsto G_P(I_d)$ is continuous on $\cP_K$ for $K=\deg D$, so (iii) bounds $B_A$ uniformly on $\cA_D$, and \eqref{eq:gradient-norm} proves (i).
    
    \emph{(iii)$\Rightarrow$(iv).}
    Let $S_n^\circ$ be as in Section~\ref{sec:stability-criteria}, so that $D(\Sighat)/D(AA^\top)$ has the law of $P_A(S_n^\circ)$ under $\Pr_{AA^\top}$.
    Direct expansion gives $\EE(\|S_n^\circ-I_d\|_F^2)=d(d+1)/n$, so $\|S_n^\circ-I_d\|_F=O_p(n^{-1/2})$.
    By (iii), Lemma~\ref{lem:normalized-polynomial-control} applies to $\{P_A:A\in\cA_D\}$ with some $\epsilon$ and $C$, and on the event $\|S_n^\circ-I_d\|_F\le\epsilon$, whose probability tends to one, the second bound in \eqref{eq:poly-bounds} gives $|\sqrt{n}\{P_A(S_n^\circ)-1\}-\sqrt{n}\tr\{B_A(S_n^\circ-I_d)\}|\le C\sqrt{n}\|S_n^\circ-I_d\|_F^2=o_p(1)$ uniformly in $A$.
    This is \eqref{eq:uniform-expansion}.
    The leading term has mean zero and variance $2\|B_A\|_F^2$, bounded uniformly in $A$, so Chebyshev's inequality and the remainder bound prove \eqref{eq:uniform-tightness}.
    
    \emph{(iv)$\Rightarrow$(i).}
    At each fixed $\Sigma\in\cU_D$, the delta method gives $\sqrt{n}\{D(\Sighat)/D(\Sigma)-1\}\dto N_1\{0,\kappa_D(\Sigma)\}$.
    If $\bar{\kappa}_D=\infty$, then for every $M>0$ one can choose a fixed $\Sigma\in\cU_D$ with $\kappa_D(\Sigma)>0$ such that the limiting two-sided tail probability $2\{1-\Phi(M/\sqrt{\kappa_D(\Sigma)})\}$ is arbitrarily close to one.
    Since the limiting normal law has no atom at $\pm M$, the $\limsup_n$ of the supremum in \eqref{eq:uniform-tightness} equals one for every $M$, contradicting (iv).
\end{proof}

\begin{proof}[Proof of Corollary~\ref{cor:uniform-normal}]
    For $B\in\Sym_d$ and $Z\sim N_d(0,I_d)$, $|Z^\top BZ-\tr B|\le\|B\|_F(\|Z\|^2+\sqrt{d})$, so $\EE(|Z^\top BZ-\tr B|^3)\le C_d\|B\|_F^3$ with $C_d$ depending only on $d$.
    Division by the cube of the standard deviation $\sqrt{2}\|B\|_F$ bounds the standardized third absolute moments uniformly over nonzero $B$, so the Berry--Esseen inequality gives an $O(n^{-1/2})$ bound, uniform in $A$, on the distribution-function distance between the standardized leading term of \eqref{eq:uniform-expansion} and $\Phi$.
    On $\cC_D(v_0)$ the remainder in \eqref{eq:uniform-expansion} divided by $\sqrt{\kappa_D(\Sigma)}\ge\sqrt{v_0}$ is uniformly $o_p(1)$.
    For $\epsilon>0$, the distribution function of the standardized relative error lies between those of the standardized leading term shifted by $\pm\epsilon$, up to the probability that the remainder exceeds $\epsilon$;
    letting $n\to\infty$ and then $\epsilon\downarrow0$, uniform continuity of $\Phi$ proves \eqref{eq:Gaussian-uniform-normal}.
    For the logarithmic version, \eqref{eq:uniform-tightness} and $\log(1+x)=x+O(x^2)$ show that $\sqrt{n}\log\{D(\Sighat)/D(\Sigma)\}$ differs from $\sqrt{n}\{D(\Sighat)/D(\Sigma)-1\}$ by $O_p(n^{-1/2})$ uniformly in $\Sigma$.

    For plug-in studentization, put $B_A(S)=G_{P_A}(S)/P_A(S)$ for $A\in\operatorname{GL}_d(\RR)$ and $S\succ0$, so that $B_A(I_d)=B_A$.
    Congruence invariance gives $\kappa_D(ASA^\top)=2\tr\{(B_A(S)S)^2\}$.
    Lemma~\ref{lem:normalized-polynomial-control} bounds $B_A(S)$ uniformly near $I_d$ and gives a uniform Lipschitz bound for $B_A(S)-B_A$.
    Consequently there are constants $0<\epsilon_0<1/2$ and $0<C<\infty$ such that
    \begin{equation*}
        \sup_{A\in\operatorname{GL}_d(\RR)}
        |\kappa_D(ASA^\top)-\kappa_D(AA^\top)|
        \le
        C\|S-I_d\|_F
        \quad
        (\|S-I_d\|_F\le\epsilon_0)
        .
    \end{equation*}
    Taking $S=S_n^\circ$ and using $\|S_n^\circ-I_d\|_F=O_p(n^{-1/2})$ yields, for every $\epsilon>0$,
    \begin{equation*}
        \sup_{\Sigma\succ0}
        \Pr_\Sigma\bigl\{
            |\kappa_D(\Sighat)-\kappa_D(\Sigma)|>\epsilon
        \bigr\}
        \longrightarrow
        0
        .
    \end{equation*}
    On $\cC_D(v_0)$, the bound $\kappa_D(\Sigma)\ge v_0$ makes this uniform relative consistency.
    Also $D$ is nonconstant because $\cC_D(v_0)$ is nonempty, so $\kappa_D(\Sighat)>0$ almost surely for $n\ge d$ by Section~\ref{supp:sample-domains}.
    Each plug-in statistic equals its population-standardized counterpart multiplied by $\{\kappa_D(\Sigma)/\kappa_D(\Sighat)\}^{1/2}$.
    This factor converges to one uniformly on $\cC_D(v_0)$, and the already proved normal approximations give uniform tightness of both counterparts.
    The same distribution-function comparison therefore proves both plug-in conclusions.

    In particular, suppose $D$ is positive on the cone, take $0<\alpha<1$, and write $z_p=\Phi^{-1}(p)$ for $0<p<1$.
    The interval with endpoints $D(\Sighat)\exp\{\pm z_{1-\alpha/2}\sqrt{\kappa_D(\Sighat)/n}\}$ has coverage converging to $1-\alpha$ uniformly on $\cC_D(v_0)$.
    Its coverage event is that the absolute value of the plug-in standardized logarithmic relative error is at most $z_{1-\alpha/2}$.
\end{proof}

\subsection{Closure properties of the relatively stable class}
\label{supp:closure}

The next lemma records the closure properties and the two bounds on $\kappa_D$ cited in Section~\ref{sec:stability-criteria}.

\begin{lemma}[Closure of the relatively stable class]
\label{lem:bcd-closure}
    Let $m\ge2$ be an integer, let $D_1,\ldots,D_m$ be polynomials positive and relatively stable on $\Smat^d_{++}$, and let $b_1,\ldots,b_m>0$.
    Then $D_1D_2$, $\sum_jb_jD_j$ and $\Sigma\mapsto D_1(A\Sigma A^\top)$ for invertible $A$ are positive and relatively stable, with
    \begin{equation*}
        \sqrt{\bar{\kappa}_{D_1D_2}}
        \le\sqrt{\bar{\kappa}_{D_1}}+\sqrt{\bar{\kappa}_{D_2}}
        ,\quad
        \bar{\kappa}_{\sum_j b_jD_j}\le\max_j\bar{\kappa}_{D_j}
        .
    \end{equation*}
    Every compression determinant satisfies $\kappa_{\Delta_V}\equiv2\dim V$, and every homogeneous polynomial $D$ of degree $K\ge1$ satisfies $\kappa_D\ge2K^2/d$ on $\cU_D$.
\end{lemma}

\begin{proof}[Proof of Lemma~\ref{lem:bcd-closure}]
    On the positive cone, $J_{D_1D_2}=J_{D_1}+J_{D_2}$ and $J_{\sum_jb_jD_j}=\sum_jw_jJ_{D_j}$ with positive weights $w_j=b_jD_j/\sum_ib_iD_i$ summing to one.
    Since $\kappa_D=2\|J_D\|_F^2$, the Frobenius triangle inequality followed by the supremum proves the two bounds.
    Congruence invariance is the congruence rule of Lemma~\ref{lem:trace-main}, and the compression-determinant value is Lemma~\ref{lem:flag-relative-gradient} with a single factor.
    For homogeneous $D$, Euler's identity gives $\tr J_D=K$ wherever $D\ne0$, and Cauchy--Schwarz for the real eigenvalues of $J_D$ gives $2\|J_D\|_F^2\ge2K^2/d$.
\end{proof}

\subsection{Proof of Corollary~\ref{cor:denominator-removal}}
\label{supp:denominator-removal}

The corollary follows from the exact identity of Section~\ref{sec:ratio-removal} and the uniform tightness \eqref{eq:uniform-tightness};
the same argument covers the fixed laws of Section~\ref{sec:fixed-law}.

\begin{proof}[Proof of Corollary~\ref{cor:denominator-removal}]
    Relative stability makes $D$ nonzero at every positive-definite covariance, so $\tau_n$ is defined and $D(\Sighat)\ne0$ almost surely for $n\ge d$.
    Then $\ell_n(\hat{\tau}-\tau_n)-\zeta_n=\zeta_n\{D(\Sigma_n)/D(\Sighat)-1\}$.
    The bracket is $o_p(1)$ by \eqref{eq:uniform-tightness}, whose bound is uniform in the covariance and therefore applies along $\Sigma_n$, and $\zeta_n=O_p(1)$ by assumption.
    Under $\Pr_{A_n,Q}$ for a fixed law $Q$ as in Section~\ref{sec:fixed-law} and any sequence $A_n$, Proposition~\ref{prop:fixed-law-stability} (i) gives the same conclusion with $\Sigma_n=A_nA_n^\top$;
    the event $D(\Sighat)=0$ then has probability tending to zero, so any convention for $\hat{\tau}$ on it leaves the conclusion unchanged.
\end{proof}

\subsection{Proof of Proposition~\ref{prop:fixed-law-stability}, with fourth-moment and centering calculations}
\label{supp:fixed-law}

Lemma~\ref{lem:normalized-polynomial-control} is deterministic, so the Gaussian argument transfers to a fixed law once the second-moment bound \eqref{eq:fixed-moment} replaces its Gaussian counterpart;
the empirical studentization rests on an exact sandwich identity in standardized coordinates.
Throughout, $Q$, $S_n$, $\Sighat=AS_nA^\top$ and $\hat{v}_D$ are as in Section~\ref{sec:fixed-law}, $Z$ has law $Q$, and $\Gamma_Z=\cov_Q\{\vech(ZZ^\top)\}$, possibly singular, so that $v_Q(A)=\iota(B_A)^\top\Gamma_Z\iota(B_A)$.

\begin{proof}[Proof of Proposition~\ref{prop:fixed-law-stability}]
    \emph{Part (i).}
    Independence and the zero mean of $\vech(ZZ^\top-I_d)$ give
    \begin{equation}
        \label{eq:fixed-moment}
        \EE_Q(\|S_n-I_d\|_F^2)
        =
        \frac{1}{n}\EE_Q(\|ZZ^\top-I_d\|_F^2)
        =
        \frac{\EE_Q(\|Z\|^4)-d}{n}
        ,
    \end{equation}
    where the second equality uses $\tr\{(ZZ^\top)^2\}=\|Z\|^4$ and $\EE_Q(\|Z\|^2)=d$.
    If $\|S_n-I_d\|_F\le1/2$ then $S_n\succ0$, so $\Pr_Q(S_n\not\succ0)\le4\{\EE_Q(\|Z\|^4)-d\}/n$, and congruence gives the same bound for $\Sighat$ under every $A$.
    Since $D$ is relatively stable, Theorem~\ref{thm:uniform-denominator} (iii) and Lemma~\ref{lem:normalized-polynomial-control} apply to $\{P_A:A\in\operatorname{GL}_d(\RR)\}$, and \eqref{eq:fixed-moment} gives, exactly as in the Gaussian case,
    \begin{equation}
        \label{eq:fixed-expansion}
        \sqrt{n}\{P_A(S_n)-1\}
        =
        \iota(B_A)^\top\sqrt{n}\vech(S_n-I_d)+r_{n,A}
        ,
        \quad
        \sup_A|r_{n,A}|
        =
        o_p(1)
        ,
    \end{equation}
    on an event of probability tending to one that does not depend on $A$.
    The leading term has mean zero and variance $v_Q(A)\le\|B_A\|_F^2\EE_Q(\|Z\|^4)$, bounded uniformly in $A$, which proves (i).
    
    \emph{Part (ii).}
    The central limit theorem gives $\sqrt{n}\vech(S_n-I_d)\dto N_{p_d}(0,\Gamma_Z)$.
    If the normal approximation of the leading term in \eqref{eq:fixed-expansion}, standardized by $\sqrt{v_Q(A)}$, failed to be uniform on $\cA_Q(v_0)$, there would be $\epsilon>0$, sample sizes $n_k\to\infty$ and $A_k\in\cA_Q(v_0)$ at which its distribution function stays at distance at least $\epsilon$ from $\Phi$;
    by boundedness of $B_A$ we may pass to a subsequence with $B_{A_k}\to B$.
    Tightness of $\sqrt{n}\vech(S_n-I_d)$ makes $\iota(B_{A_k})^\top\sqrt{n_k}\vech(S_{n_k}-I_d)=\iota(B)^\top\sqrt{n_k}\vech(S_{n_k}-I_d)+o_p(1)$, the right-hand side converges to $N_1\{0,\iota(B)^\top\Gamma_Z\iota(B)\}$, and $v_Q(A_k)\to\iota(B)^\top\Gamma_Z\iota(B)\ge v_0$.
    P\'olya's theorem then gives uniform convergence of the standardized distribution functions along the subsequence, a contradiction.
    The remainder in \eqref{eq:fixed-expansion} divided by $\sqrt{v_Q(A)}\ge\sqrt{v_0}$ is uniformly $o_p(1)$, so the sandwich argument in the proof of Corollary~\ref{cor:uniform-normal} transfers the uniform approximation to the relative error, and $\log(1+x)=x+O(x^2)$ transfers it to the logarithmic relative error.
    
    For the empirical variance, let $\hat{\Gamma}_Z=n^{-1}\sum_i\{\vech(Z_iZ_i^\top)-\vech(S_n)\}\{\vech(Z_iZ_i^\top)-\vech(S_n)\}^\top$.
    Entrywise laws of large numbers give $\hat{\Gamma}_Z\pto\Gamma_Z$, each entry requiring moments of order at most four.
    Put $B_A(S)=G_{P_A}(S)/P_A(S)=A^\top G_D(ASA^\top)A/D(ASA^\top)$, so that $B_A(I_d)=B_A$ and, by the third bound in \eqref{eq:poly-bounds}, $\sup_A\|B_A(S_n)-B_A\|_F=O_p(n^{-1/2})$.
    For $H\in\Sym_d$, $\nabla D(\Sighat)^\top\vech(AHA^\top)/D(\Sighat)=\tr\{B_A(S_n)H\}$, and $y_i=\vech(AZ_iZ_i^\top A^\top)$ with $\bar{y}=\vech(AS_nA^\top)$, so \eqref{eq:empirical-scores} gives the exact identity
    \begin{equation}
    \label{eq:sandwich-transport}
        \hat{v}_D=
        \iota\{B_A(S_n)\}^\top\hat{\Gamma}_Z\,\iota\{B_A(S_n)\}
        ,
    \end{equation}
    the empirical variance of $Z_i^\top B_A(S_n)Z_i$, $i\le n$.
    Uniform boundedness of $\iota\{B_A(S_n)\}$, its uniform $O_p(n^{-1/2})$ distance from $\iota(B_A)$ and $\hat{\Gamma}_Z\pto\Gamma_Z$ give $\sup_A\Pr_{A,Q}(|\hat{v}_D-v_Q(A)|>\epsilon)\to0$ for every $\epsilon>0$ by \eqref{eq:sandwich-transport}.
    On $\cA_Q(v_0)$ this is uniform relative consistency, so $\hat{v}_D>0$ with probability tending to one uniformly, and Slutsky's lemma gives the empirically studentized statement, completing (ii).
    
    \emph{Part (iii).}
    The map $H\mapsto\var_Q(Z^\top HZ)=\iota(H)^\top\Gamma_Z\iota(H)$ is continuous and, by assumption, positive on the Frobenius unit sphere, so compactness gives $\lambda_Q>0$ with $\var_Q(Z^\top HZ)\ge\lambda_Q\|H\|_F^2$ for all $H\in\Sym_d$.
    If $\bar{\kappa}_D=\infty$, then by \eqref{eq:gradient-norm} the matrices $B_A$, $A\in\cA_D$, have unbounded norm, so $v_Q(A)\ge\lambda_Q\|B_A\|_F^2$ is unbounded on $\cA_D$.
    At each fixed $A\in\cA_D$ the delta method gives $\sqrt{n}\{P_A(S_n)-1\}\dto N_1\{0,v_Q(A)\}$, and the argument for (iv)$\Rightarrow$(i) in Theorem~\ref{thm:uniform-denominator} contradicts uniform tightness.
\end{proof}

The nondegeneracy assumption in (iii) cannot be dropped.

\begin{example}[Degenerate quadratic scores]
    \label{ex:fixed-law-degeneracy}
    In dimension one, let $Q$ put probability $1/2$ at each of $-1$ and $1$.
    For $X=AZ$ with $A\ne0$, the known-mean sample variance is identically $A^2$, so every polynomial has zero relative error on its nonzero population domain, whereas $D(\sigma)=\sigma-1$ has $\kappa_D(\sigma)=2\sigma^2/(\sigma-1)^2$, unbounded for $\sigma>0$, $\sigma\ne1$.
    The quadratic score $Z^2$ is constant, so $\var_Q(Z^\top HZ)=0$ for every $H$.
\end{example}

\subsubsection{Variance under isotropic fourth moments}
\label{supp:fourth-moment}
We say that $Q$ has \emph{isotropic fourth moments} with parameter $\eta$ if $\EE_Q\{(Z^\top HZ)(Z^\top KZ)\}=\eta\{\tr(H)\tr(K)+2\tr(HK)\}$ for all $H,K\in\Sym_d$.
Taking $H=K=B\in\Sym_d$ gives
\begin{equation}
    \label{eq:isotropic-variance}
    \var_Q(Z^\top BZ)
    =
    2\eta\tr(B^2)+(\eta-1)(\tr B)^2
    ,
\end{equation}
and taking $H=K=I_d$ gives $\eta=\EE_Q(\|Z\|^4)/\{d(d+2)\}\ge d/(d+2)$, with equality if and only if $\|Z\|^2=d$ almost surely.
For a flag power, $B_A$ has eigenvalues $m_1,\ldots,m_d$, so $v_Q(A)=2\eta\sum_im_i^2+(\eta-1)(\sum_im_i)^2$ for every $A$.
At $\eta=d/(d+2)$, a matrix $B_A$ that is a scalar multiple of $I_d$ gives $v_Q(A)=0$, which is why Proposition~\ref{prop:fixed-law-stability} (ii) is stated on $\cA_Q(v_0)$.
Elliptical laws with finite fourth moment, standardized to covariance $I_d$, have isotropic fourth moments.

\subsubsection{Unknown means}
\label{supp:fixed-law-centering}
For observations $\mu+AZ_i$, the divisor-$n$ centered covariance is $A(S_n-\bar{Z}\bar{Z}^\top)A^\top$ with $\bar{Z}=n^{-1}\sum_iZ_i=O_p(n^{-1/2})$ by finite second moments, so the standardized change is $\bar{Z}\bar{Z}^\top=O_p(n^{-1})$, and a change to divisor $n-1$ adds another $O_p(n^{-1})$ term.
The uniform derivative bounds of Lemma~\ref{lem:normalized-polynomial-control} therefore leave \eqref{eq:fixed-expansion} unchanged.
Expanding the centered empirical fourth moments in $\bar{Z}$ bounds their change by a fixed sum of terms $\|\bar{Z}\|^jn^{-1}\sum_i(1+\|Z_i\|^{4-j})$, $1\le j\le4$, which are $o_p(1)$, so the empirical score covariance still converges to $\Gamma_Z$.
These statements are uniform in $A$ because every quantity is obtained from $S_n$ and $\bar{Z}$ through the fixed congruence by $A$;
they extend the unknown-mean conventions of Section~\ref{supp:unknown-mean} to fixed laws.

\subsection{Rotationally invariant and isotropic laws}
\label{supp:rotational-sampling}

This subsection proves that flag powers retain a relative pivot under rotationally invariant $Q$, as cited in Section~\ref{sec:fixed-law}, explains the common denominator law within each design of Section~\ref{sec:sampling-laws}, and shows that under isotropic fourth moments constancy of $v_Q$ characterizes flag powers among homogeneous polynomials but not among all polynomials, as cited in Section~\ref{sec:discussion}.

\begin{lemma}[Flag pivots under rotationally invariant sampling]
\label{lem:rotational-flag-pivot}
    Let $Q$ be as in Section~\ref{sec:fixed-law} and suppose that $OZ$ has law $Q$ for every orthogonal $O$ whenever $Z$ has law $Q$.
    If $D$ is a flag power, then for every invertible $A$ and every $n\ge1$,
    \begin{equation}
        \label{eq:rotational-flag-pivot}
        \frac{D(AS_nA^\top)}{D(AA^\top)}
        \stackrel{d}{=}
        \frac{D(S_n)}{D(I_d)}
        .
    \end{equation}
\end{lemma}

The common law in \eqref{eq:rotational-flag-pivot} depends on $Q$.

\begin{proof}[Proof of Lemma~\ref{lem:rotational-flag-pivot}]
    Since $Q$ is invariant under orthogonal transformations, we may work in orthonormal coordinates adapted to the flag.
    Let $T$ be the lower-triangular Cholesky factor of $AA^\top$ and $O=T^{-1}A$, which is orthogonal.
    The relative invariance $D(THT^\top)=\{\prod_iT_{ii}^{2m_i}\}D(H)$ of Section~\ref{sec:classification} and $OS_nO^\top\stackrel{d}{=}S_n$ give $D(AS_nA^\top)=D(TOS_nO^\top T^\top)\stackrel{d}{=}\{\prod_iT_{ii}^{2m_i}\}D(S_n)$, while $D(AA^\top)=\{\prod_iT_{ii}^{2m_i}\}D(I_d)$, so the multiplier cancels.
\end{proof}

For an arbitrary $Q$ as in Section~\ref{sec:fixed-law}, the same cancellation holds whenever $A$ is nonsingular and lower triangular in adapted coordinates, directly by the relative invariance of $D$.
This argument does not require rotational invariance of $Q$ or positive diagonal entries of $A$.
The loading $A_v$ of the mediation design in \eqref{eq:loading} is lower triangular in the coordinates $(X,M,Y)$, to which the flag of $D_{\mathrm{med}}$ is adapted, so within each of the designs (G), (E) and (D) of Section~\ref{sec:sampling-laws} the sample-to-population ratio of $D_{\mathrm{med}}$ has the same law for every $v$;
under (E) that common law differs from the Gaussian reference, which produces the denominator undercoverage in Table~\ref{tab:denominator}.

Under isotropic fourth moments, a homogeneous polynomial $D$ of degree $K\ge1$ satisfies $v_Q(A)=\eta\kappa_D(AA^\top)+(\eta-1)K^2$ on $\cA_D$, by \eqref{eq:isotropic-variance} and Euler's identity $\tr B_A=K$.
Since $\eta>0$, constancy of $v_Q$ on $\cA_D$ is equivalent to constancy of $\kappa_D$ on $\cU_D$.
The constant value of $\kappa_D$ is positive by Lemma~\ref{lem:bcd-closure}, so Theorem~\ref{thm:sn-pivot-classification} makes $D$ a flag power.
The constant value of $v_Q$ may be zero, as noted after \eqref{eq:isotropic-variance}.
Among homogeneous polynomials, constancy of the actual first-order relative variance under an isotropic law is therefore equivalent to self-normalization.
Homogeneity cannot be dropped.

\begin{example}[A nonhomogeneous first-order solution under spherical sampling]
\label{ex:spherical-nonhomogeneous}
    Let $d\ge2$, let $Q$ be uniform on the sphere of radius $\sqrt{d}$, and put $D(\Sigma)=\sigma_{11}(1+\det\Sigma)$.
    This positive nonhomogeneous polynomial has $v_Q(A)=2(d-1)/(d+2)$ for every invertible $A$, it is not a flag power, and $D(AS_nA^\top)/D(AA^\top)$ has a covariance-dependent law for every $n\ge1$.
\end{example}

\begin{proof}[Proof of Example~\ref{ex:spherical-nonhomogeneous}]
    Put $u=A^\top e_1/\|A^\top e_1\|$ and $\omega=\det(AA^\top)/\{1+\det(AA^\top)\}$.
    Logarithmic differentiation gives $B_A=uu^\top+\omega I_d$, and $\|Z\|^2=d$ almost surely, so $v_Q(A)=\var_Q\{(u^\top Z)^2\}=2\eta+(\eta-1)=2(d-1)/(d+2)$ by \eqref{eq:isotropic-variance} with $\eta=d/(d+2)$.
    The nonzero homogeneous components of $D$ have degrees $1$ and $d+1$, so $D$ is not a flag power.
    For $A=\sqrt{t}I_d$ with $t>0$, the sample-to-population ratio is $S_{n,11}(1+t^d\det S_n)/(1+t^d)$.
    Since $\tr S_n=d$, $\det S_n\le1$ with equality only if $S_n=I_d$;
    conditionally on the first $n-1$ observations this fixes $Z_nZ_n^\top$ and permits at most two values of $Z_n$, to which the continuous spherical law assigns probability zero.
    As $S_{n,11}>0$ almost surely and $S_{n,11}\le d$, $\EE_Q(S_{n,11}\det S_n)<\EE_Q(S_{n,11})=1$, so the expectation of the ratio, $\{1+t^d\EE_Q(S_{n,11}\det S_n)\}/(1+t^d)$, is a strictly decreasing function of $t$, and covariance independence fails at every $n$.
\end{proof}

\subsection{Limits along first-order weak denominator sequences}
\label{supp:local-ratio}

This subsection proves the denominator limits stated after Definition~\ref{def:bcd-weak-sequence} and the accompanying ratio limit.

\begin{proposition}[Local limits of the denominator statistic and the ratio]
\label{prop:transfer}
    In the known-mean Gaussian experiment, let $D,N\in\RR[\Sym_d]$, let $\Sigma_n$ be a first-order weak denominator sequence for $D$ with limits $\Sigma_\ast$ and $\delta$, and put $s_{D,\ast}=s_D(\Sigma_\ast)$ and $\mu=\delta/s_{D,\ast}$.
    Then $\sqrt{n}D(\Sighat)\dto N_1(\delta,s_{D,\ast}^2)$ and \eqref{eq:weak-diagnostic} holds.
    If moreover $\tau_n=N(\Sigma_n)/D(\Sigma_n)\to\tau_\ast\in\RR$, put $g_\ast=\nabla N(\Sigma_\ast)-\tau_\ast\nabla D(\Sigma_\ast)$ and suppose $s_{g,\ast}=\{g_\ast^\top\Gamma(\Sigma_\ast)g_\ast\}^{1/2}>0$.
    Then, jointly with \eqref{eq:weak-diagnostic},
    \begin{equation}
        \label{eq:local-ratio-limit}
        \hat{\tau}-\tau_n
        \dto
        \frac{\omega_\ast\{r_\ast Z+(1-r_\ast^2)^{1/2}Z_0\}}{\mu+Z}
        ,\quad
        \omega_\ast
        =
        \frac{s_{g,\ast}}{s_{D,\ast}}
        ,\quad
        r_\ast
        =
        \frac{g_\ast^\top\Gamma(\Sigma_\ast)\nabla D(\Sigma_\ast)}{s_{g,\ast}s_{D,\ast}}
        ,
    \end{equation}
    where $Z$ is the standard normal variable in \eqref{eq:weak-diagnostic} and $Z_0$ is an independent standard normal variable.
\end{proposition}

Thus the denominator statistic has a noncentral $\chi^2_1$ limit, while the ratio error converges to the ratio of correlated normals that underlies Fieller's interval and weak-instrument asymptotics \citep{fieller1954,staigerstock1997}.

\begin{proof}[Proof of Proposition~\ref{prop:transfer}]
    Let $X_{n1},\ldots,X_{nn}$ be the observations at $\Sigma_n$, put $Y_{ni}=\vech(X_{ni}X_{ni}^\top-\Sigma_n)$ and $Y_n=n^{-1/2}\sum_iY_{ni}=\sqrt{n}\vech(\Sighat-\Sigma_n)$.
    Since $X_{ni}$ is $\Sigma_n^{1/2}$ times a standard normal vector and $\Sigma_n$ is bounded, $\EE(\|Y_{ni}\|^4)$ is bounded uniformly in $n$ by a standard Gaussian eighth moment, and $\cov(Y_{ni})=\Gamma(\Sigma_n)\to\Gamma(\Sigma_\ast)$.
    For a scalar projection of $Y_n$ with positive limiting variance, Lyapunov's condition follows from the uniform fourth-moment bound, and a projection with zero limiting variance converges to zero in $L^2$.
    The Cram\'er--Wold device therefore gives $Y_n\dto N_{p_d}\{0,\Gamma(\Sigma_\ast)\}$, whether or not $\Sigma_\ast$ or $\Gamma(\Sigma_\ast)$ is singular.
    Since the second derivatives of $D$ are bounded on a fixed neighborhood of $\Sigma_\ast$,
    \begin{equation*}
        \sqrt{n}D(\Sighat)
        =
        \sqrt{n}D(\Sigma_n)+\nabla D(\Sigma_n)^\top Y_n+o_p(1)
        \dto
        \delta+s_{D,\ast}Z
        =
        s_{D,\ast}(\mu+Z)
        ,
    \end{equation*}
    and $s_D^2(\Sighat)\pto s_{D,\ast}^2>0$ by continuity, so $\hat{F}_D=\{\sqrt{n}D(\Sighat)\}^2/s_D^2(\Sighat)\dto(\mu+Z)^2$, which is \eqref{eq:weak-diagnostic}.
    Under the additional assumptions, $\tau_n$ is bounded and the same expansion gives $\sqrt{n}\{N(\Sighat)-\tau_nD(\Sighat)\}=g_n^\top Y_n+o_p(1)$ with $g_n=\nabla N(\Sigma_n)-\tau_n\nabla D(\Sigma_n)\to g_\ast$.
    The pair $(g_n^\top Y_n,\nabla D(\Sigma_n)^\top Y_n)$ converges jointly to a centered normal vector with variances $s_{g,\ast}^2$, $s_{D,\ast}^2$ and covariance $r_\ast s_{g,\ast}s_{D,\ast}$, which is the law of $(s_{g,\ast}\{r_\ast Z+(1-r_\ast^2)^{1/2}Z_0\},s_{D,\ast}Z)$.
    Since $\hat{\tau}-\tau_n=\sqrt{n}\{N(\Sighat)-\tau_nD(\Sighat)\}/\{\sqrt{n}D(\Sighat)\}$ and the limiting denominator $s_{D,\ast}(\mu+Z)$ vanishes with probability zero, division and the continuous mapping theorem prove \eqref{eq:local-ratio-limit}.
\end{proof}

\section{Proofs for covariance minors and denominator examples}
\label{supp:minors}

This section proves the covariance-minor and positive-polynomial assertions of Section~\ref{sec:stability} and verifies the reduced denominators in Sections~\ref{sec:recursive}--\ref{sec:napkin}.
We use the conventions of Section~\ref{supp:formulation} and write $\Delta_I=\det\Sigma_I$ for a principal minor, with $\Delta_\emptyset=1$.
The covariance minor $M_{I,J}=\det\Sigma_{I,J}$ uses increasing row and column orders.
Any signs introduced by reordering are retained until they cancel or are absorbed into a nonzero scalar factor.

\subsection{Minor products and positive boundary zeros}
\label{supp:minor-products}

We prove the minor-product criterion and the boundary degeneracy of positive polynomials, starting with the irreducibility needed to separate the zero sets of distinct minors.

\begin{lemma}[Irreducibility of a covariance minor]
\label{lem:bcd-minor-irreducible}
    For equal-size nonempty index sets $I,J\subseteq[d]$, the polynomial $M_{I,J}$ is irreducible over both $\RR$ and $\CC$.
\end{lemma}

\begin{proof}[Proof of Lemma~\ref{lem:bcd-minor-irreducible}]
    Work over $\KK=\RR$ or $\CC$.
    The principal case is Lemma~\ref{lem:symdet-irreducible}.
    Otherwise put $q=|I\cap J|$ and $s=|I\setminus J|=|J\setminus I|\ge1$.
    Reordering the common indices first gives, up to sign,
    \begin{equation*}
        f
        =
        \det\begin{pmatrix}H&C\\E&T\end{pmatrix}
        ,
    \end{equation*}
    where $H$ is symmetric of size $q$, $C$ is $q\times s$, $E$ is $s\times q$ and $T$ is $s\times s$.
    Apart from the symmetry of $H$, their entries are distinct indeterminates because $I\setminus J$ and $J\setminus I$ are disjoint.
    Matrix arguments in polynomial-ring notation below denote these indeterminates.
    
    The determinant of a generic square matrix is irreducible over any field of characteristic zero.
    To see this, use induction on its size and write the determinant as $(\det A)z-v^\top\operatorname{adj}(A)u$, where $A$ is the smaller generic square block, $u,v$ are independent vectors of indeterminates and $z$ is a further indeterminate.
    The induction starts with size one.
    A nonzero coefficient of $v^\top\operatorname{adj}(A)u$ is a cofactor of $A$, of smaller degree than $\det A$.
    Thus the two coefficients in $z$ are coprime, and Gauss' lemma proves the induction step as in Lemma~\ref{lem:symdet-irreducible}.
    This proves the result for $q=0$.
    
    For $q\ge1$, put $\delta=\det H$.
    Work in the unique factorization domain $\cR=\KK[H,C,E,\delta^{-1}]$, obtained by allowing powers of $\delta$ in denominators.
    The map $T\mapsto T-EH^{-1}C$ is an invertible change of its polynomial variables.
    The Schur-complement identity gives $f=\delta\det(T-EH^{-1}C)$.
    Over the fraction field of $\cR$, the translated determinant is irreducible by the preceding argument.
    Its coefficient of $\prod_{i=1}^sT_{ii}$ is one, so it is primitive over $\cR$.
    Gauss' lemma gives irreducibility in $\cR[T]$.
    Any nonconstant factor of $f$ that becomes invertible after adjoining $\delta^{-1}$ must divide a power of $\delta$.
    Since $\delta$ is irreducible, such a factor is a scalar multiple of a power of $\delta$.
    However, $\delta\nmid f$.
    Indeed, set $H=\diag(I_{q-1},0)$, $C_{q1}=E_{1q}=1$, all other entries of $C,E$ to zero, and $T=\diag(0,I_{s-1})$, omitting empty identity blocks.
    This gives $\delta=0$ and $f=-1$.
    Hence $f$ is irreducible before adjoining $\delta^{-1}$, and adjoining unused covariance coordinates preserves irreducibility.
\end{proof}

\subsubsection{Proof of Proposition~\ref{prop:bcd-minor-products}}
\label{supp:minor-products-proof}

\begin{proof}[Proof of Proposition~\ref{prop:bcd-minor-products}]
    \emph{Relative stability.}
    If every factor is principal, Lemma~\ref{lem:bcd-closure} gives relative stability.
    Every nonprincipal minor vanishes at $I_d$, so a product containing one is not relatively stable by Lemma~\ref{lem:bcd-nozero}.
    
    \emph{A smooth positive-definite zero.}
    Let $f=M_{I,J}$ be nonprincipal, put $s=|I\setminus J|$, and pair $I\setminus J=\{u_1,\ldots,u_s\}$ with $J\setminus I=\{v_1,\ldots,v_s\}$.
    Choose $0<\rho_0<1$ and set
    \begin{equation*}
        \Sigma_0
        =
        I_d+\rho_0\sum_{h=1}^{s-1}
        (e_{u_h}e_{v_h}^\top+e_{v_h}e_{u_h}^\top)
        .
    \end{equation*}
    The paired coordinates are disjoint, and the eigenvalues are $1\pm\rho_0$ on each paired block and one elsewhere, so $\Sigma_0\succ0$.
    With $E_s=e_{u_s}e_{v_s}^\top+e_{v_s}e_{u_s}^\top$, direct evaluation gives $f(\Sigma_0+xE_s)=\epsilon\rho_0^{s-1}x$ for a sign $\epsilon\in\{-1,1\}$.
    This includes $s=1$, when the sum is empty.
    Thus $f(\Sigma_0)=0$ and $\dd f_{\Sigma_0}(E_s)\ne0$, so the zero is smooth, meaning $\nabla f(\Sigma_0)\ne0$.
    Equation~\eqref{eq:sD} also gives $s_f(\Sigma_0)>0$.
    
    \emph{A factor of multiplicity one.}
    Suppose $f$ has exponent one in $D$.
    The entry $x=\sigma_{u_sv_s}$ occurs only once in the determinant, so $f=a(y)x+b(y)$, where $y$ collects the remaining covariance coordinates.
    Near $\Sigma_0$, $a(y)\ne0$ and $f=0$ is the graph $x=-b(y)/a(y)$ inside the positive-definite cone.
    If another nonassociate irreducible factor $g$ vanished on an open subset of this graph, substitution and clearing powers of $a$ would give the zero polynomial in $y$.
    It follows that $f\mid g$ in $\RR(y)[x]$, where $\RR(y)$ is the field of rational functions in $y$.
    Since $f$ is irreducible and depends on $x$, it is primitive as a polynomial in $x$.
    Gauss' lemma would give $f\mid g$ in the original ring, contradicting Lemma~\ref{lem:bcd-minor-irreducible} and nonassociation.
    
    Each other factor therefore has a closed zero set with empty interior in the graph.
    Their finite union has empty interior, so choose $\Sigma_1$ on the graph where all the other factors are nonzero.
    The graph consists of smooth zeros of $f$, and the product rule gives $\nabla D(\Sigma_1)\ne0$.
    Choose $H\in\Sym_d$ with $b_0=\dd D_{\Sigma_1}(H)\ne0$.
    For all sufficiently large $n$, $\Sigma_n=\Sigma_1+n^{-1/2}H$ is positive definite, $D(\Sigma_n)\ne0$ and $\sqrt{n}D(\Sigma_n)\to b_0$.
    Also $s_D(\Sigma_1)>0$ by \eqref{eq:sD}.
    Assign any point of $\cU_D$ to the finitely many remaining sample sizes to obtain a first-order weak denominator sequence.
    
    \emph{The converse.}
    Suppose every nonprincipal factor has exponent at least two, and let $\Sigma_\ast\succeq0$ satisfy $D(\Sigma_\ast)=0$.
    If two distinct factors vanish, or a vanishing factor has exponent at least two, the product rule gives $G_D(\Sigma_\ast)=0$.
    Otherwise the only vanishing factor is a principal minor $M=\Delta_J$ with exponent one, and $G_D(\Sigma_\ast)$ is a scalar multiple of $G_M(\Sigma_\ast)$.
    The polynomial identity $s_M^2=2|J|M^2$ then gives $s_D(\Sigma_\ast)=0$.
    Thus every positive-semidefinite zero has zero first-order variance, which excludes a first-order weak denominator sequence.
\end{proof}

\subsubsection{Proof of Proposition~\ref{prop:positive-boundary-degeneracy}}
\label{supp:positive-boundary-proof}

\begin{proof}[Proof of Proposition~\ref{prop:positive-boundary-degeneracy}]
    For $\Sigma_\ast\succeq0$ and $\epsilon>0$, $\Sigma_\ast+\epsilon I_d\succ0$, so continuity makes $D$ nonnegative on the closed cone.
    Suppose $D(\Sigma_\ast)=0$ and $\rank\Sigma_\ast=r>0$.
    Choose $A\in\RR^{d\times r}$ of full column rank with $\Sigma_\ast=AA^\top$.
    The polynomial $P(H)=D(AHA^\top)$ is nonnegative on $\Smat^r_{++}$ and vanishes at $I_r$.
    Since $I_r$ is an interior minimum, the chain rule gives $G_P(I_r)=A^\top G_D(\Sigma_\ast)A=0$.
    Cyclicity of trace yields $s_D^2(\Sigma_\ast)=2\tr\{(A^\top G_D(\Sigma_\ast)A)^2\}=0$.
    If $r=0$, the conclusion follows directly from \eqref{eq:sD}.
    In a first-order weak denominator sequence, continuity gives $D(\Sigma_\ast)=0$, whereas the definition requires $s_D(\Sigma_\ast)>0$.
    The variance identity therefore excludes such a sequence.
\end{proof}

\subsection{Relative variances of minors and positive polynomials}
\label{supp:volume-overlap}

We derive the variance formulas used in Table~\ref{tab:catalogue} and verify the positive-polynomial examples cited in Sections~\ref{sec:stability-criteria} and~\ref{sec:weak-minors}.

\begin{lemma}[Relative variance of two compression determinants]
    \label{lem:two-subspace}
    For nonzero subspaces $V,W\subseteq\RR^d$ and $\Sigma\succ0$, let $\Pi_V(\Sigma)$ and $\Pi_W(\Sigma)$ be the orthogonal projections onto $\Sigma^{1/2}V$ and $\Sigma^{1/2}W$.
    For positive integers $h,k$,
    \begin{equation}
        \label{eq:two-volumes}
        \kappa_{\Delta_V^h\Delta_W^k}(\Sigma)
        =
        2\bigl\{
            h^2\dim V+k^2\dim W
            +2hk\tr\{\Pi_V(\Sigma)\Pi_W(\Sigma)\}
        \bigr\}
        .
    \end{equation}
    This is constant over $\Sigma\succ0$ if and only if $V\subseteq W$ or $W\subseteq V$.
    For $C\subseteq[d]$ and distinct $a,b\in[d]\setminus C$, $\kappa_{\Delta_{C\cup\{a\}}\Delta_{C\cup\{b\}}}=8|C|+4+4\rho_{ab\cdot C}^2$.
    For a nonempty predictor set $\cJ$ and response coordinate $Y\notin\cJ$, $\kappa_{\sigma_{YY}\Delta_{\cJ}}=2(|\cJ|+1)+4R^2_{Y\cdot\cJ}$.
\end{lemma}

\begin{proof}[Proof of Lemma~\ref{lem:two-subspace}]
    For an orthonormal basis matrix $B_V$ of $V$,
    \begin{equation*}
        \frac{\Sigma^{1/2}G_{\Delta_V}\Sigma^{1/2}}{\Delta_V}
        =
        \Sigma^{1/2}B_V(B_V^\top\Sigma B_V)^{-1}B_V^\top\Sigma^{1/2}
        =
        \Pi_V(\Sigma)
        .
    \end{equation*}
    The product rule gives $J_{\Delta_V^h\Delta_W^k}=h\Pi_V(\Sigma)+k\Pi_W(\Sigma)$ in the notation of Section~\ref{supp:stability}.
    Squaring and taking the trace proves \eqref{eq:two-volumes}.
    If the spaces are nested, the projection overlap equals the smaller dimension.
    Conversely, constant relative variance makes $\Delta_V^h\Delta_W^k$ self-normalizing.
    Theorem~\ref{thm:sn-pivot-classification} and uniqueness of its irreducible compression-determinant factors in Lemma~\ref{lem:symdet-irreducible} then force $V$ and $W$ to be nested.
    
    For the coordinate spaces indexed by $C\cup\{a\}$ and $C\cup\{b\}$, the common transformed subspace is $\Sigma^{1/2}\operatorname{span}\{e_i:i\in C\}$.
    Projecting away this subspace leaves two unit directions whose inner product is $\rho_{ab\cdot C}$, by the residual-covariance formula.
    The overlap trace is therefore $|C|+\rho_{ab\cdot C}^2$.
    For the line indexed by $Y$ and the predictor space indexed by $\cJ$, the overlap is $\Sigma_{Y,\cJ}\Sigma_{\cJ}^{-1}\Sigma_{\cJ,Y}/\sigma_{YY}=R^2_{Y\cdot\cJ}$.
    Substitution in \eqref{eq:two-volumes} proves the two specializations.
\end{proof}

The next proposition gives the denominator statistic for an almost-principal minor and relates it to a one-coefficient regression $F$ statistic.

\begin{proposition}[Variance and denominator statistic of an almost-principal minor]
    \label{prop:bcd-partial-minor}
    Let $C\subseteq[d]$ have cardinality $q\ge0$, and let $a,b\in[d]\setminus C$ be distinct.
    Put $M=M_{C\cup\{a\},C\cup\{b\}}$, $\Delta_a=\Delta_{C\cup\{a\}}$ and $\Delta_b=\Delta_{C\cup\{b\}}$.
    Then
    \begin{equation}
        \label{eq:partial-variance}
        s_M^2
        =
        \Delta_a\Delta_b+(2q+1)M^2
    \end{equation}
    is a polynomial identity on $\Sym_d$.
    At every $\Sighat\succ0$, $\hat{F}_M=n\hat{\rho}_{ab\cdot C}^{\,2}/\{1+(2q+1)\hat{\rho}_{ab\cdot C}^{\,2}\}$, where $\hat{\rho}_{ab\cdot C}$ is the partial correlation computed from $\Sighat$.
\end{proposition}

For the usual test of the coefficient of $X_b$ in the regression of $X_a$ on $(X_b,X_C)$, let $\nu>0$ be the residual degrees of freedom.
Its statistic is $F_{\mathrm{reg}}=\nu\hat{\rho}_{ab\cdot C}^{\,2}/(1-\hat{\rho}_{ab\cdot C}^{\,2})$, so
\begin{equation*}
    \hat{F}_M
    =
    \frac{nF_{\mathrm{reg}}}{\nu+(2q+2)F_{\mathrm{reg}}}
    .
\end{equation*}
Here $\nu=n-q-1$ without an intercept.
With an intercept, $\nu=n-q-2$ and the partial correlation and denominator statistic use the empirically centered covariance.

\begin{proof}[Proof of Proposition~\ref{prop:bcd-partial-minor}]
    Let $X\sim N_d(0,\Sigma)$ with $\Sigma\succ0$, and put $\delta=\Delta_C$.
    For $q>0$, define $R_a=X_a-\Sigma_{a,C}\Sigma_C^{-1}X_C$ and $R_b=X_b-\Sigma_{b,C}\Sigma_C^{-1}X_C$.
    For $q=0$, set $R_a=X_a$, $R_b=X_b$ and $\delta=1$.
    Write $u=\var(R_a)$, $v=\var(R_b)$ and $z=\cov(R_a,R_b)$.
    Schur complementation gives $M=\pm\delta z$, $\Delta_a=\delta u$ and $\Delta_b=\delta v$, with the sign determined by the index orders.
    
    For $q>0$, the influence functions in the convention of \eqref{eq:S1-linear-score} are
    \begin{equation*}
        \phi_\delta
        =
        \delta(X_C^\top\Sigma_C^{-1}X_C-q)
        ,\quad
        \phi_z
        =
        R_aR_b-z
        .
    \end{equation*}
    For $q=0$, use $\phi_\delta=0$ and the same formula for $\phi_z$.
    The determinant derivative gives $\phi_\delta$.
    In differentiating the residual covariance, the coefficient-derivative terms vanish because both residuals are orthogonal to the regressors.
    Under Gaussian sampling the residual pair is independent of $X_C$, so the two scores are uncorrelated, with variances $2q\delta^2$ and $uv+z^2$.
    The product-rule score $\pm(z\phi_\delta+\delta\phi_z)$ has variance $\delta^2uv+(2q+1)\delta^2z^2$.
    This proves \eqref{eq:partial-variance} on the positive-definite cone and hence on $\Sym_d$.
    The sample formula follows from $\rho_{ab\cdot C}^2=M^2/(\Delta_a\Delta_b)$, and substitution of $F_{\mathrm{reg}}$ gives the displayed regression relation.
\end{proof}

\subsubsection{Positive polynomials and relative stability}
\label{supp:positive-examples}

The following examples separate positivity from relative stability and show that positive sums of products of compression determinants do not exhaust the relatively stable class.

\begin{example}[Positive polynomials and relative stability]
    \label{ex:bcd-positive-unstable}
    In dimension two, $D=\sigma_{11}^2+\sigma_{12}^2$ is positive on the positive-definite cone, but $\Sigma_t=\left(\begin{smallmatrix}t^2&t^2\\t^2&1\end{smallmatrix}\right)$, $0<t<1$, satisfies $D(\Sigma_t)=2t^4$ and $\kappa_D(\Sigma_t)=7+t^{-2}$.
    For every $d\ge1$ and integer $r\ge1$, $\tr(\Sigma^r)$ is relatively stable with first-order relative variance at most $2r^2$.
    In dimension two, $\tr(\Sigma^2)$ is not a finite positive linear combination of products of compression determinants.
\end{example}

\begin{proof}[Proof of Example~\ref{ex:bcd-positive-unstable}]
    Write $x=\sigma_{11}$, $y=\sigma_{12}$ and $z=\sigma_{22}$.
    For $D=x^2+y^2$, $G_D=\left(\begin{smallmatrix}2x&y\\y&0\end{smallmatrix}\right)$, so \eqref{eq:sD} gives $s_D^2=8x^4+16x^2y^2+4y^4+4xy^2z$.
    Substitution gives the stated values, and $\det\Sigma_t=t^2(1-t^2)>0$.
    
    For $D=\tr(\Sigma^r)$, the gradient is $G_D=r\Sigma^{r-1}$.
    If $\lambda_1,\ldots,\lambda_d>0$ are the eigenvalues of $\Sigma$, then
    \begin{equation*}
        \kappa_D(\Sigma)
        =
        2r^2\frac{\sum_i\lambda_i^{2r}}{(\sum_i\lambda_i^r)^2}
        \le
        2r^2
        .
    \end{equation*}
    
    Suppose $\tr(\Sigma^2)$ in dimension two were a finite positive linear combination of products of compression determinants.
    Evaluation on $tI_2$, $t>0$, makes each summand a positive multiple of a nonnegative integer power of $t$.
    Comparison with $2t^2$ forces every summand to have degree two.
    Each is therefore a positive multiple of $\det\Sigma$ or of $(v^\top\Sigma v)(w^\top\Sigma w)$ with nonzero $v,w\in\RR^2$.
    On $\Sigma=\diag(x,z)$, the $xz$ coefficient is positive for a determinant term and equals $v_1^2w_2^2+v_2^2w_1^2\ge0$ for a product term.
    Since the target is $x^2+z^2$ there, every such coefficient must vanish.
    Thus there is no determinant term, and both vectors in each product lie on the same coordinate axis.
    These products cannot yield the term $2\sigma_{12}^2$ in $\tr(\Sigma^2)$, a contradiction.
\end{proof}

\subsection{Recursive regression, path products and static interventions}
\label{supp:recursive}

We prove Propositions~\ref{prop:bcd-regression-source} and~\ref{prop:path-products}, then identify the denominators of static-intervention contrasts in the same full-predecessor order.

\begin{proof}[Proof of Proposition~\ref{prop:bcd-regression-source}]
    Cramer's rule writes $\beta_{ji}=K_{ji}/\Delta_{j-1}$, where $K_{ji}$ is the determinant obtained by replacing column $i$ of $\Sigma_{1:j-1}$ by $\Sigma_{\{1,\ldots,j-1\},j}$.
    Up to sign it is a nonprincipal minor of order $j-1$.
    Thus a polynomial in these coefficients has a common denominator $\prod_r\Delta_r^{b_r}$ for some nonnegative integers $b_r$.
    The $\Delta_r$ are pairwise nonassociate irreducibles.
    Reduction in the unique factorization domain $\RR[\Sym_d]$ leaves a denominator of exactly this form, or a constant.
    Strict lower triangularity gives $\beta^d=0$ and $L=I_d+\beta+\cdots+\beta^{d-1}$, proving the assertion for $L$.
\end{proof}

\begin{proof}[Proof of Proposition~\ref{prop:path-products}]
    For $q+1\le j\le q+\ell$, the numerator of $\beta_{j+1,j}$ is $K_j=M_{\{1,\ldots,j\},\{1,\ldots,j-1,j+1\}}$.
    It is irreducible by Lemma~\ref{lem:bcd-minor-irreducible}.
    It is nonassociate to $\Delta_r$ if $r\ne j$ by degree, and nonassociate to $\Delta_j$ because it vanishes at $I_d$ while $\Delta_j(I_d)=1$.
    Consequently $\prod_{j=q+1}^{q+\ell}K_j$ is coprime to $\prod_{j=q+1}^{q+\ell}\Delta_j$.
    The multiplicities are $\ell$ in the first $q+1$ positions, $\ell-1,\ldots,1$ in the next $\ell-1$ positions, and zero thereafter.
    Their squared sum is $q\ell^2+\sum_{h=1}^{\ell}h^2$.
    Lemma~\ref{lem:flag-structure} gives the stated constant.
\end{proof}

For the causal interpretation below, fix $\Sigma\succ0$ and consider $X=\beta(\Sigma)X+\epsilon$, where $\epsilon=(\epsilon_1,\ldots,\epsilon_d)^\top$ has independent mean-zero coordinates with variances $v_1,\ldots,v_d$ from Section~\ref{sec:recursive}.
A static intervention replaces the equations at selected coordinates by fixed values.

\begin{proposition}[Static interventions in a recursive linear model]
    \label{prop:bcd-gformula}
    For $\cT\subseteq[d]$, form $\beta^{\cT}$ by setting the rows of $\beta(\Sigma)$ indexed by $\cT$ to zero.
    Fix $u,h\in\RR^d$ with $u_j=0$ for $j\notin\cT$.
    The change in $\EE(h^\top X)$ from setting $X_j=0$ to setting $X_j=u_j$ for $j\in\cT$ is
    \begin{equation*}
        T_{\cT,u,h}(\Sigma)
        =
        h^\top\sum_{r=0}^{d-1}(\beta^{\cT})^r u
        .
    \end{equation*}
    Its reduced denominator divides $\prod_{j\in[d]\setminus\cT,\ j\ge2}\Delta_{j-1}$.
    For $1\le t<y\le d$, put $T_{t,y}=T_{\{t\},e_t,e_y}$.
    Then
    \begin{equation*}
        T_{t,y}
        =
        L_{yt}
        =
        \{\Sigma_{1:t}^{-1}\Sigma_{\{1,\ldots,t\},y}\}_t
        =
        \frac{M_{\{1,\ldots,t\},\{1,\ldots,t-1,y\}}}{\Delta_t}
        ,
    \end{equation*}
    and the final fraction is reduced with denominator self-normalization constant $2t$.
\end{proposition}

\begin{proof}[Proof of Proposition~\ref{prop:bcd-gformula}]
    Let $\mu$ be the change in the mean vector under the intervention.
    The intervened equations give $\mu=\beta^{\cT}\mu+u$.
    Since $\beta^{\cT}$ is strictly lower triangular, $(\beta^{\cT})^d=0$ and $(I_d-\beta^{\cT})^{-1}=\sum_{r=0}^{d-1}(\beta^{\cT})^r$.
    This proves the formula for $T_{\cT,u,h}$.
    Each matrix-power entry is a sum of products along strictly increasing coordinate paths.
    Such a path visits any row at most once, and each visited nonintervened row $j$ contributes a coefficient denominator $\Delta_{j-1}$.
    The displayed product clears every term, so the reduced denominator divides it.
    
    For a single intervention, the predecessors of $t$ have zero mean contrast, the contrast at $t$ is one, and the subsequent recursion propagates the column $L_{\cdot t}$.
    In the observational system, $X=L\epsilon$ and $\epsilon_t$ is the residual of $X_t$ regressed on $X_1,\ldots,X_{t-1}$.
    Hence $\cov(X_y,\epsilon_t)=L_{yt}v_t$.
    The coefficient on $X_t$ in the regression of $X_y$ on $X_{1:t}$ is $\cov(X_y,\epsilon_t)/v_t$, which proves the regression identity.
    Cramer's rule gives the minor fraction.
    Its numerator is irreducible by Lemma~\ref{lem:bcd-minor-irreducible} and is nonassociate to $\Delta_t$, since it vanishes at $I_d$ whereas $\Delta_t(I_d)=1$.
    The fraction is therefore reduced, and $\kappa_{\Delta_t}=2t$.
\end{proof}

\subsection{Reduced denominators in the catalogue}
\label{supp:catalogue}

We verify the GLS, regression, precision, correlation, IV and scalar linear proximal entries of Table~\ref{tab:catalogue}, together with the raw- and residual-covariance representations of adjusted mediation.

\begin{proposition}[Generalized least-squares weights]
    \label{prop:bcd-gls}
    Fix $\cX\in\RR^{d\times p}$ of full column rank with $1\le p<d$.
    Choose $Q_\parallel\in\RR^{d\times p}$ and $Q_\perp\in\RR^{d\times(d-p)}$ so that $(Q_\parallel,Q_\perp)$ is orthogonal and $\cX=Q_\parallel R_0$ for an invertible $R_0\in\RR^{p\times p}$.
    Then
    \begin{equation}
        \label{eq:GLS-identity}
        W_{\cX}
        =
        R_0^{-1}
        \{Q_\parallel^\top-Q_\parallel^\top\Sigma Q_\perp
        (Q_\perp^\top\Sigma Q_\perp)^{-1}Q_\perp^\top\}
        .
    \end{equation}
    Every nonconstant scalar entry has reduced denominator $\Delta_{\operatorname{Range}(\cX)^\perp}$, up to a nonzero scalar, with self-normalization constant $2(d-p)$.
    The same polynomial is the reduced common denominator, meaning the least common multiple of the scalar reduced denominators.
\end{proposition}

\begin{proof}[Proof of Proposition~\ref{prop:bcd-gls}]
    In the orthogonal basis $(Q_\parallel,Q_\perp)$, write $\tilde{\Sigma}=\left(\begin{smallmatrix}S_{11}&S_{12}\\ S_{21}&S_{22}\end{smallmatrix}\right)$ in blocks of sizes $p$ and $d-p$.
    The first block row of its inverse is $[K^{-1},-K^{-1}S_{12}S_{22}^{-1}]$, where $K=S_{11}-S_{12}S_{22}^{-1}S_{21}$.
    Substitution in the definition of $W_{\cX}$ gives \eqref{eq:GLS-identity}.
    The polynomial $\det S_{22}=\Delta_{\operatorname{Range}(\cX)^\perp}$ is irreducible and clears every scalar denominator.
    Each reduced denominator is therefore constant or a scalar multiple of $\det S_{22}$.
    An entry with constant denominator is a polynomial invariant under $\Sigma\mapsto t\Sigma$ for $t>0$.
    Homogeneous-component comparison makes such an entry constant, proving the assertion for every nonconstant entry.
    
    To identify the common denominator, \eqref{eq:GLS-identity} gives $S_{12}S_{22}^{-1}=-R_0W_{\cX}Q_\perp$.
    The coordinates of $S_{12}$ and the distinct coordinates of $S_{22}$ are independent indeterminates.
    The coefficients of an entry of $S_{12}S_{22}^{-1}$ are cofactor ratios of $S_{22}$, with nonzero numerators of degree smaller than $\det S_{22}$.
    Thus $\det S_{22}$ cannot cancel from all these entries, and the recovery identity makes it the reduced common denominator of $W_{\cX}$.
    Its self-normalization constant is $2(d-p)$ by Lemma~\ref{lem:flag-structure}.
\end{proof}

\subsubsection{Regression and precision entries}
\label{supp:regression-precision}
An ordinary regression coefficient with $p_0\ge1$ predictors is an almost-principal minor divided by its principal predictor determinant.
The two minors are irreducible and nonassociate, so the fraction is reduced and its denominator has self-normalization constant $2p_0$.

A precision entry is a nonzero cofactor divided by $\det\Sigma$.
Its numerator has degree $d-1$, whereas the denominator is irreducible of degree $d$, so the fraction is reduced.
For fixed $u\in\RR^d\setminus\{0\}$, the numerator of $u^\top\Sigma^{-1}u$ is $u^\top\operatorname{adj}(\Sigma)u$, a polynomial of degree $d-1$ that equals $\|u\|^2$ at $I_d$.
The same reduction argument applies and gives denominator constant $2d$ in both cases.
For $d=1$, the cofactor is one.

\subsubsection{Squared partial correlation and the coefficient of determination}
\label{supp:correlation-denominators}
Squared partial correlation is $M^2/(\Delta_a\Delta_b)$ in the notation of Proposition~\ref{prop:bcd-partial-minor}.
The three minors are irreducible and pairwise nonassociate, so this fraction is reduced.
Lemma~\ref{lem:two-subspace} gives the first-order relative variance of its denominator and the diagnostic band in Section~\ref{sec:catalogue}.

For the nonempty predictor set $\cJ$ of size $p_0$ and $r_Y=\Sigma_{\cJ,Y}$, the numerator of $R^2_{Y\cdot\cJ}$ is $r_Y^\top\operatorname{adj}(\Sigma_{\cJ})r_Y$ and the denominator is $\sigma_{YY}\Delta_{\cJ}$.
The numerator does not depend on $\sigma_{YY}$ and is nonzero.
It is not divisible by $\Delta_{\cJ}$, since the assignment $\Sigma_{\cJ}=\diag(I_{p_0-1},0)$ and $r_Y=e_{p_0}\in\RR^{p_0}$ gives numerator one and $\Delta_{\cJ}=0$.
For $p_0=1$, the adjugate is the scalar one by the empty-determinant convention.
The fraction is reduced, and Lemma~\ref{lem:two-subspace} gives the first-order relative variance of its denominator.

\subsubsection{Conditional IV and the scalar linear proximal formula}
\label{supp:iv-proximal-formulas}
For conditional IV, order the coordinates as $(C,A,Z,Y)$, where $C$ is a possibly empty covariate index set and $A,Z,Y$ are distinct coordinates outside $C$ representing treatment, instrument and outcome.
For $U,V\in\{A,Z,Y\}$, put $\sigma_{UV\cdot C}=\sigma_{UV}-\Sigma_{U,C}\Sigma_C^{-1}\Sigma_{C,V}$, with $\sigma_{UV\cdot C}=\sigma_{UV}$ when $C$ is empty.
For $\sigma_{ZA\cdot C}\ne0$, Schur complementation gives the covariance formula
\begin{equation*}
    T_{\mathrm{IV}}
    =
    \frac{\sigma_{ZY\cdot C}}{\sigma_{ZA\cdot C}}
    =
    \frac{M_{C\cup\{Z\},C\cup\{Y\}}}
    {M_{C\cup\{Z\},C\cup\{A\}}}
    .
\end{equation*}
The numerator and denominator are irreducible by Lemma~\ref{lem:bcd-minor-irreducible} and nonassociate because only the numerator depends on outcome covariances.
Thus the fraction is reduced, and $C=\emptyset$ gives marginal IV.

For the scalar linear proximal formula, take centered Gaussian coordinates $(A,Z,W,Y)$ with positive-definite covariance, where $Z,W$ are proxies of an unmeasured confounder.
A linear outcome bridge is a function $h(A,W)=\theta A+\lambda W$, with $\theta,\lambda\in\RR$, satisfying $\EE\{Y-h(A,W)\mid A,Z\}=0$ \citep{tchetgen2024}.
Its treatment coefficient $\theta$ is the scalar linear proximal functional.
Put $\delta=\det\Sigma_{\{A,Z\}}$ and, for $V\in\{W,Y\}$, define
\begin{equation*}
    u_V=\sigma_{ZZ}\sigma_{AV}-\sigma_{AZ}\sigma_{ZV}
    ,\quad
    v_V=\sigma_{AA}\sigma_{ZV}-\sigma_{AZ}\sigma_{AV}
    .
\end{equation*}
The coefficients on $A,Z$ in $\EE(V\mid A,Z)$ are $u_V/\delta,v_V/\delta$.
For $v_W\ne0$, the bridge equation gives $\lambda=v_Y/v_W$ and
\begin{equation*}
    \begin{split}
        &T_{\mathrm{prox}}
        =\theta
        =
        \frac{u_Yv_W-u_Wv_Y}{\delta v_W}
        =
        \frac{N_{\mathrm{prox}}}{D_{\mathrm{prox}}}
        ,\\
        &N_{\mathrm{prox}}
        =
        \sigma_{AY}\sigma_{ZW}-\sigma_{AW}\sigma_{ZY}
        ,\quad
        D_{\mathrm{prox}}
        =
        \sigma_{AA}\sigma_{ZW}-\sigma_{AZ}\sigma_{AW}
        .
    \end{split}
\end{equation*}
The identity $u_Yv_W-u_Wv_Y=\delta N_{\mathrm{prox}}$ displays the cancellation of $\delta$.
Both remaining polynomials are covariance minors up to sign, hence irreducible by Lemma~\ref{lem:bcd-minor-irreducible}.
They are nonassociate because only $N_{\mathrm{prox}}$ depends on outcome covariances.
Thus $D_{\mathrm{prox}}=M_{\{A,Z\},\{A,W\}}$ is the reduced denominator, corresponding to $C=\{A\}$, $a=Z$ and $b=W$ in Table~\ref{tab:catalogue}.
Both formulas have an almost-principal denominator of multiplicity one, so Proposition~\ref{prop:bcd-minor-products} gives relative instability and a first-order weak denominator sequence.

\subsubsection{Adjusted mediation in raw and residual covariance coordinates}
\label{supp:adjusted-mediation}
This subsection gives the example of representation dependence cited in Remark~\ref{rem:denominator-representation}.
Use the coordinate order $(C,X,M,Y)$ with $|C|=q$.
The adjusted product is $T_{q,2}$ in Proposition~\ref{prop:path-products}, so its reduced denominator is $\Delta_{C\cup\{X\}}\Delta_{C\cup\{X,M\}}$ with self-normalization constant $8q+10$.
Put $I=(X,M,Y)$ and let $S=\Sigma_{I\cdot C}=\Sigma_{I,I}-\Sigma_{I,C}\Sigma_C^{-1}\Sigma_{C,I}$ be the residual covariance of $I$ given $C$, with $S=\Sigma_{I,I}$ when $C$ is empty.
Schur complementation gives
\begin{equation*}
    S_{XX}
    =
    \frac{\Delta_{C\cup\{X\}}}{\Delta_C}
    ,\quad
    \det S_{\{X,M\}}
    =
    \frac{\Delta_{C\cup\{X,M\}}}{\Delta_C}
    .
\end{equation*}
Thus $S_{XX}\det S_{\{X,M\}}=\Delta_{C\cup\{X\}}\Delta_{C\cup\{X,M\}}/\Delta_C^2$ as a function of the raw covariance.
As a polynomial in the three-dimensional residual covariance $S$, it is a flag power with multiplicities $(2,1,0)$ and self-normalization constant $10$.

\subsection{The Gaussian Napkin functional and a conditional IV comparison}
\label{supp:napkin}

We prove Example~\ref{ex:bcd-napkin} and compare its denominator with that of a conditional IV formula for the same Gaussian structural coefficient.

\begin{proof}[Proof of Example~\ref{ex:bcd-napkin}]
    \emph{Gaussian reduction.}
    Use the coordinate order $(Z,W,X,Y)$.
    Let $p(w)$ be the marginal density of $W$ and $p(x\mid z,w)$ the conditional density of $X$ given $Z=z$ and $W=w$.
    For a fixed intervention value $x\in\RR$ and a fixed value $z\in\RR$ of the trapdoor variable $Z$, the identifying mean is the following ratio, with both integrals taken over $w\in\RR$:
    \begin{equation}
        \label{eq:napkin-identification}
        \frac{\displaystyle\int
        \EE(Y\mid X=x,Z=z,W=w)\,p(x\mid z,w)p(w)\dd w}
        {\displaystyle\int p(x\mid z,w)p(w)\dd w}
    \end{equation}
    under the graphical model and its positivity conditions \citep{helske2021,guo2025}.
    Put $w_0=\sigma_{WW}$.
    For a centered Gaussian law, the normalized weight in \eqref{eq:napkin-identification} is a Gaussian density in $w$ obtained from $W\sim N_1(0,w_0)$ and $X\mid Z=z,W=w\sim N_1(a_{\mathrm{nap}}z+b_{\mathrm{nap}}w,s)$.
    Completing the square gives its mean $b_{\mathrm{nap}}w_0(x-a_{\mathrm{nap}}z)/(s+b_{\mathrm{nap}}^2w_0)$.
    The coefficient of $x$ is therefore $T_{\mathrm{nap}}$ from Section~\ref{sec:napkin}.
    
    Write $H=\Sigma_{\{Z,W\}}$, $\delta=\det H$ and $T=\det\Sigma_{\{Z,W,X\}}$, and retain $r_X,r_Y,m,m_Y$ from the example.
    The regression normal equations give
    \begin{equation*}
        \begin{split}
            b_{\mathrm{nap}}
            =
            \frac{m}{\delta}
            ,\quad
            s
            =
            \frac{T}{\delta}
            ,\quad
            \theta
            =
            \frac{\delta\sigma_{XY}-r_X^\top\operatorname{adj}(H)r_Y}{T}
            ,\quad
            \eta
            =
            \frac{m_Y-\theta m}{\delta}
            .
        \end{split}
    \end{equation*}
    Hence
    \begin{equation*}
        T_{\mathrm{nap}}
        =
        \frac{\theta\delta T+w_0mm_Y}{\delta T+w_0m^2}
        =
        \frac{\delta^2\sigma_{XY}
        -\delta r_X^\top\operatorname{adj}(H)r_Y+w_0mm_Y}
        {\delta T+w_0m^2}
        ,
    \end{equation*}
    proving \eqref{eq:napkin-polynomials}.
    
    \emph{Irreducibility, reduction and positivity.}
    Regard the denominator as a polynomial in the covariance entry $\sigma_{XX}$,
    \begin{equation*}
        D_{\mathrm{nap}}
        =
        \delta^2\sigma_{XX}-\delta r_X^\top\operatorname{adj}(H)r_X+w_0m^2
        .
    \end{equation*}
    Its leading coefficient is $\delta^2$, and its constant coefficient modulo $\delta$ is $w_0m^2$.
    The irreducible polynomial $\delta$ divides neither $w_0$ nor $m$.
    The former follows by degree, and the latter follows because $m$ and $\delta$ are nonassociate irreducible minors by Lemma~\ref{lem:bcd-minor-irreducible}.
    Thus the two coefficients are coprime, and the polynomial is primitive and linear in $\sigma_{XX}$.
    Gauss' lemma proves irreducibility, which is preserved upon adjoining the outcome coordinates.
    The numerator is linear in $\sigma_{XY}$ with coefficient $\delta^2$.
    If $D_{\mathrm{nap}}$ divided the numerator, it would divide this coefficient, which is impossible because $D_{\mathrm{nap}}$ has degree one in $\sigma_{XX}$.
    The fraction is therefore reduced.
    Since $\delta,T,w_0>0$ for $\Sigma\succ0$, $D_{\mathrm{nap}}=\delta T+w_0m^2>0$ on that cone.
    
    \emph{A path in the structural model.}
    The covariance block of the constructed $(Z,W,X)$ is
    \begin{equation*}
        \Sigma_t^{ZWX}
        =
        \begin{pmatrix}
            1&1&\alpha+1\\
            1&1+t&\alpha+1+t\\
            \alpha+1&\alpha+1+t&(\alpha+1)^2+t+1
        \end{pmatrix}
        .
    \end{equation*}
    Its loading from $(U,V,R)$ has determinant $\sqrt{t}$, so the block is positive definite for $t>0$.
    Here $\delta=m=T=t$ and $w_0=1+t$, giving $D_{\mathrm{nap}}=t^2(t+2)$.
    
    To extend the path to the four-variable model, let $E$ be standard normal and independent of $(U,V,R)$.
    Fix $\tau_0,\lambda_0\in\RR$ and set $Y=\tau_0X+\lambda_0\{W-(1+t)R\}+E$.
    The loading from $(U,V,R,E)$ is invertible, so the four-dimensional covariance is positive definite.
    Define $\epsilon_W=W$, $\epsilon_Z=Z-W/(1+t)$, $\epsilon_X=W+R$ and $\epsilon_Y=\lambda_0\{W-(1+t)R\}+E$.
    Then
    \begin{equation*}
        W=\epsilon_W
        ,\quad
        Z=\frac{1}{1+t}W+\epsilon_Z
        ,\quad
        X=\alpha Z+\epsilon_X
        ,\quad
        Y=\tau_0X+\epsilon_Y
        .
    \end{equation*}
    The Gaussian variable $\epsilon_Z$ is independent of the other errors because it is uncorrelated with $W,R,E$.
    Also $\cov(\epsilon_X,\epsilon_Y)=\lambda_0\{(1+t)-(1+t)\}=0$.
    The only possibly nonzero covariances between distinct errors are $\cov(\epsilon_W,\epsilon_X)=1+t$ and $\cov(\epsilon_W,\epsilon_Y)=\lambda_0(1+t)$, as specified by the bidirected edges.
    Both are nonzero when $\lambda_0\ne0$, and $\alpha\ne0$ is the coefficient of the directed edge $Z\to X$.
    The projection coefficients satisfy $b_{\mathrm{nap}}=1$, $s=1$, $\theta=\tau_0-\lambda_0(1+t)$ and $\eta=\lambda_0(t+2)$, so $T_{\mathrm{nap}}=\tau_0$ along this path.
    
    \emph{First-order relative variance.}
    Write $D_{\mathrm{nap}}=\delta^2(s+w_0b_{\mathrm{nap}}^2)$.
    Along the path the influence functions are
    \begin{equation*}
        \phi_{\log\delta}=(U^2-1)+(V^2-1)
        ,\quad
        \phi_s=R^2-1
        ,\quad
        \phi_{w_0}=W^2-(1+t)
        ,\quad
        \phi_{b_{\mathrm{nap}}}=VR/\sqrt{t}
        .
    \end{equation*}
    The last formula follows by residualizing $W$ on $Z$, which gives $W-Z=\sqrt{t}V$, and using the regression error $R$.
    The logarithmic chain rule yields
    \begin{equation}
        \label{eq:napkin-log-score}
        \begin{split}
            \phi_{\log D_{\mathrm{nap}}}
            &=
            \frac{2t+5}{t+2}(U^2-1)
            +\frac{3t+4}{t+2}(V^2-1)
            +\frac{1}{t+2}(R^2-1)
            \\
            &\quad
            +\frac{2\sqrt{t}}{t+2}UV
            +\frac{2(t+1)}{\sqrt{t}(t+2)}VR
            .
        \end{split}
    \end{equation}
    The variables $U^2-1,V^2-1,R^2-1,UV,VR$ are pairwise uncorrelated with variances $2,2,2,1,1$.
    The variance of \eqref{eq:napkin-log-score} is therefore the expression in \eqref{eq:napkin-path}, which diverges as $t\to0$.
    
    \emph{Sample-to-population limit.}
    Let $\xi_i=(U_i,V_i,R_i)^\top$, $i=1,\ldots,n$, be independent $N_3(0,I_3)$ vectors and put $S=n^{-1}\sum_i\xi_i\xi_i^\top=(s_{ij})$.
    For $n\ge3$, $S\succ0$ almost surely.
    Define $g(S)=(s_{11}s_{23}-s_{12}s_{13})/(s_{11}s_{22}-s_{12}^2)$ and $r(S)=s_{33}-S_{3,1:2}S_{1:2}^{-1}S_{1:2,3}$.
    Hats below denote the sample versions of $\delta,b_{\mathrm{nap}},s,w_0$.
    Regression algebra gives
    \begin{equation*}
        \hat{\delta}=t\det S_{1:2}
        ,\quad
        \hat{b}_{\mathrm{nap}}=1+t^{-1/2}g(S)
        ,\quad
        \hat{s}=r(S)
        ,\quad
        \hat{w}_0=s_{11}+2\sqrt{t}\,s_{12}+ts_{22}
        .
    \end{equation*}
    These identities hold for every $\alpha$.
    Since $S\pto I_3$, $r(S)\pto1$.
    The derivative of $g$ at $I_3$ selects $s_{23}$, so $\sqrt{n}g(S)\dto Z_0\sim N_1(0,1)$.
    Set $t=t_n=1/n$ in the sample identities, so that $\hat{w}_0\pto1$.
    Substitution gives
    \begin{equation*}
        \begin{split}
            \frac{D_{\mathrm{nap}}(\Sighat)}{D_{\mathrm{nap}}(\Sigma_{t_n})}
            &=
            \frac{(\det S_{1:2})^2}{t_n+2}
            \{r(S)+\hat{w}_0(1+\sqrt{n}g(S))^2\}
            \\
            &\dto
            \frac{1+(1+Z_0)^2}{2}
            .
        \end{split}
    \end{equation*}
    This proves \eqref{eq:napkin-relative-limit}.
\end{proof}

\subsubsection{A conditional instrumental-variable representation}
\label{supp:napkin-iv}

In the Gaussian structural model, write $Z=b_{ZW}W+\epsilon_Z$, $X=\alpha Z+\epsilon_X$ and $Y=\tau_0X+\epsilon_Y$.
The graphical error restrictions make $\epsilon_Z$ independent of $(W,\epsilon_X,\epsilon_Y)$.
Thus $\tilde{Z}=Z-\sigma_{ZW}W/\sigma_{WW}=\epsilon_Z$ satisfies $\cov(\tilde{Z},Y)=\tau_0\cov(\tilde{Z},X)$.
Put $D_{\mathrm{IV}}=\sigma_{WW}\sigma_{ZX}-\sigma_{WZ}\sigma_{WX}$.
For $\alpha\ne0$, $\cov(\tilde{Z},X)=\alpha\var(\epsilon_Z)\ne0$, and the ambient formula
\begin{equation*}
    \frac{\sigma_{WW}\sigma_{ZY}-\sigma_{WZ}\sigma_{WY}}
    {D_{\mathrm{IV}}}
    =
    \tau_0
\end{equation*}
holds on that model.
The numerator and denominator are, up to signs, nonassociate irreducible almost-principal minors, so the fraction is reduced.

On the path of Example~\ref{ex:bcd-napkin}, $D_{\mathrm{IV}}(\Sigma_t)=\alpha t$.
The residual variances and covariance given $W$ are $\var(Z\mid W)=t/(1+t)$, $\var(X\mid W)=1+\alpha^2t/(1+t)$ and $\cov(Z,X\mid W)=\alpha t/(1+t)$.
Proposition~\ref{prop:bcd-partial-minor} with $q=1$ therefore gives
\begin{equation*}
    \kappa_{D_{\mathrm{IV}}}(\Sigma_t)
    =
    4+\frac{1+t}{\alpha^2t}
    .
\end{equation*}
Fix $t>0$ and let $\Sigma_n$ be the constructed four-variable covariance with $\alpha=\alpha_n=h/\sqrt{n}$ for a fixed $h\ne0$.
Then $\Sigma_n\to\Sigma_\ast\succ0$, $\sqrt{n}D_{\mathrm{IV}}(\Sigma_n)=ht$ and $s_{D_{\mathrm{IV}}}^2(\Sigma_\ast)=t(1+t)>0$.
This is a first-order weak denominator sequence.

The change $X\mapsto X+\alpha Z$ leaves the quantities $\delta,T,m,w_0$ from the preceding proof unchanged, both in the population and in the sample.
Consequently the sample-to-population law of $D_{\mathrm{nap}}$ does not depend on $\alpha$.
For fixed $\alpha\ne0$, both first-order relative variances diverge as $t\to0$.
This comparison distinguishes the denominator behavior of two ambient formulas for the same Gaussian structural coefficient.

\section{Proofs for mediation and regression-coordinate inference}
\label{supp:mediation}

This section proves Proposition~\ref{prop:mediation-boundary-inference} and the local coefficient limits of Remark~\ref{rem:mediation-local-regimes}.
It also derives the moment-inversion formula and the extensions to unknown means.
Empirical score covariances use divisor $n$ throughout.
For a differentiable covariance functional $f$, write $\phi_f(X;\Sigma)=\dd f_\Sigma(XX^\top-\Sigma)$ for its known-mean influence function, componentwise when $f$ is vector-valued.
We suppress the arguments of $\phi_f$ when they are clear.

\subsection{Regression scores and residual variances}
\label{supp:regression-scores}

We derive the regression influence functions and the Gaussian score orthogonality used in the standard-error and moment-inversion calculations.

\begin{lemma}[Regression scores and orthogonality]
    \label{lem:bcd-regression-orthogonality}
    Let $X=(X_1,\ldots,X_d)^\top\sim N_d(0,\Sigma)$ with $\Sigma\succ0$.
    For $2\le j\le d$, put $\beta_j=\Sigma_{1:j-1}^{-1}\Sigma_{\{1,\ldots,j-1\},j}$ and $\epsilon_j=X_j-\beta_j^\top X_{1:j-1}$.
    Set $\epsilon_1=X_1$ and $v_i=\var(\epsilon_i)$ for $1\le i\le d$.
    For $2\le j\le d$ and $1\le i\le d$,
    \begin{equation}
        \label{eq:regression-scores}
        \phi_{\beta_j}=\Sigma_{1:j-1}^{-1}X_{1:j-1}\epsilon_j
        ,\quad
        \phi_{v_i}=\epsilon_i^2-v_i
        .
    \end{equation}
    Coefficient-score vectors from distinct rows are uncorrelated, and every coefficient score is uncorrelated with every residual-variance score.
    For a flag power $D=\gamma\prod_{i=1}^dv_i^{m_i}$ in this order, $\phi_{\log|D|}=\sum_{i=1}^dm_i(\epsilon_i^2/v_i-1)$.
    If a differentiable scalar functional $\tau$ depends only on these regression coefficients, then $\nabla\tau(\Sigma)^\top\Gamma(\Sigma)\nabla D(\Sigma)=0$ for every $\Sigma\succ0$.
\end{lemma}

\begin{proof}[Proof of Lemma~\ref{lem:bcd-regression-orthogonality}]
    Differentiating $\Sigma_{1:j-1}\beta_j=\Sigma_{\{1,\ldots,j-1\},j}$ in the direction $XX^\top-\Sigma$ gives the coefficient formula in \eqref{eq:regression-scores}.
    The residual is uncorrelated with every regressor, so coefficient-derivative terms vanish when differentiating its variance.
    This gives $\phi_{v_j}=\epsilon_j^2-v_j$, including $j=1$ directly.
    Gaussian independence of $X_{1:j-1}$ and $\epsilon_j$ also gives $\cov(\phi_{\beta_j})=v_j\Sigma_{1:j-1}^{-1}$.
    
    The Gaussian innovations $\epsilon_1,\ldots,\epsilon_d$ are independent and centered.
    For $j<k$, conditioning on $X_{1:k-1}$ gives zero cross covariance between $\phi_{\beta_j}$ and $\phi_{\beta_k}$.
    For $i<j$, conditioning on the first $j-1$ innovations gives zero covariance between $\phi_{\beta_j}$ and $\phi_{v_i}$.
    For $i=j$, independence of $X_{1:j-1}$ and $\epsilon_j$ and the zero mean of $X_{1:j-1}$ give the same conclusion.
    For $i>j$, $\phi_{v_i}$ is independent of $\phi_{\beta_j}$ and has mean zero.
    
    Logarithmic differentiation gives the stated influence function of $D$.
    The chain rule expresses $\phi_\tau$ as a linear combination of coefficient scores, so $\cov(\phi_\tau,\phi_D)=0$.
    By Lemma~\ref{lem:trace-main}, this covariance equals $\nabla\tau^\top\Gamma\nabla D$.
\end{proof}

\subsection{Exact coefficient and standard-error scaling}
\label{supp:canonical}

We prove Proposition~\ref{prop:mediation-boundary-inference} (i) by expressing the coefficients and their standard errors through one standardized sample.
Let $\xi_i=(\xi_{Xi},\xi_{Mi},\xi_{Yi})^\top$ be independent copies of $\xi$ with the fixed law $Q$ of Section~\ref{sec:mediation}.
Thus $\EE_Q(\xi)=0$, $\cov_Q(\xi)=I_3$ and $\EE_Q(\|\xi\|^4)<\infty$.
Write $S=n^{-1}\sum_{i=1}^n\xi_i\xi_i^\top=(s_{ij})$, put $\Gamma_\xi=\cov_Q\{\vech(\xi\xi^\top)\}$, and let $\hat{\Gamma}_\xi$ be the empirical covariance of $\vech(\xi_i\xi_i^\top)$ with divisor $n$.
Also put $\Omega=\cov_Q\{(\xi_X\xi_M,\xi_M\xi_Y)^\top\}$.
These definitions of $Q$, $\Gamma_\xi$ and $\Omega$ apply to the remaining mediation calculations.
On $S\succ0$, let $r_M(S)=s_{22}-s_{12}^2/s_{11}$ and $r_Y(S)=s_{33}-S_{3,1:2}S_{1:2}^{-1}S_{1:2,3}$ be the two successive sample residual variances.
We use $g_a$ and $g_b$ from Section~\ref{sec:mediation}.

The loading from $\xi$ to $(X,M,Y)^\top$ is
\begin{equation}
    \label{eq:loading}
    A_v
    =
    \begin{pmatrix}
    \sqrt{\nu_X}&0&0\\
    a\sqrt{\nu_X}&\sqrt{v}&0\\
    (ab+\gamma)\sqrt{\nu_X}&b\sqrt{v}&\sigma_\varepsilon
    \end{pmatrix}
    ,\quad
    \Sighat=A_vSA_v^\top
    .
\end{equation}
For $v,\nu_X,\sigma_\varepsilon>0$, $A_v$ is invertible, so $\Sighat\succ0$ if and only if $S\succ0$.
For an invertible $A\in\RR^{3\times3}$, let $T_A\in\RR^{6\times6}$ denote its coordinate matrix.
It is defined by $\vech(AHA^\top)=T_A\vech(H)$ for every $H\in\Sym_3$.

\begin{lemma}[Canonical scaling of coefficients and standard errors]
    \label{lem:mediation-canonical-reduction}
    Under \eqref{eq:mediation-model}, $\hat{a}=a+\sqrt{v/\nu_X}\,g_a(S)$ and $\hat{b}=b+(\sigma_\varepsilon/\sqrt{v})g_b(S)$ on $S\succ0$.
    For $J\in\{\mathrm{G},\mathrm{emp}\}$, the standard errors in \eqref{eq:effect-se} satisfy $\widehat{\operatorname{se}}_{a,J}=\sqrt{v/\nu_X}\,e_{a,n,J}$ and $\widehat{\operatorname{se}}_{b,J}=(\sigma_\varepsilon/\sqrt{v})e_{b,n,J}$, where the nonnegative factors are defined by
    \begin{equation}
        \label{eq:canonical-se}
        \begin{aligned}
            &e_{a,n,\mathrm{G}}^2
            =
            \frac{r_M(S)}{ns_{11}}
            ,\qquad
            e_{b,n,\mathrm{G}}^2
            =
            \frac{r_Y(S)}{nr_M(S)}
            ,\\
            &e_{a,n,\mathrm{emp}}^2
            =
            \frac{1}{n}\nabla g_a(S)^\top
            \hat{\Gamma}_\xi\nabla g_a(S)
            ,\quad
            e_{b,n,\mathrm{emp}}^2
            =
            \frac{1}{n}\nabla g_b(S)^\top
            \hat{\Gamma}_\xi\nabla g_b(S)
            .
        \end{aligned}
    \end{equation}
\end{lemma}

\begin{proof}[Proof of Lemma~\ref{lem:mediation-canonical-reduction}]
    Substitution of $M_i=aX_i+\sqrt{v}\xi_{Mi}$ into $\hat{a}=\sum_iX_iM_i/\sum_iX_i^2$ gives the first coefficient identity.
    The residual from regressing $M$ on $X$ in observation $i$ is $\sqrt{v}\{\xi_{Mi}-g_a(S)\xi_{Xi}\}$.
    Since $Y_i=bM_i+\gamma X_i+\sigma_\varepsilon\xi_{Yi}$, the error in the coefficient of $M$ is $\sigma_\varepsilon/\sqrt{v}$ times the coefficient of $\xi_M$ in the sample regression of $\xi_Y$ on $(\xi_X,\xi_M)$.
    The normal equations identify the latter coefficient as $g_b(S)$.
    The divisor-$n$ residual variances for $M$ on $X$ and $Y$ on $(X,M)$ are $vr_M(S)$ and $\sigma_\varepsilon^2r_Y(S)$.
    Evaluating the Gaussian coefficient variance from Lemma~\ref{lem:bcd-regression-orthogonality} at the sample covariance proves the first line of \eqref{eq:canonical-se}.
    
    The empirical covariance satisfies $\hat{\Gamma}_{\mathrm{emp}}=T_{A_v}\hat{\Gamma}_\xi T_{A_v}^\top$, and $\Gamma(A_vSA_v^\top)=T_{A_v}\Gamma(S)T_{A_v}^\top$.
    The latter is the Gaussian congruence identity of Lemma~\ref{lem:trace-main}.
    The gradient of $S\mapsto f(A_vSA_v^\top)$ is $T_{A_v}^\top\nabla f(A_vSA_v^\top)$.
    Apply this chain rule to the two coefficient identities with the population parameters fixed.
    The empirical covariance identity then proves the remaining lines of \eqref{eq:canonical-se}.
    
    With $\tilde{a}=a\sqrt{\nu_X/v}$ and $\tilde{b}=b\sqrt{v}/\sigma_\varepsilon$, the coefficient statistics satisfy
    \begin{equation*}
        t_{a,J}
        =
        \frac{\tilde{a}+g_a(S)}{e_{a,n,J}}
        ,\quad
        t_{b,J}
        =
        \frac{\tilde{b}+g_b(S)}{e_{b,n,J}}
    \end{equation*}
    on their respective positive-variance domains.
    The factors and these domains depend only on the standardized sample.
    Assigning zero to each statistic outside its domain preserves this dependence.
    Thus their joint law depends on the structural parameters only through $(\tilde{a},\tilde{b})$, proving Proposition~\ref{prop:mediation-boundary-inference} (i).
\end{proof}

\subsubsection{Gaussian-reference variances and regression normalization}
\label{supp:mediation-gaussian-normalization}
Let $\hat{v}_M$ and $\hat{\sigma}_\varepsilon^2$ be the divisor-$n$ residual variances for $M$ on $X$ and $Y$ on $(X,M)$, and let $\hat{\sigma}_{XX}$ be the sample variance of $X$.
The coefficient-score vectors for these two regressions are uncorrelated under the Gaussian reference by Lemma~\ref{lem:bcd-regression-orthogonality}.
The product rule therefore gives, at every $\Sighat\succ0$,
\begin{equation*}
    \widehat{\operatorname{se}}_{\tau_{\mathrm{med}},\mathrm{G}}^2
    =
    \frac{1}{n}\Biggl\{
        \hat{b}^2\frac{\hat{v}_M}{\hat{\sigma}_{XX}}
        +\hat{a}^2\frac{\hat{\sigma}_\varepsilon^2}{\hat{v}_M}
    \Biggr\}
    .
\end{equation*}

For the usual no-intercept regression standard errors, $r_M(S)/(ns_{11})$ and $r_Y(S)/(nr_M(S))$ in \eqref{eq:canonical-se} are replaced by $r_M(S)/\{(n-1)s_{11}\}$ and $r_Y(S)/\{(n-2)r_M(S)\}$, respectively, with $n>2$.
Under Gaussian sampling, the resulting coefficient statistics centered at their true values have Student's $t$ laws with $n-1$ and $n-2$ degrees of freedom.
Conditionally on the respective designs, each coefficient error is Gaussian and independent of its residual sum of squares, which has the matching scaled chi-squared law.

\subsection{Studentized mediation under residual-variance collapse}
\label{supp:scaled-mediation}

We prove Proposition~\ref{prop:mediation-boundary-inference} (ii)--(iii) by controlling the scaled functional $F_v$ uniformly as $v$ approaches zero.

\begin{lemma}[Uniform expansion of the scaled mediation functional]
    \label{lem:mediation-scaled-expansion}
    Fix $a,b,\gamma\in\RR$ and $\nu_X,\sigma_\varepsilon,\bar{v}>0$.
    Let $F_v$ be as in \eqref{eq:scaled-Fv}, extended to $v=0$ by $F_0(S)=a\sigma_\varepsilon g_b(S)$.
    Its first and second derivatives with respect to $S$ are bounded uniformly for $v\in[0,\bar{v}]$ on a fixed positive-definite neighborhood of $I_3$.
    Also $F_v(I_3)=0$, and its influence function there is
    \begin{equation}
        \label{eq:Lv}
        L_v
        =
        \frac{bv}{\sqrt{\nu_X}}\xi_X\xi_M
        +a\sigma_\varepsilon\xi_M\xi_Y
        .
    \end{equation}
    If $\Omega\succ0$ and $a\ne0$, then $w(v)=\var_Q(L_v)$ satisfies $\inf_{0\le v\le\bar{v}}w(v)\ge\lambda_{\min}(\Omega)a^2\sigma_\varepsilon^2>0$.
    For $v>0$ and $S\succ0$, the exact standard-error identities are
    \begin{equation}
        \label{eq:scaled-sandwich}
        nv\widehat{\operatorname{se}}_{\tau_{\mathrm{med}},\mathrm{emp}}^2
        =
        \nabla F_v(S)^\top\hat{\Gamma}_\xi\nabla F_v(S)
        ,\quad
        nv\widehat{\operatorname{se}}_{\tau_{\mathrm{med}},\mathrm{G}}^2
        =
        \nabla F_v(S)^\top\Gamma(S)\nabla F_v(S)
        .
    \end{equation}
\end{lemma}

\begin{proof}[Proof of Lemma~\ref{lem:mediation-scaled-expansion}]
    Choose a closed Frobenius ball about $I_3$ with radius less than $1/2$.
    On this ball, $s_{11}$ and $\det S_{1:2}$ are bounded away from zero, so $g_a$, $g_b$ and their derivatives through order two are bounded.
    The coefficients of these functions in $F_v$ are continuous and bounded for $v\in[0,\bar{v}]$.
    This proves the uniform bounds for derivatives with respect to $S$.
    At $I_3$, $g_a=g_b=0$ and their derivatives select the $s_{12}$ and $s_{23}$ coordinates.
    The product $g_ag_b$ has zero first derivative there, giving \eqref{eq:Lv}.
    Its variance is $(bv/\sqrt{\nu_X},a\sigma_\varepsilon)\Omega(bv/\sqrt{\nu_X},a\sigma_\varepsilon)^\top$, which proves the lower bound.
    
    For $v>0$, the exact coefficient identities give $\sqrt{v}\{\tau_{\mathrm{med}}(A_vSA_v^\top)-ab\}=F_v(S)$.
    Differentiation with the population parameters fixed gives $\nabla F_v(S)=\sqrt{v}\,T_{A_v}^\top\nabla\tau_{\mathrm{med}}(A_vSA_v^\top)$.
    Combining this with the covariance identities in the proof of Lemma~\ref{lem:mediation-canonical-reduction} proves \eqref{eq:scaled-sandwich}.
\end{proof}

\begin{proof}[Proof of Proposition~\ref{prop:mediation-boundary-inference} (ii)--(iii)]
    \emph{Empirical studentization.}
    Finite fourth moments give $S-I_3=O_p(n^{-1/2})$ and $\sqrt{n}\vech(S-I_3)\dto N_6(0,\Gamma_\xi)$.
    Let $L_{v,i}$ denote \eqref{eq:Lv} evaluated at $\xi_i$.
    Lemma~\ref{lem:mediation-scaled-expansion} gives
    \begin{equation*}
        \sqrt{n}F_{v_n}(S)
        =
        \frac{1}{\sqrt{n}}\sum_{i=1}^nL_{v_n,i}+o_p(1)
        .
    \end{equation*}
    The remainder is bounded by $C\sqrt{n}\|S-I_3\|_F^2=o_p(1)$ on a common event of probability tending to one, uniformly for $v_n\in(0,\bar{v}]$.
    Every subsequence has a further subsequence along which $v_n\to v_\ast\in[0,\bar{v}]$.
    The two coefficients in $L_{v_n}$ then converge, and $w(v_n)\to w(v_\ast)>0$.
    The finite-dimensional central limit theorem yields $\sqrt{n}F_{v_n}(S)/\sqrt{w(v_n)}\dto N_1(0,1)$ along that subsequence and hence along the original sequence.
    
    The law of large numbers gives $\hat{\Gamma}_\xi\pto\Gamma_\xi$.
    The derivative bounds imply $\sup_{0\le v\le\bar{v}}\|\nabla F_v(S)-\nabla F_v(I_3)\|=O_p(n^{-1/2})$.
    Equation~\eqref{eq:scaled-sandwich} and the positive lower bound for $w$ therefore give
    \begin{equation*}
        \frac{nv_n\widehat{\operatorname{se}}_{\tau_{\mathrm{med}},\mathrm{emp}}^2}{w(v_n)}
        \pto
        1
        .
    \end{equation*}
    Together with \eqref{eq:scaled-Fv}, this proves (ii) by Slutsky's lemma.
    The event $S\succ0$ has probability tending to one by \eqref{eq:fixed-moment}.
    Variance consistency and the lower bound for $w$ make the empirical standard error positive with probability tending to one.
    
    \emph{Gaussian-reference studentization.}
    If $\Omega=I_2$, then $w(v)=b^2v^2/\nu_X+a^2\sigma_\varepsilon^2$.
    The restriction of $\Gamma(I_3)$ to the $s_{12}$ and $s_{23}$ coordinates is $I_2$.
    Continuity of $\Gamma$ and \eqref{eq:scaled-sandwich} give $nv_n\widehat{\operatorname{se}}_{\tau_{\mathrm{med}},\mathrm{G}}^2/w(v_n)\pto1$, proving (iii) and the corresponding positive-variance assertion.
    If $\xi_M\indep(\xi_X,\xi_Y)$, then $\EE_Q(\xi_X^2\xi_M^2)=\EE_Q(\xi_M^2\xi_Y^2)=1$ and $\EE_Q(\xi_X\xi_M^2\xi_Y)=\EE_Q(\xi_M^2)\EE_Q(\xi_X\xi_Y)=0$.
    Thus this independence condition implies $\Omega=I_2$.
\end{proof}

\subsubsection{Denominator removal and precision}
\label{supp:mediation-denominator-removal}
Let $\Sigma_n=A_{v_n}A_{v_n}^\top$ and $N_{\mathrm{med}}=\tau_{\mathrm{med}}D_{\mathrm{med}}$, with the fixed coefficients of Proposition~\ref{prop:mediation-boundary-inference} (ii).
The exact normalized denominator is $P(S)=D_{\mathrm{med}}(A_vSA_v^\top)/D_{\mathrm{med}}(A_vA_v^\top)=s_{11}\det S_{1:2}$ for every $v>0$.
Since $P(I_3)=1$ and $S-I_3=O_p(n^{-1/2})$, $P(S)=1+O_p(n^{-1/2})$.
For $\ell_n=\sqrt{nv_n/w(v_n)}$, the normalized moment satisfies, on $S\succ0$,
\begin{equation*}
    \zeta_n
    :=
    \ell_n\frac{N_{\mathrm{med}}(\Sighat)-abD_{\mathrm{med}}(\Sighat)}{D_{\mathrm{med}}(\Sigma_n)}
    =
    \sqrt{\frac{n}{w(v_n)}}F_{v_n}(S)P(S)
    .
\end{equation*}
The expansion above therefore gives $\zeta_n\dto N_1(0,1)$ directly.
Corollary~\ref{cor:denominator-removal}, with its fixed-law extension, transfers this limit to $\ell_n(\hat{\tau}-ab)$.

If $v_n\to0$, then $w(v_n)\to a^2\sigma_\varepsilon^2\Omega_{22}$.
The variance consistency just proved gives $\widehat{\operatorname{se}}_{\tau_{\mathrm{med}},\mathrm{emp}}=|a|\sigma_\varepsilon\sqrt{\Omega_{22}}(nv_n)^{-1/2}\{1+o_p(1)\}$.
For $\Omega=I_2$, the Gaussian-reference standard error has the same expansion, giving the rate in Section~\ref{sec:mediation}.

\subsection{Local coefficient-product limits}
\label{supp:local-product}

We prove the joint estimator and standard-error limits underlying Remark~\ref{rem:mediation-local-regimes}, allowing both path coefficients to vary with the sample size.

\begin{proposition}[Local coefficient-product limits]
    \label{prop:bcd-local-product}
    Under \eqref{eq:mediation-model}, let $Q$ be fixed with $\Omega\succ0$, and fix $\nu_X,\sigma_\varepsilon>0$ and $\gamma\in\RR$.
    Let $a=a_n$, $b=b_n$ and $v=v_n>0$ satisfy $\alpha_n=a_n\sqrt{n\nu_X/v_n}\to a_0\in\RR$ and $\beta_n=b_n\sqrt{nv_n}/\sigma_\varepsilon\to b_0\in\RR$.
    For $(Z_1,Z_2)^\top\sim N_2(0,\Omega)$, put $H=(a_0+Z_1)(b_0+Z_2)-a_0b_0$ and $w_Z=(b_0+Z_2,a_0+Z_1)^\top$.
    Then jointly,
    \begin{equation}
        \label{eq:local-product-joint}
        \frac{n\sqrt{\nu_X}}{\sigma_\varepsilon}(\hat{\tau}-a_nb_n)
        \dto
        H
        ,\quad
        \frac{n^2\nu_X}{\sigma_\varepsilon^2}
        \widehat{\operatorname{se}}_{\tau_{\mathrm{med}},\mathrm{emp}}^2
        \dto
        w_Z^\top\Omega w_Z
        ,\quad
        \frac{n^2\nu_X}{\sigma_\varepsilon^2}
        \widehat{\operatorname{se}}_{\tau_{\mathrm{med}},\mathrm{G}}^2
        \dto
        w_Z^\top w_Z
        .
    \end{equation}
    The signed empirical and Gaussian-reference Wald statistics converge to $H/\sqrt{w_Z^\top\Omega w_Z}$ and $H/\sqrt{w_Z^\top w_Z}$, respectively.
    For $\Omega=I_2$ and $a_0=b_0=0$, both limits have law $N_1(0,1/4)$.
\end{proposition}

\begin{proof}[Proof of Proposition~\ref{prop:bcd-local-product}]
    Set $A_n=\sqrt{n}g_a(S)$ and $B_n=\sqrt{n}g_b(S)$.
    The derivatives of $g_a$ and $g_b$ at $I_3$ select the $s_{12}$ and $s_{23}$ coordinates, so $(A_n,B_n)^\top\dto N_2(0,\Omega)$.
    The canonical identities give
    \begin{equation}
        \label{eq:product-identity}
        \frac{n\sqrt{\nu_X}}{\sigma_\varepsilon}
        (\hat{\tau}-a_nb_n)
        =
        (\alpha_n+A_n)(\beta_n+B_n)-\alpha_n\beta_n
        .
    \end{equation}
    The first limit follows by continuous mapping.
    
    Differentiate the product functional with respect to $S$, keeping $a_n$, $b_n$ and $v_n$ fixed.
    Put $d_n(S)=\{\beta_n+\sqrt{n}g_b(S)\}\nabla g_a(S)+\{\alpha_n+\sqrt{n}g_a(S)\}\nabla g_b(S)$.
    The chain rule and covariance transport give
    \begin{equation*}
        \frac{n^2\nu_X}{\sigma_\varepsilon^2}
        \widehat{\operatorname{se}}_{\tau_{\mathrm{med}},\mathrm{emp}}^2
        =
        d_n(S)^\top\hat{\Gamma}_\xi d_n(S)
        ,\quad
        \frac{n^2\nu_X}{\sigma_\varepsilon^2}
        \widehat{\operatorname{se}}_{\tau_{\mathrm{med}},\mathrm{G}}^2
        =
        d_n(S)^\top\Gamma(S)d_n(S)
        .
    \end{equation*}
    The coefficients in $d_n(S)$ are tight, $S\pto I_3$ and $\hat{\Gamma}_\xi\pto\Gamma_\xi$.
    The covariance matrices of the two limiting derivative directions are $\Omega$ under $\Gamma_\xi$ and $I_2$ under $\Gamma(I_3)$.
    This proves both variance limits jointly with \eqref{eq:product-identity}.
    Since $\Omega\succ0$, the vector $w_Z$ has a density and vanishes with probability zero.
    Both limiting quadratic forms are therefore positive almost surely.
    Division gives the Wald limits, and the probabilities of zero estimated variances tend to zero.
    The event $S\succ0$ also has probability tending to one.
    
    For $\Omega=I_2$ and $a_0=b_0=0$, write $Z_1=R\cos\Theta$ and $Z_2=R\sin\Theta$, where $R=(Z_1^2+Z_2^2)^{1/2}$ is independent of the uniform angle $\Theta$ on $[0,2\pi)$.
    The signed Wald limit is $(R/2)\sin(2\Theta)$.
    The doubled angle modulo $2\pi$ is again uniform and independent of $R$, so this limit has law $N_1(0,1/4)$.
    The estimator limit in \eqref{eq:local-product-joint} remains $Z_1Z_2$.
    These limits exhibit the singular-hypothesis behavior of coefficient products \citep{drtonxiao2016}.
\end{proof}

\subsection{Moment inversion in common regression coordinates}
\label{supp:fieller}

Gaussian score orthogonality gives a samplewise expression for moment inversion, and the mediation boundary limit supplies its asymptotic coverage.

\begin{corollary}[Deterministic Fieller enlargement in regression coordinates]
    \label{cor:flag-fieller-enlargement}
    Let $D$ be a flag power with constant $c>0$ and let $N\in\RR[\Sym_d]$.
    Suppose $\tau=N/D$ depends differentiably only on the full-predecessor regression coefficients in coordinates adapted to the flag.
    For $0<\alpha<1$, let $q_{1-\alpha}$ be the $(1-\alpha)$ quantile of $\chi^2_1$.
    At $\Sighat\succ0$, put $\hat{\tau}=\tau(\Sighat)$ and $h_t=\nabla N(\Sighat)-t\nabla D(\Sighat)$.
    Let $\widehat{\operatorname{se}}_{\mathrm{G}}^2=\nabla\tau(\Sighat)^\top\Gamma(\Sighat)\nabla\tau(\Sighat)/n$ and define
    \begin{equation*}
        \cI_{1-\alpha}
        =
        \Bigl\{
            t\in\RR:
            n\{N(\Sighat)-tD(\Sighat)\}^2
            \le q_{1-\alpha}h_t^\top\Gamma(\Sighat)h_t
        \Bigr\}
        .
    \end{equation*}
    For $n>cq_{1-\alpha}$, put $b_{n,\alpha}=\{q_{1-\alpha}/(1-cq_{1-\alpha}/n)\}^{1/2}\widehat{\operatorname{se}}_{\mathrm{G}}$.
    Then
    \begin{equation}
        \label{eq:Fieller-interval}
        \cI_{1-\alpha}
        =
        \begin{cases}
            [\hat{\tau}-b_{n,\alpha},\ \hat{\tau}+b_{n,\alpha}],
            &n>cq_{1-\alpha},\\
            \RR,
            &n\le cq_{1-\alpha}
            .
        \end{cases}
    \end{equation}
    For $\tau_{\mathrm{med}}$, $c=10$, and this interval has asymptotic coverage $1-\alpha$ along the sequences of Proposition~\ref{prop:mediation-boundary-inference} (ii) when $\Omega=I_2$.
\end{corollary}

\begin{proof}[Proof of Corollary~\ref{cor:flag-fieller-enlargement}]
    Lemma~\ref{lem:bcd-regression-orthogonality} gives $\nabla\tau(S)^\top\Gamma(S)\nabla D(S)=0$ in adapted coordinates for every $S\succ0$.
    The congruence rule of Lemma~\ref{lem:trace-main} and the chain rule preserve this scalar pairing in the original coordinates.
    Differentiating $N=\tau D$ at $\Sighat$ gives $h_t=D(\Sighat)\nabla\tau(\Sighat)+(\hat{\tau}-t)\nabla D(\Sighat)$.
    Orthogonality and $s_D^2=cD^2$ imply
    \begin{equation*}
        h_t^\top\Gamma(\Sighat)h_t
        =
        D(\Sighat)^2
        \{n\widehat{\operatorname{se}}_{\mathrm{G}}^2+c(\hat{\tau}-t)^2\}
        .
    \end{equation*}
    Since $D(\Sighat)\ne0$, the inversion inequality is equivalent to
    \begin{equation*}
        (n-cq_{1-\alpha})(\hat{\tau}-t)^2
        \le
        q_{1-\alpha}n\widehat{\operatorname{se}}_{\mathrm{G}}^2
        .
    \end{equation*}
    For $n>cq_{1-\alpha}$ this gives the displayed interval, including a singleton when the estimated variance is zero.
    For $n\le cq_{1-\alpha}$ it holds for every $t\in\RR$.
    This is the quadratic moment inversion associated with Fieller inference \citep{fieller1954}.
    
    For mediation, $\tau=ab$ depends only on the common regression coefficients and $D=\Delta_1\Delta_2$, with $c=10$.
    Under $\Omega=I_2$, Proposition~\ref{prop:mediation-boundary-inference} (iii) gives a standard normal limit for the Gaussian-reference Wald statistic.
    The interval half-width is its usual Wald half-width multiplied by $(1-10q_{1-\alpha}/n)^{-1/2}\to1$.
    This proves the asymptotic coverage, since the probability of an undefined studentized statistic tends to zero.
\end{proof}

\subsection{Unknown means and intercepts}
\label{supp:mediation-centering}

We transfer the coefficient identities, boundary and local limits, and moment-inversion coverage to empirically centered observations.
Let $\bar{\xi}=n^{-1}\sum_{i=1}^n\xi_i$ and $S_c=n^{-1}\sum_i(\xi_i-\bar{\xi})(\xi_i-\bar{\xi})^\top=S-\bar{\xi}\bar{\xi}^\top$.
The divisor-$n$ centered covariance of observations with an arbitrary fixed mean added to \eqref{eq:mediation-model} is $\hat{\Sigma}_c=A_vS_cA_v^\top$.
The coefficient identities of Lemma~\ref{lem:mediation-canonical-reduction} hold with $S_c$ in place of $S$.

For empirical standard errors, use the centered quadratic scores $y_{i,c}=\vech\{(\xi_i-\bar{\xi})(\xi_i-\bar{\xi})^\top\}$, their mean $\bar{y}_c=\vech(S_c)$, and $\hat{\Gamma}_{\xi,c}=n^{-1}\sum_i(y_{i,c}-\bar{y}_c)(y_{i,c}-\bar{y}_c)^\top$.
The corresponding quadratic scores in observed coordinates are $T_{A_v}y_{i,c}$.
Their empirical covariance is therefore $T_{A_v}\hat{\Gamma}_{\xi,c}T_{A_v}^\top$ exactly.
Thus the standard-error identities hold with $(S_c,\hat{\Gamma}_{\xi,c})$ replacing $(S,\hat{\Gamma}_\xi)$.

Finite second moments give $\bar{\xi}=O_p(n^{-1/2})$, so $S_c-S=O_p(n^{-1})$.
The bounds in Lemma~\ref{lem:mediation-scaled-expansion} then give $\sup_{0\le v\le\bar{v}}\sqrt{n}|F_v(S_c)-F_v(S)|=o_p(1)$ on an event of probability tending to one.
Also $\hat{\Gamma}_{\xi,c}\pto\Gamma_\xi$ by Section~\ref{supp:fixed-law-centering}.
These facts preserve the boundary limits and their variance consistency.
For local coefficients, $\sqrt{n}\{g_a(S_c)-g_a(S)\}=o_p(1)$ and $\sqrt{n}\{g_b(S_c)-g_b(S)\}=o_p(1)$.
Substitution into \eqref{eq:product-identity} and the standard-error identities proves all the joint limits in \eqref{eq:local-product-joint} for centered sampling.
Since \eqref{eq:Fieller-interval} holds at every positive-definite sample covariance, its mediation coverage follows from the centered Gaussian-reference studentized limit as well.

The usual regression residual divisors are $n-2$ for $M$ on $(1,X)$ and $n-3$ for $Y$ on $(1,X,M)$.
Replacing the divisor-$n$ centered covariance by the unbiased one multiplies $S_c$ by $n/(n-1)$.
The coefficient estimates are unchanged by this common rescaling, and the asymptotic conclusions agree.

\section{Numerical designs, reference laws and results}
\label{supp:numerics}

This section gives the designs, interval calculations and numerical results underlying Section~\ref{sec:numerics}.
The comparisons separate covariance-free calibration, relative denominator control and precision of effect estimation.
The additional designs examine covariate adjustment, a precision functional and two forms of denominator instability.

\subsection{Design and Monte Carlo uncertainty}
\label{supp:design}

We specify the evaluation grids, the pairing of simulated samples and the measures of Monte Carlo uncertainty.
Observations are independent within each sample and have mean zero.
All sample covariances use the known-mean convention \eqref{eq:sample-covariance} with divisor $n$.
Every evaluation cell has $B=50000$ independent replications, and all confidence intervals have nominal coverage $0.95$.
Write $z_p=\Phi^{-1}(p)$ for $0<p<1$, where $\Phi$ is the standard normal distribution function.

The block experiment uses $r\in\{1,3\}$, $n\in\{50,200,1000\}$ and $\rho\in\{0,0.5,0.9,0.99\}$.
The mediation experiment uses each of the laws (G), (E) and (D) from Section~\ref{sec:sampling-laws}, with $n\in\{100,300,1000\}$ and $v\in\{1,n^{-1/2},n^{-1},n^{-2}\}$.
The $24\times24$ canonical grid uses law (G) and $n=300$ (Section~\ref{supp:illustration-canonical}).
The remaining grids are given in Section~\ref{supp:additional}.

Within a fixed block-size/sample-size pair, all covariance settings share standardized observations.
Within a fixed sampling-law/sample-size pair in the mediation experiment, all values of $v$ share standardized observations.
The canonical grid likewise uses one collection of standardized samples for all parameter points.
These common random numbers pair comparisons across settings while preserving independence between replications within each cell.
The random-number streams used for reference checks and local-limit calculations are independent of the evaluation samples.

For a binary outcome $I_b$ in replication $b$, estimate its probability by $\hat{p}=B^{-1}\sum_{b=1}^BI_b$ and its Monte Carlo standard error (MCSE) by $\{\hat{p}(1-\hat{p})/B\}^{1/2}$.
This applies to coverage, noncoverage in either tail and the excursion events defined below.
For a paired coverage difference, use the sample standard deviation of $I_{b,1}-I_{b,2}$ divided by $\sqrt{B}$, where the second index identifies the interval.
Quantile MCSEs are estimated from $25$ disjoint batches of $2000$ replications by dividing the sample standard deviation of the batch quantiles by $5$.
Coverage calculations retain all replications and count an undefined interval as noncoverage.
No statistic was undefined in the evaluation cells.
These conventions distinguish Monte Carlo uncertainty from the sampling variation of a statistic \citep{morris2019simulation}.

\subsection{Reference laws, intervals and numerical checks}
\label{supp:reference}

We specify the logarithmic reference laws, the intervals constructed from them and the numerical checks on their quantiles.

\subsubsection{Logarithmic reference laws}
\label{supp:reference-laws}
Let $s\ge1$, let $Y_1,\ldots,Y_s$ be independent with $Y_j\sim\chi^2_{\nu_j}$ and $\nu_j>0$, and take weights $m_j>0$.
Put $L=\sum_{j=1}^sm_j\log(Y_j/n)$ and define its characteristic function by $\varphi_L(t)=\EE\{\exp(\mathrm{i}tL)\}$, where $\mathrm{i}^2=-1$.
The gamma moment formula gives
\begin{equation*}
    \varphi_L(t)
    =
    \prod_{j=1}^s
    (2/n)^{\mathrm{i}tm_j}
    \frac{\Gamma_{\mathrm{gam}}(\nu_j/2+\mathrm{i}tm_j)}{\Gamma_{\mathrm{gam}}(\nu_j/2)}
    ,\quad
    t\in\RR
    .
\end{equation*}
Here $\Gamma_{\mathrm{gam}}$ is the gamma function, distinct from the quadratic-score covariance $\Gamma$.
For positive real arguments, put $\psi=\Gamma_{\mathrm{gam}}'/\Gamma_{\mathrm{gam}}$ and $\psi_1=\psi'$.
The mean $\mu_L=\EE(L)$ and variance $s_L^2=\var(L)$, with $s_L=\sqrt{\var(L)}>0$, are
\begin{equation*}
    \mu_L=\sum_{j=1}^sm_j\{\psi(\nu_j/2)+\log(2/n)\}
    ,\quad
    s_L^2=\sum_{j=1}^sm_j^2\psi_1(\nu_j/2)
    .
\end{equation*}

For the block determinant $D_{\mathrm{SN}}$, the pairs $(\nu_j,m_j)$ are $(n-j+1,1)$ for $1\le j\le2r$.
For $D_{\mathrm{st}}$ at $\rho=0$, use two independent copies of the list $(n-j+1,1)$, $1\le j\le r$.
Both first-order relative variances equal $4r$ at $\rho=0$, but their finite-sample reference laws differ.
For the Gaussian mediation denominator, the pairs are $(n,2)$ and $(n-1,1)$.
Sections~\ref{supp:numerical-adjusted} and~\ref{supp:numerical-precision} give the remaining lists.
Together these yield 24 reference specifications: 12 for the two block polynomials, three for mediation, six for the adjusted products and three for the precision reciprocal.

\subsubsection{Interval definitions}
\label{supp:numerical-intervals}
Let $q_p$ be the $p$ quantile of the chosen logarithmic reference law $L$.
For a positive functional $D$, its reference interval is
\begin{equation*}
    \left[
        D(\Sighat)\exp(-q_{0.975}),\
        D(\Sighat)\exp(-q_{0.025})
    \right]
    .
\end{equation*}
The coverage event is $q_{0.025}\le\log\{D(\Sighat)/D(\Sigma)\}\le q_{0.975}$.
When $L$ is the actual log-ratio law, this interval has exact coverage $0.95$.
Using the reference law at $\rho=0$ at other covariances is the transported-reference calibration of Section~\ref{sec:calibration}.
Lower and upper noncoverage refer to the events below and above this logarithmic reference range, respectively.

The uncorrected log-delta interval is $D(\Sighat)\exp\{\pm z_{0.975}\sqrt{\kappa_D(\Sighat)/n}\}$.
For a flag power, $\kappa_D(\Sighat)=c$.
The log-moment normal interval is $D(\Sighat)\exp\{-\mu_L\pm z_{0.975}s_L\}$, using the exact mean and variance of the specified log-ratio law.
The notation $x\exp\{a\pm b\}$ denotes the closed interval with endpoints $x\exp(a-b)$ and $x\exp(a+b)$ for $x>0$ and $b\ge0$.
The plug-in comparator for a nonnested product is its log-delta interval with $\kappa_D$ evaluated at $\Sighat$.
For mediation under a non-Gaussian law, the empirical log-delta interval replaces $\kappa_D(\Sighat)$ by the empirical relative variance defined in Section~\ref{supp:illustration-mediation}.

\subsubsection{Numerical quantiles and checks}
\label{supp:reference-checks}
We invert $\exp(-\mathrm{i}t\mu_L)\varphi_L(t)$ by a discrete Fourier transform and translate the result by $\mu_L$.
The interval for $L$ is centered at $\mu_L$ and has total width $\max(8,40s_L)$.
A $2^{17}$-point grid gives the numerical density, cumulative Simpson integration gives its distribution function, and monotone cubic interpolation gives its $0.025$, $0.5$ and $0.975$ quantiles.
Repeating the calculation on a $2^{16}$-point grid changes these quantiles by at most $1.42\times10^{-11}$ across the 24 reference specifications.
Doubling the interval width and using $2^{18}$ points gives a maximum change of $1.27\times10^{-12}$.
The integrated negative part of the numerical density is at most $1.53\times10^{-12}$ and is set to zero before normalization.
For one-factor laws, comparison with chi-squared quantiles gives a maximum log-quantile discrepancy of $1.57\times10^{-12}$.
These comparisons assess sensitivity to the numerical grid and interval.

A separate check uses $2000000$ independent chi-squared products for each reference specification.
At each of the three computed quantiles, the empirical distribution function is compared with its nominal probability using the binomial MCSE.
The largest absolute standardized discrepancy across these 72 checks is $2.464$.
These random draws check the numerically evaluated reference quantiles independently of the observations used to estimate interval coverage.

\subsection{Block covariance and transported calibration}
\label{supp:block}

We derive the transported-reference coverage limit and describe the evaluation of the block determinants and their relative variances.
Use the two blocks $I_1,I_2$ of size $r$ from Section~\ref{sec:calibration}, with $D_{\mathrm{SN}}=\det\Sigma$ and $D_{\mathrm{st}}=\det\Sigma_{I_1}\det\Sigma_{I_2}$.
Separate invertible transformations within the blocks preserve both sample-to-population ratios.
Transforming each population block covariance to the identity and then applying a singular-value decomposition to the cross block reduces the Gaussian law to the canonical correlations $\rho_1,\ldots,\rho_r$.
These are the singular values of $\Sigma_{I_1}^{-1/2}\Sigma_{I_1,I_2}\Sigma_{I_2}^{-1/2}$.
The equal-correlation design has $\rho_j=\rho$ for every $j$.
Lemmas~\ref{lem:flag-structure} and~\ref{lem:two-subspace} give
\begin{equation*}
    \kappa_{D_{\mathrm{SN}}}
    =4r
    ,\quad
    \kappa_{D_{\mathrm{st}}}
    =4r+4\tr(\Sigma_{I_1}^{-1}\Sigma_{I_1,I_2}\Sigma_{I_2}^{-1}\Sigma_{I_2,I_1})
    =4r+4\sum_{j=1}^r\rho_j^2
    .
\end{equation*}
The sample trace is evaluated by solving linear systems with the two sample diagonal blocks.

At $\rho=0$, the data blocks are independent, so separate Bartlett decompositions give the reference law for $D_{\mathrm{st}}$ in Section~\ref{supp:reference-laws}.
Let $R_{n,D}=D(\Sighat)/D(\Sigma)$ and let $q^{(0)}_{n,D,p}$ be the $p$ quantile of $\log R_{n,D}$ at $\rho=0$.
For $D=D_{\mathrm{st}}$ and fixed $r$, $\sqrt{n}q^{(0)}_{n,D,p}\to z_p\sqrt{4r}$, whereas $\sqrt{n}\log R_{n,D}\dto N_1\{0,4r(1+\rho^2)\}$ at correlation $\rho$.
These limits give \eqref{eq:transport-limit}.
The standard-deviation ratio is $\sqrt{1+\rho^2}$, which approaches $\sqrt{2}$ as $\rho\to1$.
At $\rho=0.99$, the limiting transported-reference coverage is approximately $0.8363350$.

Log determinants are evaluated by Cholesky factorization and summing the logarithms of the squared diagonal entries.
A check using fixed nonorthogonal transformations within the blocks gives log-ratio discrepancies below $9.8\times10^{-14}$ for $D_{\mathrm{SN}}$ and $1.0\times10^{-15}$ for $D_{\mathrm{st}}$.
Under the common-random-number construction, the normalized determinant $R_{n,D_{\mathrm{SN}}}$ coincides across $\rho$ to numerical precision.
Its repeated coverage entries therefore describe the same Monte Carlo estimate.
In Figure~\ref{fig:calibration} (a), markers show medians and whiskers show the $0.025$ and $0.975$ quantiles of $\sqrt{n}\log R_{n,D}$.
The coverage bars in Figure~\ref{fig:calibration} (b) instead show $\pm1.96$ MCSEs.

\subsection{Mediation scores, precision and sampling-law benchmarks}
\label{supp:illustration-mediation}

We give the score-based calculations used to evaluate the mediation intervals and derive the benchmarks for the three sampling laws.
In the fixed-coefficient designs, $\nu_X=1$, $\sigma_\varepsilon=\sqrt{0.75}$, $a=0.8$, $b=1$ and $\gamma=0.5$.
Write $\xi_i=(U_i,V_i,W_i)^\top=(\xi_{Xi},\xi_{Mi},\xi_{Yi})^\top$ for the standardized observations and $S=n^{-1}\sum_i\xi_i\xi_i^\top=(s_{ij})$.
Thus $W_i$ is the standardized outcome innovation, and the observed covariance is $\Sighat=A_vSA_v^\top$ with $A_v$ from \eqref{eq:loading}.
For $J\in\{\mathrm{G},\mathrm{emp}\}$, abbreviate $\widehat{\operatorname{se}}_{\tau_{\mathrm{med}},J}$ to $\widehat{\operatorname{se}}_J$.

On $S\succ0$, write $g_a=g_a(S)=s_{12}/s_{11}$ and put $r_i=V_i-g_aU_i$ and $v_r=n^{-1}\sum_i r_i^2$.
The second coefficient and its residual are
\begin{equation*}
    g_b=g_b(S)=\frac{\sum_i r_iW_i}{\sum_i r_i^2}
    ,\quad
    e_i=W_i-\frac{s_{13}}{s_{11}}U_i-g_b r_i
    ,\quad
    v_e=n^{-1}\sum_i e_i^2
    .
\end{equation*}
These residual variances equal $r_M(S)$ and $r_Y(S)$ in Section~\ref{supp:canonical}.
The coefficient influence scores evaluated at the sample are $a_i^\ast=U_i r_i/s_{11}$ and $b_i^\ast=r_i e_i/v_r$.
Let $q_{aa},q_{bb},q_{ab}$ be their empirical covariance entries with divisor $n$.
The normal equations make both score means zero, up to numerical roundoff.
The Gaussian-reference entries are $q_{aa,\mathrm{G}}=v_r/s_{11}$, $q_{bb,\mathrm{G}}=v_e/v_r$ and $q_{ab,\mathrm{G}}=0$.

The coefficient identities give $\hat{a}=a+\sqrt{v}g_a$ and $\hat{b}=b+\sigma_\varepsilon g_b/\sqrt{v}$.
With $\nu_X=1$, the scaled error in \eqref{eq:scaled-Fv} is $F_v(S)=a\sigma_\varepsilon g_b+bvg_a+\sigma_\varepsilon\sqrt{v}g_ag_b$.
Put $c_a=v\hat{b}$ and $c_b=\sigma_\varepsilon\hat{a}$.
Then $\hat{\tau}-ab=F_v(S)/\sqrt{v}$, and \eqref{eq:scaled-sandwich} becomes
\begin{equation*}
    nv\widehat{\operatorname{se}}_{\mathrm{emp}}^2
    =c_a^2q_{aa}+2c_ac_bq_{ab}+c_b^2q_{bb}
    ,\quad
    nv\widehat{\operatorname{se}}_{\mathrm{G}}^2
    =c_a^2q_{aa,\mathrm{G}}+c_b^2q_{bb,\mathrm{G}}
    .
\end{equation*}
These identities evaluate the original estimator and standard errors without subtracting nearly collinear raw covariance terms.
Effect coverage is the empirical frequency of $|\hat{\tau}-ab|\le z_{0.975}\widehat{\operatorname{se}}_J$.

For $D_{\mathrm{med}}$, the sample-to-population ratio is $s_{11}^2v_r$ by Section~\ref{supp:mediation-denominator-removal}.
Its empirical logarithmic score is $d_i^\ast=2(U_i^2/s_{11}-1)+(r_i^2/v_r-1)$, with empirical mean zero.
Thus $\hat{v}_{D_{\mathrm{med}}}=n^{-1}\sum_i(d_i^\ast)^2$ gives the empirical log-delta interval $D_{\mathrm{med}}(\Sighat)\exp\{\pm z_{0.975}\sqrt{\hat{v}_{D_{\mathrm{med}}}/n}\}$.
The Gaussian log-delta interval uses the constant $10$ in place of $\hat{v}_{D_{\mathrm{med}}}$.

\emph{Independent errors.}
Under (G), $U,V,W$ are independent standard normals.
Under (E), they are independent copies of $E_0-1$, where $E_0\sim\operatorname{Exp}(1)$ has rate one.
The identity $\EE(E_0^k)=k!$ for integers $k\ge0$ gives variance one and fourth moment nine for $E_0-1$.
The population relative score of $D_{\mathrm{med}}$ is $2(U^2-1)+(V^2-1)$, whose variance under (E) is $4(9-1)+(9-1)=40$.
The Gaussian reference has first-order relative variance $10$, so both its reference-law interval and its uncorrected log-delta interval have limiting coverage
\begin{equation}
    \label{eq:E-benchmark}
    2\Phi(z_{0.975}\sqrt{10/40})-1\approx0.6729050
    .
\end{equation}
For both (G) and (E), $\Omega=\cov\{(UV,VW)^\top\}=I_2$.
Proposition~\ref{prop:mediation-boundary-inference} (ii)--(iii) therefore gives limiting coverage $0.95$ for both effect Wald intervals along the fixed-coefficient paths.

\emph{Dependent errors.}
Under (D), $U,V,R$ are independent standard normals and $W=VR$.
The vector $(U,V,W)^\top$ has mean zero and covariance $I_3$.
Also $\var(UV)=1$, $\var(VW)=\EE(V^4)\EE(R^2)=3$ and $\cov(UV,VW)=0$, so $\Omega=\diag(1,3)$.
The denominator depends only on the Gaussian pair $(U,V)$ and retains its Gaussian relative law.
The scaled effect score $L_v=bv\,UV+a\sigma_\varepsilon VW$ instead has variance $v^2+1.44$, whereas its Gaussian-reference variance is $v^2+0.48$.
Consequently the Gaussian-reference Wald statistic has limiting variance $\lambda(v)=(v^2+1.44)/(v^2+0.48)$ at fixed $v$.
The corresponding coverage limits are
\begin{equation}
    \label{eq:D-benchmark}
    2\Phi\{z_{0.975}/\sqrt{\lambda(1)}\}-1
    \approx
    0.8731034
    ,\quad
    2\Phi(z_{0.975}/\sqrt{3})-1
    \approx
    0.7421914
    ,
\end{equation}
for $v=1$ and $v_n\to0$, respectively.
The empirical standard error consistently estimates the actual score variance, so its effect interval has limiting coverage $0.95$.

Table~\ref{tab:denominator} also records the finite-sample accuracy of empirical log-delta calibration.
Under (E), its denominator coverage increases from $0.8721$ at $n=100$ to $0.9374$ at $n=1000$.
Thus estimating the actual fourth-moment variance restores asymptotic calibration while leaving appreciable finite-sample error in this design.

\subsection{Canonical coefficient grid and moment inversion}
\label{supp:illustration-canonical}

We evaluate the finite-sample and local-limit Wald coverages on the canonical coefficient grid and record the moment-inversion comparison.
The finite-sample grid uses law (G) at $n=300$.
The coordinates are $h_a=a\sqrt{n\nu_X/v}$ and $h_b=b\sqrt{nv}/\sigma_\varepsilon$, with $\nu_X=1$ and $\sigma_\varepsilon=\sqrt{0.75}$.
The grid is $\cH\times\cH$, where
\begin{equation*}
    \begin{split}
        \cH=\{&0,0.25,0.5,1,2,3,4,6,8,12,16,24,32,48,64,96,\\
        &128,192,256,512,1024,2048,4096,8192\}.
    \end{split}
\end{equation*}
Let $A_n=\sqrt{n}g_a(S)$ and $B_n=\sqrt{n}g_b(S)$.
The normalized error is $H_n=n\sqrt{\nu_X}(\hat{\tau}-ab)/\sigma_\varepsilon=h_aB_n+h_bA_n+A_nB_n$ by \eqref{eq:product-identity}.
Put $\hat{s}_{H,J}=n\sqrt{\nu_X}\widehat{\operatorname{se}}_J/\sigma_\varepsilon\ge0$ for $J\in\{\mathrm{G},\mathrm{emp}\}$.
The empirical version is
\begin{equation*}
    \hat{s}_{H,\mathrm{emp}}^2
    =(h_b+B_n)^2q_{aa}+(h_a+A_n)^2q_{bb}
    +2(h_b+B_n)(h_a+A_n)q_{ab}
    .
\end{equation*}
For the Gaussian-reference version, use $q_{aa,\mathrm{G}}$, $q_{bb,\mathrm{G}}$ and $q_{ab,\mathrm{G}}$.
For each version, coverage is the empirical frequency of $|H_n|\le z_{0.975}\hat{s}_{H,J}$.
The same $50000$ standardized samples evaluate all grid points.
Figure~\ref{fig:mediation} (b) shows the empirical-standard-error version over the full grid, and Table~\ref{tab:canonical} reports selected points.

For the limit comparison, a separate stream supplies $10^6$ independent standard normal pairs $(Z_1,Z_2)$.
Proposition~\ref{prop:bcd-local-product} gives the limiting signed Wald statistic
\begin{equation}
    \label{eq:local-Wald}
    W_\infty(h_a,h_b)
    =
    \frac{h_aZ_2+h_bZ_1+Z_1Z_2}
    {\{(h_a+Z_1)^2+(h_b+Z_2)^2\}^{1/2}}
    .
\end{equation}
At each grid point, estimate $\Pr\{|W_\infty(h_a,h_b)|\le z_{0.975}\}$ from these pairs.
At $(0,0)$, the law is $N_1(0,1/4)$, with exact central probability $2\Phi(2z_{0.975})-1\approx0.9999114$.
At $(2,2)$, the simulated local-limit coverage is $0.909147$ with MCSE $0.0002874$.
Table~\ref{tab:canonical-limit} gives the selected comparisons, showing both overcoverage and undercoverage in local coefficient regimes.
The dashed path in Figure~\ref{fig:mediation} (b) fixes $a=0.8$ and $b=1$ while $v$ ranges from $1$ to $300^{-2}$.
It satisfies $h_ah_b=300(0.8)/\sqrt{0.75}$.

For $n>10z_{0.975}^2$, Corollary~\ref{cor:flag-fieller-enlargement} enlarges the Gaussian-reference Wald interval by the deterministic width factor $(1-10z_{0.975}^2/n)^{-1/2}$.
This factor is $1.2743$, $1.0709$ and $1.0198$, to four decimals, at $n=100,300,1000$.
Table~\ref{tab:precision-errors} reports the resulting Gaussian-design coverage.
Because the factor tends to one, under law (D) the fixed-coefficient coverage limits remain those in \eqref{eq:D-benchmark}.

\subsection{Additional covariance and boundary comparisons}
\label{supp:additional}

We extend the calibration comparisons to adjusted products and a precision functional, then contrast a positive unstable denominator with a first-order weak covariance.

\subsubsection{Covariate-adjusted determinant products}
\label{supp:numerical-adjusted}
Let $q=4$, let $C=(C_1,\ldots,C_q)^\top\sim N_q(0,\Gamma_C)$ with $(\Gamma_C)_{ij}=0.5^{|i-j|}$, and take independent standard normals $U,V$ independent of $C$.
In the coordinate order $(C_1,\ldots,C_q,X_1,X_2)$, set
\begin{equation*}
    X_1=\alpha^\top C+U
    ,\quad
    X_2=\beta^\top C+\sqrt{v_2}\{\rho U+\sqrt{1-\rho^2}V\}
    ,
\end{equation*}
where $\alpha,\beta\in\RR^q$ are fixed vectors with entries $\alpha_i=i/10$ and $\beta_i=(-1)^{i-1}i/10$ for $1\le i\le q$.
Use $\rho\in\{0,0.5,0.9,0.99\}$, $v_2\in\{1,10^{-6}\}$ and $n\in\{50,200,1000\}$.
In minor subscripts, $C$ denotes its set of coordinate indices.
Compare the nested product $\Delta_C\Delta_{C\cup\{X_1,X_2\}}$ with the nonnested product $\Delta_{C\cup\{X_1\}}\Delta_{C\cup\{X_2\}}$.
Both have degree ten, and the nested product has multiplicities $(2,2,2,2,1,1)$.

For the logarithmic reference laws of Section~\ref{supp:reference-laws}, both lists start with $(n-j+1,2)$ for $j=1,\ldots,4$.
The nested list ends with $(n-4,1),(n-5,1)$.
The nonnested list at $\rho=0$ ends with two independent copies of $(n-4,1)$, independent also of the preceding factors.
This follows from Gaussian projection on the covariate sample and independence of the two residual columns at $\rho=0$.
The first-order relative variances are $36$ and $36+4\rho^2$, respectively.
Thus the transported-reference coverage limit for the nonnested product is $2\Phi\{z_{0.975}/\sqrt{1+\rho^2/9}\}-1$, tending to approximately $0.937$ as $\rho\to1$.

For numerical evaluation, let $S_{\mathrm{res}}=\Sighat_{(X_1,X_2)\cdot C}$ be the sample residual covariance given $C$ and put $\hat{\Delta}_C=\det\Sighat_C$.
The two sample polynomials are $\hat{\Delta}_C^2\det S_{\mathrm{res}}$ and $\hat{\Delta}_C^2(S_{\mathrm{res}})_{11}(S_{\mathrm{res}})_{22}$, by Section~\ref{supp:adjusted-mediation} and the same Schur-complement identity.
Their logarithms retain the full term $2\log\hat{\Delta}_C$.
Changing $v_2$ preserves the relative laws, and all paired coverage entries in Table~\ref{tab:adjusted} agree to four decimals across the two scales.

\subsubsection{A reciprocal precision functional}
\label{supp:numerical-precision}
In the six-dimensional block design, put $\theta(\Sigma)=1/(\1_6^\top\Sigma^{-1}\1_6)$ and $\hat{\theta}=\theta(\Sighat)$, where $\1_6$ is the all-ones vector.
This functional is the minimum of $w^\top\Sigma w$ subject to $\1_6^\top w=1$.
Corollary~\ref{cor:bcd-precision} gives $\hat{\theta}/\theta\stackrel{d}{=}\chi^2_{n-5}/n$ and $\kappa_\theta=2$ for $n\ge6$.
Thus the logarithmic reference list has the single pair $(n-5,1)$.
Apply the intervals of Section~\ref{supp:numerical-intervals} with $\theta$ in place of $D$.
The reciprocal is evaluated by solving $\Sighat x=\1_6$ and taking $1/(\1_6^\top x)$.
Each of the 12 cells with $n\in\{50,200,1000\}$ and $\rho\in\{0,0.5,0.9,0.99\}$ uses the observations from its block comparison.

\subsubsection{The Napkin denominator and a matched flag power}
\label{supp:numerical-napkin}
Use the covariance path of Example~\ref{ex:bcd-napkin} with nonzero structural coefficient $\alpha=0.4$.
Put $\delta=\det\Sigma_{\{Z,W\}}$, $T=\det\Sigma_{\{Z,W,X\}}$ and $m=\sigma_{ZZ}\sigma_{WX}-\sigma_{ZW}\sigma_{ZX}$.
The positive polynomials $D_0=\delta T$ and $D_{\mathrm{nap}}=\delta T+\sigma_{WW}m^2$ both have degree five.
Along the path, their values are $t^2$ and $t^2(t+2)$.
The flag power $D_0$ has multiplicities $(2,2,1)$ and constant $18$.
We use $t_n=1/n$ with $n\in\{50,200,1000,5000\}$.

Let $S=(s_{ij})$ be the known-mean sample covariance of the independent standard normals $(U,V,R)$ in that example.
Set $\Delta=\det S_{1:2}$, $T_S=\det S$, $K_S=s_{11}s_{23}-s_{12}s_{13}$ and $\hat{w}_t=s_{11}+2\sqrt{t}s_{12}+ts_{22}$.
For $D\in\{D_0,D_{\mathrm{nap}}\}$, define the sample-to-population ratio at $\Sigma_t$ as $D(\Sighat)/D(\Sigma_t)$ and denote the two ratios by $R_{n,0}$ and $R_{n,\mathrm{nap}}$.
The sample identities in Section~\ref{supp:napkin} give
\begin{equation*}
    R_{n,0}=\Delta T_S
    ,\quad
    R_{n,\mathrm{nap}}
    =
    \frac{\Delta T_S+\hat{w}_t(\Delta+K_S/\sqrt{t})^2}{2+t}
    .
\end{equation*}
These expressions retain the full raw-covariance polynomials and are independent of $\alpha$.
At $t_n=1/n$, $\sqrt{n}K_S\dto Z_0\sim N_1(0,1)$, while $\Delta$, $T_S$ and $\hat{w}_{t_n}$ converge to one.
Hence $R_{n,\mathrm{nap}}\dto R_\infty=\{1+(1+Z_0)^2\}/2$, as in \eqref{eq:napkin-relative-limit}, and
\begin{equation}
    \label{eq:napkin-probability}
    \begin{split}
        \Pr(|R_{n,\mathrm{nap}}-1|>1/2)
        &\longrightarrow
        1-\Phi(\sqrt{2}-1)+\Phi(-\sqrt{2}-1)
        \\
        &\approx0.3472435
        .
    \end{split}
\end{equation}
The limiting law has support $[1/2,\infty)$ and no atom at $1/2$.
Thus the lower excursion probability tends to zero, and the upper excursion probability tends to the value in \eqref{eq:napkin-probability}.

At $n=5000$, the estimated lower and upper excursion probabilities are $0.0503$ and $0.3480$, and the estimated total is $0.3982$.
The upper estimate is close to its limit, while the lower tail accounts for the remaining discrepancy in the total.
For the matched flag power, the estimated total is $0.00076$ at $n=1000$ and no excursion occurs in the $50000$ replications at $n=5000$.
Table~\ref{tab:napkin} and Figure~\ref{fig:napkin} display these probabilities.
Figure~\ref{fig:napkin-qq} compares theoretical quantiles of $R_\infty$ on the horizontal axis with empirical quantiles of $R_{n,\mathrm{nap}}$ on the vertical axis.

\subsubsection{A first-order weak covariance}
\label{supp:numerical-weak}
For $D=\sigma_{12}$, use independent observations with law $N_2(0,\Sigma_n)$, where $\Sigma_n=\left(\begin{smallmatrix}1&\rho_n\\\rho_n&1\end{smallmatrix}\right)$ and $\rho_n=h/\sqrt{n}$.
Take $h\in\{1,2\}$ and $n\in\{50,200,1000,5000\}$.
Proposition~\ref{prop:transfer} gives $D(\Sighat)/D(\Sigma_n)\dto1+Z_0/h$ and $\hat{F}_D\dto(h+Z_0)^2$, where $Z_0\sim N_1(0,1)$.
Consequently, $\Pr\{|D(\Sighat)/D(\Sigma_n)-1|>1/2\}\to2\Phi(-h/2)$, whose values are $0.6171$ and $0.3173$ for $h=1$ and $h=2$, respectively.
This design uses its own standardized samples, shared across the two values of $h$ at each sample size.
Table~\ref{tab:weak} reports the estimated excursion probabilities and relative-ratio quantiles.

\subsection{Tables and additional figures}
\label{supp:tables}

The tables give all block, fixed-coefficient mediation and adjustment settings, selected canonical points and the additional boundary comparisons.
Tables~\ref{tab:block} and~\ref{tab:corrections} report block calibration.
Tables~\ref{tab:mediation}--\ref{tab:canonical-limit} report effect inference, denominator calibration and local coefficient regimes.
Tables~\ref{tab:adjusted}--\ref{tab:weak} give the additional designs.
Probabilities and quantiles are estimated from $B=50000$ replications unless a caption specifies otherwise.
Rounding can make displayed tail probabilities or paired differences differ from arithmetic performed on the displayed components.

\begin{table}[tbp]
\centering
\caption{
Block calibration over the full grid.
The columns $D_{\mathrm{SN}}$ and $D_{\mathrm{st}}$ give coverage using each polynomial's own reference law at $\rho=0$.
Plug-in is the covariance-adapted log-delta coverage for $D_{\mathrm{st}}$.
Lower and Upper are the frequencies of $\log R_{n,D_{\mathrm{st}}}$ below and above its transported central reference range, respectively.
MCSE refers to transported-reference coverage for $D_{\mathrm{st}}$.
}
\label{tab:block}
\begin{tabular}{@{}lrrrrrrrr@{}}
\toprule
$r$ & $n$ & $\rho$ & $D_{\mathrm{SN}}$ & $D_{\mathrm{st}}$ & Plug-in & Lower & Upper & MCSE \\
\midrule
$1$ & $50$ & $0$ & $0.9505$ & $0.9503$ & $0.9476$ & $0.0248$ & $0.0249$ & $0.0010$ \\
$1$ & $50$ & $0.5$ & $0.9505$ & $0.9222$ & $0.9445$ & $0.0373$ & $0.0406$ & $0.0012$ \\
$1$ & $50$ & $0.9$ & $0.9505$ & $0.8558$ & $0.9451$ & $0.0698$ & $0.0744$ & $0.0016$ \\
$1$ & $50$ & $0.99$ & $0.9505$ & $0.8373$ & $0.9470$ & $0.0793$ & $0.0834$ & $0.0017$ \\
$1$ & $200$ & $0$ & $0.9495$ & $0.9502$ & $0.9497$ & $0.0252$ & $0.0246$ & $0.0010$ \\
$1$ & $200$ & $0.5$ & $0.9495$ & $0.9223$ & $0.9500$ & $0.0377$ & $0.0400$ & $0.0012$ \\
$1$ & $200$ & $0.9$ & $0.9495$ & $0.8566$ & $0.9482$ & $0.0701$ & $0.0733$ & $0.0016$ \\
$1$ & $200$ & $0.99$ & $0.9495$ & $0.8372$ & $0.9484$ & $0.0794$ & $0.0834$ & $0.0017$ \\
$1$ & $1000$ & $0$ & $0.9498$ & $0.9496$ & $0.9496$ & $0.0246$ & $0.0259$ & $0.0010$ \\
$1$ & $1000$ & $0.5$ & $0.9498$ & $0.9201$ & $0.9489$ & $0.0390$ & $0.0409$ & $0.0012$ \\
$1$ & $1000$ & $0.9$ & $0.9498$ & $0.8545$ & $0.9488$ & $0.0711$ & $0.0744$ & $0.0016$ \\
$1$ & $1000$ & $0.99$ & $0.9498$ & $0.8369$ & $0.9493$ & $0.0800$ & $0.0832$ & $0.0017$ \\
$3$ & $50$ & $0$ & $0.9489$ & $0.9495$ & $0.9241$ & $0.0256$ & $0.0248$ & $0.0010$ \\
$3$ & $50$ & $0.5$ & $0.9489$ & $0.9196$ & $0.9219$ & $0.0395$ & $0.0409$ & $0.0012$ \\
$3$ & $50$ & $0.9$ & $0.9489$ & $0.8558$ & $0.9275$ & $0.0707$ & $0.0735$ & $0.0016$ \\
$3$ & $50$ & $0.99$ & $0.9489$ & $0.8361$ & $0.9296$ & $0.0807$ & $0.0832$ & $0.0017$ \\
$3$ & $200$ & $0$ & $0.9500$ & $0.9500$ & $0.9440$ & $0.0242$ & $0.0259$ & $0.0010$ \\
$3$ & $200$ & $0.5$ & $0.9500$ & $0.9214$ & $0.9437$ & $0.0384$ & $0.0402$ & $0.0012$ \\
$3$ & $200$ & $0.9$ & $0.9500$ & $0.8549$ & $0.9447$ & $0.0715$ & $0.0736$ & $0.0016$ \\
$3$ & $200$ & $0.99$ & $0.9500$ & $0.8381$ & $0.9453$ & $0.0800$ & $0.0819$ & $0.0016$ \\
$3$ & $1000$ & $0$ & $0.9490$ & $0.9491$ & $0.9481$ & $0.0255$ & $0.0253$ & $0.0010$ \\
$3$ & $1000$ & $0.5$ & $0.9490$ & $0.9214$ & $0.9497$ & $0.0397$ & $0.0389$ & $0.0012$ \\
$3$ & $1000$ & $0.9$ & $0.9490$ & $0.8546$ & $0.9494$ & $0.0735$ & $0.0719$ & $0.0016$ \\
$3$ & $1000$ & $0.99$ & $0.9490$ & $0.8360$ & $0.9492$ & $0.0836$ & $0.0804$ & $0.0017$ \\
\bottomrule
\end{tabular}
\end{table}

\begin{table}[tbp]
\centering
\caption{
Coverage of intervals for the generalized variance $D_{\mathrm{SN}}=\det\Sigma$ in dimension $2r$.
Exact, Delta and Log moments are the reference-law, uncorrected log-delta and log-moment normal intervals of Section~\ref{supp:numerical-intervals}.
Difference is the paired Exact-minus-Delta coverage difference, and MCSE is its Monte Carlo standard error.
The relative law is independent of $\rho$.
}
\label{tab:corrections}
\begin{tabular}{@{}lrrrrrr@{}}
\toprule
$r$ & $n$ & Exact & Delta & Log moments & Difference & MCSE \\
\midrule
$1$ & $50$ & $0.9505$ & $0.9414$ & $0.9501$ & $0.0091$ & $0.0008$ \\
$1$ & $200$ & $0.9495$ & $0.9480$ & $0.9500$ & $0.0015$ & $0.0006$ \\
$1$ & $1000$ & $0.9498$ & $0.9495$ & $0.9500$ & $0.0003$ & $0.0004$ \\
$3$ & $50$ & $0.9489$ & $0.8441$ & $0.9493$ & $0.1049$ & $0.0017$ \\
$3$ & $200$ & $0.9500$ & $0.9262$ & $0.9500$ & $0.0238$ & $0.0011$ \\
$3$ & $1000$ & $0.9490$ & $0.9445$ & $0.9490$ & $0.0045$ & $0.0007$ \\
\bottomrule
\end{tabular}
\end{table}

\begin{table}[tbp]
\centering
\caption{
Mediation under the three sampling laws in Section~\ref{supp:illustration-mediation}.
$C_{\mathrm{G}}$ and $C_{\mathrm{emp}}$ are Wald coverage estimates using Gaussian-reference and empirical standard errors.
$T_p$ is the estimated $p$ quantile of $(\hat{\tau}-ab)/\widehat{\operatorname{se}}_{\mathrm{emp}}$.
Median SE is the median empirical standard error, and MCSE refers to $C_{\mathrm{emp}}$.
}\label{tab:mediation}
\begin{tabular}{@{}lrrrrrrrr@{}}
\toprule
Law & $n$ & $v$ & $C_{\mathrm{G}}$ & $C_{\mathrm{emp}}$
& $T_{0.025}$ & $T_{0.975}$ & \shortstack{Median\\SE} & MCSE\\
\midrule
G & $100$ & $1$ & $0.9441$ & $0.9373$ & $-2.2526$ & $1.8998$ & $0.1197$ & $0.0011$ \\
G & $100$ & $n^{-1/2}$ & $0.9458$ & $0.9381$ & $-2.1029$ & $2.0252$ & $0.2163$ & $0.0011$ \\
G & $100$ & $n^{-1}$ & $0.9451$ & $0.9375$ & $-2.0773$ & $2.0666$ & $0.6768$ & $0.0011$ \\
G & $100$ & $n^{-2}$ & $0.9451$ & $0.9376$ & $-2.0759$ & $2.0693$ & $6.7677$ & $0.0011$ \\
G & $300$ & $1$ & $0.9483$ & $0.9456$ & $-2.0810$ & $1.9158$ & $0.0698$ & $0.0010$ \\
G & $300$ & $n^{-1/2}$ & $0.9474$ & $0.9448$ & $-2.0033$ & $2.0068$ & $0.1656$ & $0.0010$ \\
G & $300$ & $n^{-1}$ & $0.9474$ & $0.9448$ & $-1.9937$ & $2.0130$ & $0.6869$ & $0.0010$ \\
G & $300$ & $n^{-2}$ & $0.9474$ & $0.9446$ & $-1.9942$ & $2.0141$ & $11.8952$ & $0.0010$ \\
G & $1000$ & $1$ & $0.9503$ & $0.9493$ & $-2.0236$ & $1.9153$ & $0.0384$ & $0.0010$ \\
G & $1000$ & $n^{-1/2}$ & $0.9506$ & $0.9500$ & $-1.9700$ & $1.9552$ & $0.1230$ & $0.0010$ \\
G & $1000$ & $n^{-1}$ & $0.9506$ & $0.9499$ & $-1.9687$ & $1.9565$ & $0.6911$ & $0.0010$ \\
G & $1000$ & $n^{-2}$ & $0.9507$ & $0.9499$ & $-1.9685$ & $1.9564$ & $21.8537$ & $0.0010$ \\
E & $100$ & $1$ & $0.9484$ & $0.9210$ & $-2.5181$ & $1.8166$ & $0.1136$ & $0.0012$ \\
E & $100$ & $n^{-1/2}$ & $0.9461$ & $0.9213$ & $-2.4654$ & $1.8941$ & $0.2033$ & $0.0012$ \\
E & $100$ & $n^{-1}$ & $0.9454$ & $0.9227$ & $-2.4208$ & $1.9095$ & $0.6353$ & $0.0012$ \\
E & $100$ & $n^{-2}$ & $0.9452$ & $0.9226$ & $-2.4172$ & $1.9117$ & $6.3536$ & $0.0012$ \\
E & $300$ & $1$ & $0.9498$ & $0.9356$ & $-2.3408$ & $1.7907$ & $0.0674$ & $0.0011$ \\
E & $300$ & $n^{-1/2}$ & $0.9482$ & $0.9364$ & $-2.3041$ & $1.8267$ & $0.1592$ & $0.0011$ \\
E & $300$ & $n^{-1}$ & $0.9483$ & $0.9362$ & $-2.3015$ & $1.8264$ & $0.6598$ & $0.0011$ \\
E & $300$ & $n^{-2}$ & $0.9483$ & $0.9360$ & $-2.2981$ & $1.8259$ & $11.4265$ & $0.0011$ \\
E & $1000$ & $1$ & $0.9487$ & $0.9448$ & $-2.1696$ & $1.8376$ & $0.0378$ & $0.0010$ \\
E & $1000$ & $n^{-1/2}$ & $0.9482$ & $0.9435$ & $-2.1676$ & $1.8554$ & $0.1205$ & $0.0010$ \\
E & $1000$ & $n^{-1}$ & $0.9481$ & $0.9438$ & $-2.1690$ & $1.8541$ & $0.6769$ & $0.0010$ \\
E & $1000$ & $n^{-2}$ & $0.9482$ & $0.9438$ & $-2.1698$ & $1.8541$ & $21.4035$ & $0.0010$ \\
D & $100$ & $1$ & $0.8650$ & $0.9288$ & $-2.3804$ & $1.8849$ & $0.1469$ & $0.0012$ \\
D & $100$ & $n^{-1/2}$ & $0.7392$ & $0.9275$ & $-2.1840$ & $2.1215$ & $0.3433$ & $0.0012$ \\
D & $100$ & $n^{-1}$ & $0.7338$ & $0.9271$ & $-2.1570$ & $2.1427$ & $1.0809$ & $0.0012$ \\
D & $100$ & $n^{-2}$ & $0.7339$ & $0.9271$ & $-2.1552$ & $2.1446$ & $10.8100$ & $0.0012$ \\
D & $300$ & $1$ & $0.8711$ & $0.9422$ & $-2.1485$ & $1.8895$ & $0.0880$ & $0.0010$ \\
D & $300$ & $n^{-1/2}$ & $0.7399$ & $0.9415$ & $-2.0465$ & $2.0352$ & $0.2771$ & $0.0010$ \\
D & $300$ & $n^{-1}$ & $0.7396$ & $0.9408$ & $-2.0430$ & $2.0401$ & $1.1518$ & $0.0011$ \\
D & $300$ & $n^{-2}$ & $0.7396$ & $0.9408$ & $-2.0435$ & $2.0411$ & $19.9526$ & $0.0011$ \\
D & $1000$ & $1$ & $0.8715$ & $0.9483$ & $-2.0257$ & $1.9147$ & $0.0490$ & $0.0010$ \\
D & $1000$ & $n^{-1/2}$ & $0.7412$ & $0.9473$ & $-1.9765$ & $1.9857$ & $0.2102$ & $0.0010$ \\
D & $1000$ & $n^{-1}$ & $0.7411$ & $0.9473$ & $-1.9764$ & $1.9900$ & $1.1817$ & $0.0010$ \\
D & $1000$ & $n^{-2}$ & $0.7411$ & $0.9473$ & $-1.9763$ & $1.9899$ & $37.3675$ & $0.0010$ \\
\bottomrule
\end{tabular}
\end{table}

\begin{table}[tbp]
\centering
\caption{
Mediation-denominator coverage under the three sampling laws.
Gaussian reference uses the Gaussian chi-squared product law.
Gaussian delta uses $\kappa_D=10$, and Empirical delta uses $\hat{v}_{D_{\mathrm{med}}}$ from Section~\ref{supp:illustration-mediation}.
The reference law is exact under (G) and (D), while under (E) its limiting coverage is \eqref{eq:E-benchmark}.
The denominator ratio is identical across $v$ for each paired standardized sample.
The last column is the MCSE of Empirical delta coverage.
}\label{tab:denominator}
\begin{tabular}{@{}lrrrrr@{}}
\toprule
Law & $n$ & \shortstack{Gaussian\\reference} & \shortstack{Gaussian\\delta} & \shortstack{Empirical\\delta} & \shortstack{Empirical\\MCSE} \\
\midrule
G & $100$ & $0.9509$ & $0.9479$ & $0.9374$ & $0.0011$ \\
G & $300$ & $0.9494$ & $0.9492$ & $0.9459$ & $0.0010$ \\
G & $1000$ & $0.9497$ & $0.9495$ & $0.9480$ & $0.0010$ \\
E & $100$ & $0.6870$ & $0.6759$ & $0.8721$ & $0.0015$ \\
E & $300$ & $0.6802$ & $0.6759$ & $0.9169$ & $0.0012$ \\
E & $1000$ & $0.6737$ & $0.6720$ & $0.9374$ & $0.0011$ \\
D & $100$ & $0.9502$ & $0.9459$ & $0.9367$ & $0.0011$ \\
D & $300$ & $0.9500$ & $0.9490$ & $0.9451$ & $0.0010$ \\
D & $1000$ & $0.9502$ & $0.9496$ & $0.9482$ & $0.0010$ \\
\bottomrule
\end{tabular}
\end{table}

\begin{table}[tbp]
\centering
\caption{
Precision of the mediated-effect estimator under law (G).
Error columns are empirical quantiles of $\hat{\tau}-ab$, and Median SE is the median of $\widehat{\operatorname{se}}_{\mathrm{emp}}$.
Inversion is the estimated coverage of the Gaussian-reference moment-inversion interval in Corollary~\ref{cor:flag-fieller-enlargement}.
}\label{tab:precision-errors}
\begin{tabular}{@{}lrrrrrr@{}}
\toprule
$n$ & $v$ & Error $0.025$ & Error $0.5$ & Error $0.975$ & \shortstack{Median\\SE} & Inversion \\
\midrule
$100$ & $1$ & $-0.2313$ & $-0.0033$ & $0.2531$ & $0.1197$ & $0.9832$ \\
$100$ & $n^{-1/2}$ & $-0.4385$ & $-0.0014$ & $0.4434$ & $0.2163$ & $0.9838$ \\
$100$ & $n^{-1}$ & $-1.3809$ & $0.0007$ & $1.3825$ & $0.6768$ & $0.9832$ \\
$100$ & $n^{-2}$ & $-13.8170$ & $0.0135$ & $13.8090$ & $6.7677$ & $0.9831$ \\
$300$ & $1$ & $-0.1334$ & $-0.0011$ & $0.1432$ & $0.0698$ & $0.9630$ \\
$300$ & $n^{-1/2}$ & $-0.3277$ & $0.0012$ & $0.3347$ & $0.1656$ & $0.9621$ \\
$300$ & $n^{-1}$ & $-1.3587$ & $0.0055$ & $1.3771$ & $0.6869$ & $0.9622$ \\
$300$ & $n^{-2}$ & $-23.5258$ & $0.1040$ & $23.8349$ & $11.8952$ & $0.9624$ \\
$1000$ & $1$ & $-0.0742$ & $-0.0004$ & $0.0766$ & $0.0384$ & $0.9541$ \\
$1000$ & $n^{-1/2}$ & $-0.2423$ & $0.0003$ & $0.2402$ & $0.1230$ & $0.9549$ \\
$1000$ & $n^{-1}$ & $-1.3630$ & $0.0012$ & $1.3517$ & $0.6911$ & $0.9551$ \\
$1000$ & $n^{-2}$ & $-43.0998$ & $0.0459$ & $42.7816$ & $21.8537$ & $0.9551$ \\
\bottomrule
\end{tabular}
\end{table}

\begin{table}[tbp]
\centering
\caption{
Wald coverage using empirical standard errors under law (G) at $n=300$.
Rows and columns are selected values of $h_a$ and $h_b$ from the $24\times24$ grid in Section~\ref{supp:illustration-canonical}.
Each entry uses the same $50000$ standardized samples. These are finite-sample coverage estimates.
}\label{tab:canonical}
\begin{tabular}{@{}lrrrrrrr@{}}
\toprule
$h_a\backslash h_b$ & $0$ & $1$ & $2$ & $4$ & $8$ & $16$ & $32$ \\
\midrule
$0$ & $0.9998$ & $0.9991$ & $0.9947$ & $0.9748$ & $0.9538$ & $0.9485$ & $0.9472$ \\
$1$ & $0.9992$ & $0.9488$ & $0.9249$ & $0.9529$ & $0.9520$ & $0.9476$ & $0.9463$ \\
$2$ & $0.9955$ & $0.9263$ & $0.9059$ & $0.9342$ & $0.9491$ & $0.9479$ & $0.9465$ \\
$4$ & $0.9739$ & $0.9521$ & $0.9337$ & $0.9312$ & $0.9433$ & $0.9471$ & $0.9468$ \\
$8$ & $0.9528$ & $0.9516$ & $0.9485$ & $0.9429$ & $0.9421$ & $0.9465$ & $0.9467$ \\
$16$ & $0.9472$ & $0.9477$ & $0.9474$ & $0.9465$ & $0.9458$ & $0.9460$ & $0.9479$ \\
$32$ & $0.9458$ & $0.9464$ & $0.9464$ & $0.9462$ & $0.9464$ & $0.9464$ & $0.9470$ \\
\bottomrule
\end{tabular}
\end{table}

\begin{table}[tbp]
\centering
\caption{
Local-limit coverage $\Pr\{|W_\infty(h_a,h_b)|\le z_{0.975}\}$ at the points in Table~\ref{tab:canonical}.
Entries are estimates from $10^6$ independent standard normal pairs, shared across points.
The limiting statistic is defined in \eqref{eq:local-Wald}.
}\label{tab:canonical-limit}
\begin{tabular}{@{}lrrrrrrr@{}}
\toprule
$h_a\backslash h_b$ & $0$ & $1$ & $2$ & $4$ & $8$ & $16$ & $32$ \\
\midrule
$0$ & $0.9999$ & $0.9994$ & $0.9963$ & $0.9774$ & $0.9567$ & $0.9513$ & $0.9499$ \\
$1$ & $0.9994$ & $0.9496$ & $0.9281$ & $0.9568$ & $0.9555$ & $0.9513$ & $0.9500$ \\
$2$ & $0.9963$ & $0.9284$ & $0.9091$ & $0.9378$ & $0.9524$ & $0.9511$ & $0.9500$ \\
$4$ & $0.9776$ & $0.9565$ & $0.9376$ & $0.9341$ & $0.9467$ & $0.9504$ & $0.9501$ \\
$8$ & $0.9570$ & $0.9556$ & $0.9524$ & $0.9466$ & $0.9452$ & $0.9490$ & $0.9499$ \\
$16$ & $0.9516$ & $0.9515$ & $0.9515$ & $0.9506$ & $0.9490$ & $0.9485$ & $0.9497$ \\
$32$ & $0.9503$ & $0.9502$ & $0.9503$ & $0.9503$ & $0.9501$ & $0.9497$ & $0.9494$ \\
\bottomrule
\end{tabular}
\end{table}

\begin{table}[tbp]
\centering
\caption{
Calibration of degree-ten covariate-adjusted products with $q=4$.
Nested is $\Delta_C\Delta_{C\cup\{X_1,X_2\}}$ and Nonnested is $\Delta_{C\cup\{X_1\}}\Delta_{C\cup\{X_2\}}$.
Both coverage columns use their own reference laws at $\rho=0$.
Plug-in uses the nonnested product's covariance-adapted log-delta interval.
The last column is the MCSE of its transported-reference coverage.
}\label{tab:adjusted}
\begin{tabular}{@{}lrrrrrr@{}}
\toprule
$n$ & $\rho$ & $v_2$ & Nested & Nonnested & \shortstack{Nonnested\\plug-in} & \shortstack{Nonnested\\MCSE} \\
\midrule
$50$ & $0$ & $1$ & $0.9520$ & $0.9520$ & $0.8809$ & $0.0010$ \\
$50$ & $0$ & $10^{-6}$ & $0.9520$ & $0.9520$ & $0.8809$ & $0.0010$ \\
$50$ & $0.5$ & $1$ & $0.9520$ & $0.9484$ & $0.8803$ & $0.0010$ \\
$50$ & $0.5$ & $10^{-6}$ & $0.9520$ & $0.9484$ & $0.8803$ & $0.0010$ \\
$50$ & $0.9$ & $1$ & $0.9520$ & $0.9422$ & $0.8832$ & $0.0010$ \\
$50$ & $0.9$ & $10^{-6}$ & $0.9520$ & $0.9422$ & $0.8832$ & $0.0010$ \\
$50$ & $0.99$ & $1$ & $0.9520$ & $0.9390$ & $0.8837$ & $0.0011$ \\
$50$ & $0.99$ & $10^{-6}$ & $0.9520$ & $0.9390$ & $0.8837$ & $0.0011$ \\
$200$ & $0$ & $1$ & $0.9491$ & $0.9491$ & $0.9322$ & $0.0010$ \\
$200$ & $0$ & $10^{-6}$ & $0.9491$ & $0.9491$ & $0.9322$ & $0.0010$ \\
$200$ & $0.5$ & $1$ & $0.9491$ & $0.9453$ & $0.9321$ & $0.0010$ \\
$200$ & $0.5$ & $10^{-6}$ & $0.9491$ & $0.9453$ & $0.9321$ & $0.0010$ \\
$200$ & $0.9$ & $1$ & $0.9491$ & $0.9383$ & $0.9337$ & $0.0011$ \\
$200$ & $0.9$ & $10^{-6}$ & $0.9491$ & $0.9383$ & $0.9337$ & $0.0011$ \\
$200$ & $0.99$ & $1$ & $0.9491$ & $0.9358$ & $0.9336$ & $0.0011$ \\
$200$ & $0.99$ & $10^{-6}$ & $0.9491$ & $0.9358$ & $0.9336$ & $0.0011$ \\
$1000$ & $0$ & $1$ & $0.9489$ & $0.9489$ & $0.9461$ & $0.0010$ \\
$1000$ & $0$ & $10^{-6}$ & $0.9489$ & $0.9489$ & $0.9461$ & $0.0010$ \\
$1000$ & $0.5$ & $1$ & $0.9489$ & $0.9469$ & $0.9466$ & $0.0010$ \\
$1000$ & $0.5$ & $10^{-6}$ & $0.9489$ & $0.9469$ & $0.9466$ & $0.0010$ \\
$1000$ & $0.9$ & $1$ & $0.9489$ & $0.9389$ & $0.9475$ & $0.0011$ \\
$1000$ & $0.9$ & $10^{-6}$ & $0.9489$ & $0.9389$ & $0.9475$ & $0.0011$ \\
$1000$ & $0.99$ & $1$ & $0.9489$ & $0.9368$ & $0.9475$ & $0.0011$ \\
$1000$ & $0.99$ & $10^{-6}$ & $0.9489$ & $0.9368$ & $0.9475$ & $0.0011$ \\
\bottomrule
\end{tabular}
\end{table}

\begin{table}[tbp]
\centering
\caption{
Interval coverage for $\theta(\Sigma)=1/(\1_6^\top\Sigma^{-1}\1_6)$ in the six-dimensional block design.
Exact uses $\hat{\theta}/\theta\stackrel{d}{=}\chi^2_{n-5}/n$.
Delta uses $\kappa_\theta=2$, and Log moments uses the exact logarithmic mean and variance in a normal approximation.
MCSE refers to Exact coverage.
}\label{tab:precision}
\begin{tabular}{@{}lrrrrr@{}}
\toprule
$n$ & $\rho$ & Exact & Delta & Log moments & Exact MCSE \\
\midrule
$50$ & $0$ & $0.9484$ & $0.8837$ & $0.9491$ & $0.0010$ \\
$50$ & $0.5$ & $0.9490$ & $0.8840$ & $0.9500$ & $0.0010$ \\
$50$ & $0.9$ & $0.9500$ & $0.8840$ & $0.9507$ & $0.0010$ \\
$50$ & $0.99$ & $0.9504$ & $0.8864$ & $0.9507$ & $0.0010$ \\
$200$ & $0$ & $0.9525$ & $0.9374$ & $0.9529$ & $0.0010$ \\
$200$ & $0.5$ & $0.9516$ & $0.9382$ & $0.9516$ & $0.0010$ \\
$200$ & $0.9$ & $0.9511$ & $0.9381$ & $0.9511$ & $0.0010$ \\
$200$ & $0.99$ & $0.9509$ & $0.9372$ & $0.9507$ & $0.0010$ \\
$1000$ & $0$ & $0.9504$ & $0.9476$ & $0.9502$ & $0.0010$ \\
$1000$ & $0.5$ & $0.9498$ & $0.9478$ & $0.9501$ & $0.0010$ \\
$1000$ & $0.9$ & $0.9509$ & $0.9483$ & $0.9511$ & $0.0010$ \\
$1000$ & $0.99$ & $0.9507$ & $0.9467$ & $0.9510$ & $0.0010$ \\
\bottomrule
\end{tabular}
\end{table}

\begin{table}[tbp]
\centering
\caption{
Excursion estimates for the positive degree-five denominators along $t_n=1/n$.
Flag-power total estimates $\Pr(|R_{n,0}-1|>1/2)$.
Napkin lower and upper estimate $\Pr(R_{n,\mathrm{nap}}<1/2)$ and $\Pr(R_{n,\mathrm{nap}}>3/2)$; Total and MCSE refer to their union.
The zero in the flag-power column corresponds to no observed excursion among $50000$ replications.
}\label{tab:napkin}
\begin{tabular}{@{}lrrrrr@{}}
\toprule
$n$ & \shortstack{Flag-power\\total} & \shortstack{Napkin\\lower} & \shortstack{Napkin\\upper} & \shortstack{Napkin\\total} & \shortstack{Total\\MCSE} \\
\midrule
$50$ & $0.3700$ & $0.2440$ & $0.3250$ & $0.5691$ & $0.0022$ \\
$200$ & $0.0840$ & $0.1404$ & $0.3407$ & $0.4811$ & $0.0022$ \\
$1000$ & $0.0008$ & $0.0804$ & $0.3450$ & $0.4254$ & $0.0022$ \\
$5000$ & $0.0000$ & $0.0503$ & $0.3480$ & $0.3982$ & $0.0022$ \\
\bottomrule
\end{tabular}
\end{table}

\begin{table}[tbp]
\centering
\caption{
A first-order weak covariance along $\rho_n=h/\sqrt{n}$.
Excursion estimates $\Pr(|D(\Sighat)/D(\Sigma_n)-1|>1/2)$, and MCSE gives its Monte Carlo standard error.
Limit is the theoretical value $2\Phi(-h/2)$.
Ratio columns are empirical quantiles of $D(\Sighat)/D(\Sigma_n)$.
}\label{tab:weak}
\begin{tabular}{@{}lrrrrrr@{}}
\toprule
$n$ & $h$ & Excursion & MCSE & Limit & Ratio $0.025$ & Ratio $0.975$ \\
\midrule
$50$ & $1$ & $0.6135$ & $0.0022$ & $0.6171$ & $-0.9058$ & $3.0455$ \\
$50$ & $2$ & $0.3284$ & $0.0021$ & $0.3173$ & $0.0461$ & $2.0728$ \\
$200$ & $1$ & $0.6177$ & $0.0022$ & $0.6171$ & $-0.9470$ & $3.0057$ \\
$200$ & $2$ & $0.3223$ & $0.0021$ & $0.3173$ & $0.0266$ & $2.0176$ \\
$1000$ & $1$ & $0.6158$ & $0.0022$ & $0.6171$ & $-0.9609$ & $2.9610$ \\
$1000$ & $2$ & $0.3172$ & $0.0021$ & $0.3173$ & $0.0184$ & $1.9818$ \\
$5000$ & $1$ & $0.6188$ & $0.0022$ & $0.6171$ & $-0.9619$ & $2.9684$ \\
$5000$ & $2$ & $0.3205$ & $0.0021$ & $0.3173$ & $0.0188$ & $1.9869$ \\
\bottomrule
\end{tabular}
\end{table}

\begin{figure}[tbp]
\centering
\includegraphics[width=.78\linewidth]{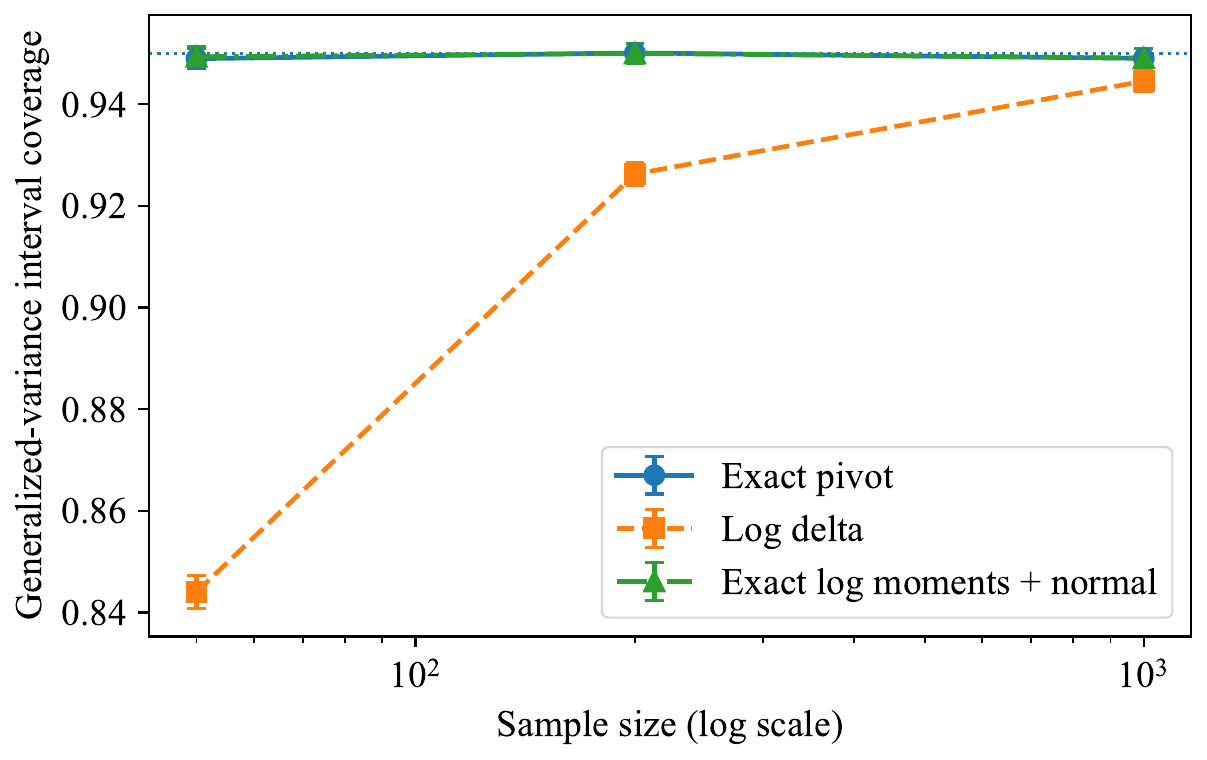}
\caption{
Generalized-variance interval coverage for $r=3$ (dimension six).
The three curves correspond to the Exact, Delta and Log moments columns of Table~\ref{tab:corrections}.
The dotted line marks nominal coverage $0.95$.
Error bars show $\pm1.96$ MCSEs from $50000$ replications.
}
\label{fig:log-corrections}
\end{figure}

\begin{figure}[tbp]
\centering
\includegraphics[width=.78\linewidth]{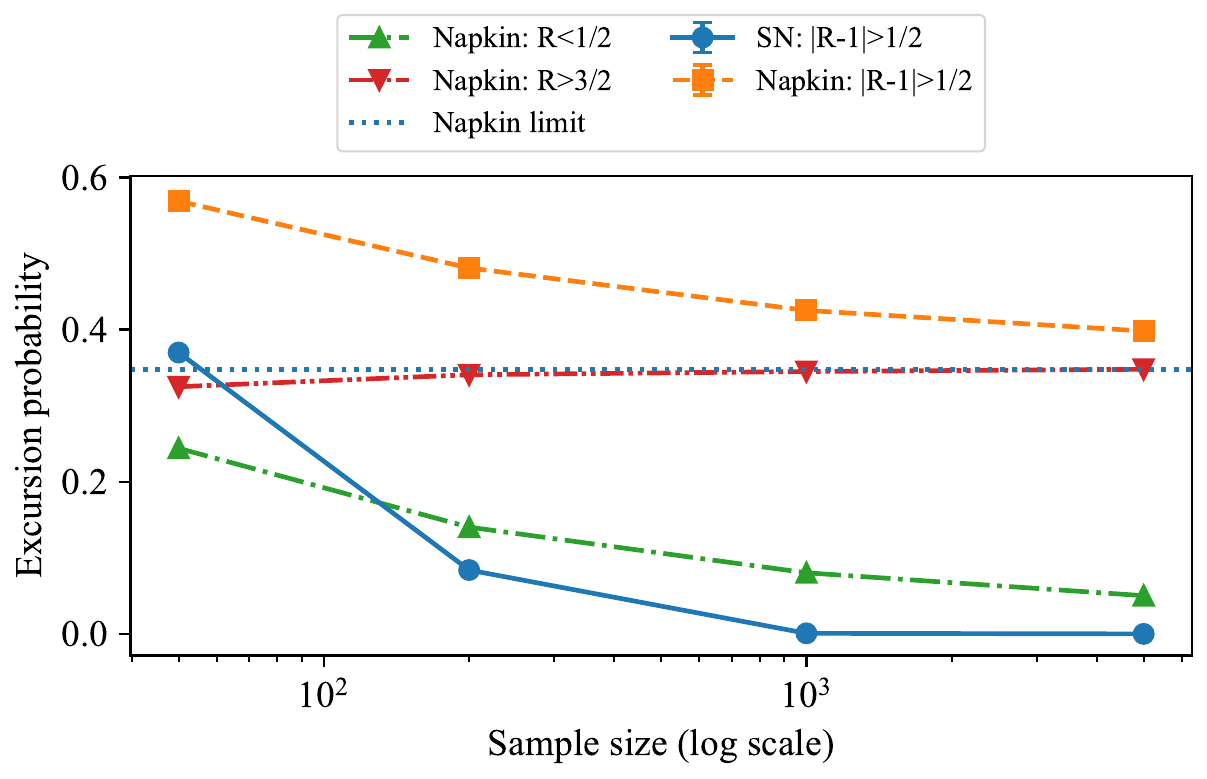}
\caption{
Excursion estimates for the positive degree-five denominators $D_0$ and $D_{\mathrm{nap}}$ along $t_n=1/n$.
The legend's $R$ denotes the sample-to-population ratio for the indicated denominator.
The lower and upper Napkin curves give excursions below $1/2$ and above $3/2$.
The dotted line is the common limit $0.3472435$ of the upper and total Napkin excursion probabilities.
The lower Napkin and total flag-power probabilities tend to zero.
}
\label{fig:napkin}
\end{figure}

\begin{figure}[tbp]
\centering
\includegraphics[width=.66\linewidth]{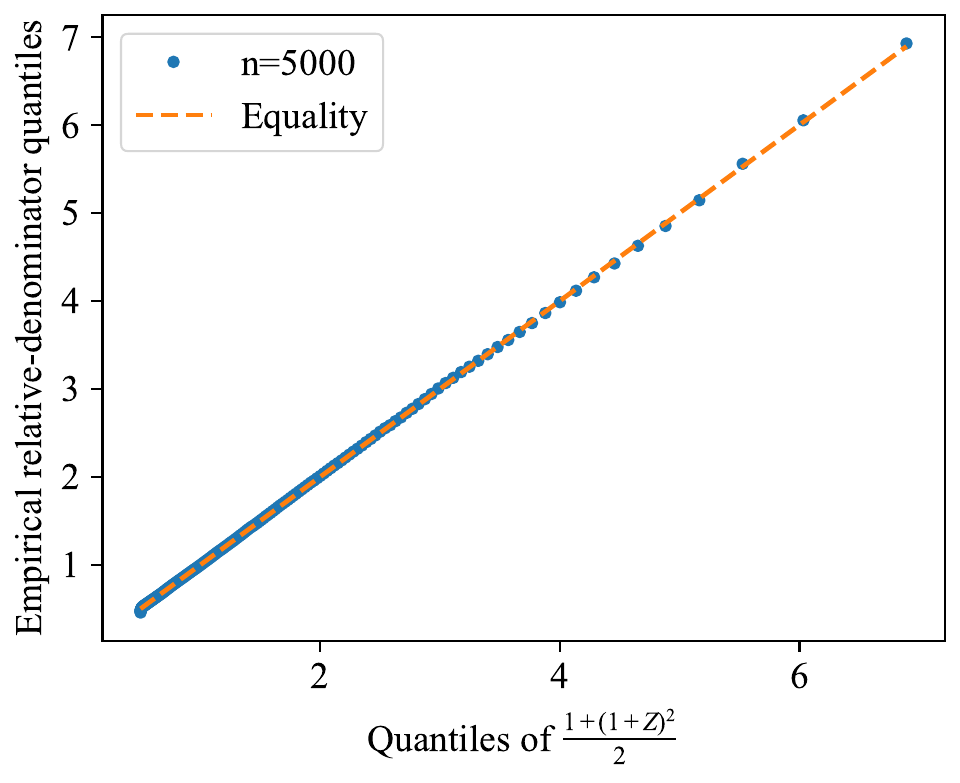}
\caption{
Quantile comparison for the Napkin sample-to-population ratio at $n=5000$.
The horizontal axis gives quantiles of $\{1+(1+Z)^2\}/2$ with $Z\sim N_1(0,1)$, the limiting law in Section~\ref{supp:numerical-napkin}.
The vertical axis gives empirical quantiles from $50000$ replications.
The plot uses the 199 probabilities $0.005,0.010,\ldots,0.995$ and the dashed equality line.
The lower-end discrepancy corresponds to the finite-sample excursions below $1/2$ in Figure~\ref{fig:napkin}.
}
\label{fig:napkin-qq}
\end{figure}

\end{document}